\documentclass{amsart}

\usepackage{amsmath}
\usepackage{amssymb}
\usepackage{amsthm}
\usepackage{mathrsfs}
\usepackage{mathtools}
\usepackage{tikz}
\usepackage{cleveref}
\usepackage{tikz-cd}
\usepackage{mathtools}

\usepackage{placeins}

\usepackage{float}

\usepackage{comment}

\theoremstyle{definition}
\newtheorem{thm}{Theorem}[section]
\crefname{thm}{Theorem}{Theorems}
\newtheorem{cor}[thm]{Corollary}
\newtheorem{prop}[thm]{Proposition}
\crefname{prop}{Proposition}{Propositions}
\newtheorem{lem}[thm]{Lemma}
\crefname{lem}{Lemma}{Lemmas}

\newtheorem{defn}[thm]{Definition}

\crefname{defn}{Definition}{Definitions}

\newtheorem{exmp}[thm]{Example}

\newtheorem{rmk}[thm]{Remark}

\newtheorem*{ack*}{Acknowledgements}
\newtheorem*{clm*}{Claim}

\newcommand*{\mb}[1]{\mathbb{#1}}
\newcommand{\pt}{{\rm pt}}
\newcommand{\ms}{\mathscr}
\newcommand{\mc}{\mathcal}
\newcommand{\wt}{\widetilde}
\newcommand*{\on}[1]{\operatorname{#1}}
\newcommand{\Gm}{\mb{G}_m}

\begin{document}

\author{Hunter Spink, Dennis Tseng}
%\address{California, United States}
%\email{hspink@math.harvard.edu}
%\thanks{The author would like to thank Harvard University for providing travel funding.}

\title{$PGL_2$-equivariant strata of point configurations in $\mb{P}^1$}

\begin{abstract}
We compute the integral Chow ring of the quotient stack $[(\mb{P}^1)^n/PGL_2]$, which contains $\mc{M}_{0,n}$ as a dense open, and determine a natural $\mb{Z}$-basis for the Chow ring in terms of certain ordered incidence strata. We further show that all $\mb{Z}$-linear relations between the classes of ordered incidence strata arise from an analogue of the WDVV relations in $A^\bullet(\overline{\mc{M}}_{0,n})$.

Next we compute the classes of unordered incidence strata in the integral Chow ring of the quotient stack $[{\rm Sym}^n\mb{P}^1/PGL_2]$ and classify all $\mb{Z}$-linear relations between the strata via these analogues of WDVV relations. %classify all $\mb{Z}$-linear relations between the strata.

% and show that every $\mb{Z}$-linear relation between these strata in the rational Chow ring induces a relation in the integral Chow ring, despite the presence of $2$-torsion. We also show that the classes of certain unordered incidence strata give a $\mb{Q}$-basis for the rational Chow ring of the quotient stack $[{\rm Sym}^n\mb{P}^1/PGL_2]$ up to degree $n-2$, and show how to write every unordered strata in this basis. 
%that the classes of certain unordered incidence strata give a $\mb{Q}$-basis for the rational Chow ring of the quotient stack $[{\rm Sym}^n\mb{P}^1/PGL_2]$ up to degree $n-2$.

%Furthermore, we show that a $\mb{Z}$-linear combination of these classes is zero if and only if it is zero in the rational Chow ring of $[{\rm Sym}^n\mb{P}^1/PGL_2]$ despite the presence of $2$-torsion, reducing the problem of determining relations between these classes to finding relations between certain rational polynomials.

Finally, we compute the rational Chow rings of the complement of a union of unordered incidence strata.
%Linear relations between classes of strata in the ordered and unordered cases above imply universal relations between counts in enumerative problems in a relative setting.
%We determine the linear relations between equivariant Chow classes of strata in $(\mb{P}^1)^n$ and ${\rm Sym}^n \mb{P}^1$ using a combination of geometric and combinatorial techniques. These relations imply certain universal relations between counts in enumerative geometry involving point coincidences in $\mb{P}^1$ in a relative setting. In particular, we determine 
%We also classify the relations which can occur between the $\binom{n}{m}$ pullbacks of a class $\alpha \in H^\bullet(X^m;\mb{Q})^{S_m}$ to $H^\bullet(X^n;\mb{Q})$ for any topological space $X$ as a consequence of certain ``hyperoctahedral relations'' between elementary symmetric polynomials.
\end{abstract}

\maketitle
\section{Introduction}
\label{intro}
We consider $PGL_2$-equivariant Chow classes of incidence strata corresponding to point configurations in $\mb{P}^1$. %These universal relations descend to relations between counts in enumerative problems in a relative setting, and are related to intersection theory on the moduli space of $n$-pointed curves $\mc{M}_{0,n}$. 
Our results concern both ordered point configurations, parametrized by $(\mb{P}^1)^n$, and unordered point configurations, parametrized by ${\rm Sym}^n\mb{P}^{1}\cong \mb{P}^n$, previously considered $GL_2$-equivariantly by Feh\'er, N\'emethi, and Rim\'anyi \cite{FNR06}. The equivariant Chow rings $A^\bullet_{PGL_2}((\mb{P}^1)^n)$ and $A^\bullet_{PGL_2}(\mb{P}^n)$ can be defined as the integral Chow rings of the quotient stacks $[(\mb{P}^1)^n/PGL_2]$ and $[\mb{P}^n/PGL_2]$ respectively, so the $PGL_2$-equivariant Chow classes of incidence strata are of interest because they specialize to relative classes in $A^\bullet(\mc{P}^n)$ and $A^\bullet(\text{\rm Sym}^n\mc{P})$ respectively for any $\mb{P}^1$-bundle $\mc{P} \to B$.

The equivariant Chow ring $A^{\bullet}_{PGL_2}(\pt)\cong \mb{Z}[c_2,c_3]/(2c_3)$ was computed by Pandharipande \cite{Pandharipande}, and the 2-torsion is reflected in the fact that $PGL_2$ is not special. Consequently, restriction to a maximal torus is only injective rationally \cite[Proposition 6]{EG98}, which is the main obstacle to computing integral Chow classes. This is in contrast with the analogous $GL_2$-equivariant Chow rings, which are substantially easier to work with as $GL_2$ is special. Also, classes in the $GL_2$-equivariant Chow rings of $(\mb{P}^1)^n$ and $\mb{P}^n$ only specialize to classes in $A^{\bullet}(\mc{P}^n)$ and $A^{\bullet}({\rm Sym}^n \mc{P})$ when $\mc{P}\to B$ is the projectivization of a rank $2$ vector bundle. 

The reader may refer to \Cref{universalrelations} and \Cref{eqinttheory} for an exposition on how equivariant Chow classes yield universal relations between relative Chow classes in bundles, and \Cref{exampleintro} for example applications.

\subsection{Ordered strata in $[(\mb{P}^1)^n/PGL_2]$}
\label{Orderedintro}

The moduli space $\mc{M}_{0,n}$ of $n$ distinct points on $\mb{P}^1$ is the quotient of $(\mb{P}^1)^n\setminus\bigcup_{i<j}\Delta_{i,j}$ by the free action of $PGL_2$, where $\Delta_{i,j}$ is the locus in $(\mb{P}^1)^n$ where the $i$th and $j$th coordinates are equal. This is classically compactified by the variety $\overline{\mc{M}}_{0,n}$ of stable genus zero $n$-pointed curves.

We study the quotient stack $[(\mb{P}^1)^n/PGL_2]$ containing $\mc{M}_{0,n}$ as a dense open and in particular its integral Chow ring $A^\bullet_{PGL_{2}}((\mb{P}^1)^n)$ as defined in \cite[Section 5]{EG98}. This stack is stratified by certain incidence strata $\Delta_P\subset (\mb{P}^1)^n$ for $P$ a partition of $[n]:=\{1,\ldots,n\}$, the loci where the $i$th and $j$th coordinates are equal if $i$ and $j$ are in the same part of $P$.

We compute a ring presentation in \Cref{bigthmintro} for $A^\bullet_{PGL_{2}}((\mb{P}^1)^n)$ similar to that of $A^{\bullet}(\overline{\mc{M}}_{0,n})$ computed by Keel \cite{Keel}. The incidence strata $\Delta_P$ play a fundamental role in the equivariant Chow ring: in \Cref{goodbasis} we compute a $\mb{Z}$-basis for $A^\bullet_{PGL_2}((\mb{P}^1)^n)$, which consists in degree $\leq n-2$ of certain incidence strata, and in \Cref{basisthmintro} we show all relations between incidence strata arise from an analogue of the WDVV relation in $A^{\bullet}(\overline{\mc{M}}_{0,4})$ (see \Cref{WDVVintro}).

%Using a combination of geometric and combinatorial techniques, w

%, and construct an additive basis via these $\Delta_P$-classes, which up to degree $n-2$ is given by an explicit subset of the $\Delta_P$ classes.

%In particular, we show the integral Chow ring   In particular, the integral Chow ring can be presented in terms of diagonal classes, which we now define. 

\begin{comment}
\begin{defn}
Let $\psi_i \in A^{\bullet}_{PGL_2}((\mb{P}^1)^n)$ be the pullback of $c_1(T_{\mb{P}^1}^{\vee})$ under the $i$th projection $(\mb{P}^1)^n\to \mb{P}^1$. 
\end{defn}
\end{comment}
\begin{comment}
\begin{defn}
For distinct $i,j \in \{1,\ldots,n\}$, we denote by $\Delta_{i,j}$ (where the two indices are interchangeable) both for the locus in $(\mb{P}^1)^n$ where the $i$th and $j$th coordinates coincide, and the class of this locus in $ A^\bullet_{PGL_2}((\mb{P}^1)^n)$.
\end{defn}
\end{comment}

\begin{thm}
\label{bigthmintro}
The following are true.
\begin{enumerate}
\item (\Cref{mainthm})
For $n \ge 3$, the ring $A^\bullet_{PGL_2}((\mb{P}^1)^n)=\frac{\mb{Z}[\Delta_{i,j}]_{1 \le i < j \le n}}{\text{relations}}$, where the relations are (notating $\Delta_{j,i}:=\Delta_{i,j}$ for $j>i$)
\begin{enumerate}
\item $\Delta_{i,j}+\Delta_{k,l}=\Delta_{i,k}+\Delta_{j,l}$ for distinct $i,j,k,l$ \hfill (square relations)
\item $\Delta_{i,j}\Delta_{i,k}=\Delta_{i,j}\Delta_{j,k}$ for distinct $i,j,k$. \hfill (diagonal relations)
%\item $\psi_i=\Delta_{i,j}+\Delta_{i,k}-\Delta_{j,k}$ for any distinct indices $i,j,k$.
%\item $\psi_i+\psi_j=2\Delta_{i,j}$
\end{enumerate}
\item
(\Cref{bettinumber}) For $n \ge 1$, the group $A^k_{PGL_2}((\mb{P}^1)^n)$ is a free $\mb{Z}$-module of rank
$$\sum_{\substack{i \le k\\i \equiv k\text{ mod 2}}} \binom{n}{i}.$$
\item
(\Cref{PGL2injective}) For $n\ge 1$, the natural map from $A^\bullet_{PGL_2}((\mb{P}^1)^n)$ to $$A^\bullet_{GL_2}((\mb{P}^1)^n)\cong \mb{Z}[u,v]^{S_2}[H_1,\ldots,H_n]/(F(H_1),\ldots,F(H_n)),$$ is injective, where $u+v$ and $uv$ are the first and second chern classes of the standard representation of $GL_2$, $F(z)=(z+u)(z+v)$, and $H_i$ is $c_1(\ms{O}(1))\in A^\bullet_{PGL_2}((\mb{P}^1))$ pulled back via projection to the $i$th factor.

This identifies $A^\bullet_{PGL_2}((\mb{P}^1)^n)$ with the subring of $A^\bullet_{GL_2}((\mb{P}^1)^n)$ generated by $H_i+H_j+u+v$ for distinct $i,j$ and $2H_i+u+v$ for all $i$, and this maps $$\Delta_{i,j} \mapsto H_i+H_j+u+v.$$
%$$\psi_i \mapsto 2H_i +u+v$$
\item
(\Cref{equivariantcohomology}) If the base field is $\mb{C}$, then for all $n\ge 1$ the map $A^\bullet_{PGL_2}((\mb{P}^1)^n)\to H^{\bullet}_{PGL_2}((\mb{P}^1)^n)$ to equivariant cohomology is an isomorphism. 
\end{enumerate}
\end{thm}

The square relations  $\Delta_{i,j}+\Delta_{k,l}=\Delta_{i,k}+\Delta_{j,l}$ for distinct $i,j,k,l$ are analogous to the WDVV relations on $A^{\bullet}(\overline{M}_{0,n})$ pulled back from $A^{\bullet}(\overline{M}_{0,4})\cong A^{\bullet}(\mb{P}^1)$ (see \Cref{WDVVintro}). 

The diagonal relations $\Delta_{i,j}\Delta_{i,k}=\Delta_{i,j}\Delta_{j,k}$ are geometrically obvious as $\Delta_{i,j} \cap \Delta_{i,k}$ and $\Delta_{i,j}\cap \Delta_{j,k}$ both give the locus where the $i$th, $j$th, and $k$th coordinates are all equal. In particular, repeated intersections in this fashion allow us to reconstruct all $\Delta_P$. 
%\begin{rmk}
%For $n\leq 2$ we will use also consider the classes $\psi_i \in A^{\bullet}_{PGL_2}((\mb{P}^1)^n)$, the pullback of $c_1(T_{\mb{P}^1}^{\vee})$ under the $i$th projection $(\mb{P}^1)^n\to \mb{P}^1$, to give a presentation which works uniformly for all $n$.
%\end{rmk}

We will in fact show that the classes of the $\Delta_P$ for $P$ a partition of $\{1,\ldots,n\}$ into $d \ge 2$ parts generate $A^{n-d}_{PGL_2}((\mb{P}^1)^n)$ $\mb{Z}$-linearly. Surprisingly, we can produce a $\mb{Z}$-basis for $A^\bullet_{PGL_2}((\mb{P}^1)^n)$ represented by certain $\Delta_P$ (at least in degrees $\le n-2$).
\begin{defn}
Call a partition $P$ of $\{1,\ldots,n\}$ \emph{good} if it can be written as $P=\{A_1,\ldots,A_d\}$ with $A_1 \sqcup A_2$ an initial segment of $\{1,\ldots,n\}$, and $A_3,\ldots,A_d$ intervals.
\end{defn}
\begin{thm}[{\Cref{mainthm}}]
\label{goodbasis}
For $n\ge 3$, the additive group $A^{k}_{PGL_2}((\mb{P}^1)^n)$ has a $\mb{Z}$-basis consisting of the following.
\begin{enumerate}
\item
If $k \le n-2$, the classes $\Delta_P$ for $P$ a good partition into $n-k$ parts.
\item
If $k > n-2$, the classes $\Delta_{i_P,j_P}^{k-n+2} \Delta_P$ for $P$ a partition of $\{1,\ldots,n\}$ into two parts and $\Delta_{i_{\{[n]\}},j_{\{[n]\}}}^{k-n+1}\Delta_{\{[n]\}}$, where for each $P$ the pair $i_P,j_P$ are chosen to lie in the same part of $P$.
\end{enumerate}
%The classes $\Delta_P$ for $P$ a good partition into $\ge 2$ parts form a $\mb{Z}$-basis for $A^{\le n-2}_{PGL_2}((\mb{P}^1)^n)$.
%For $P$ a partition of $\{1,\ldots,n\}$ into at most two parts, choose $a_P\in \{1,\ldots,n\}$. 
%The additive group $A^{\ge n-1}_{PGL_2}((\mb{P}^1)^n)$ has a $\mb{Z}$-basis given by $\psi_{1}^k \Delta_P$ over partitions $P$ of $\{1,\ldots,n\}$ into at most two parts, and $k \ge 0$.
\end{thm}
%In particular, this implies that the $\Delta_P$ for $P$ a partition of $n$ into $d \ge 2$ parts generate $A^{n-d}_{PGL_2}((\mb{P}^1)^n)$ additively (see \Cref{Pgenerate} for both a combinatorial proof and a proof via excision). 
In \Cref{algandexmp} we describe a simple algorithm to write arbitrary classes in this $\mb{Z}$-basis, along with a worked example.

In addition, we show that all relations between the $\Delta_P$ are generated by pushforwards of square relations. The method  of proof will in fact supply an algorithm to write every $\Delta_Q$ as a $\mb{Z}$-linear combination of $\Delta_P$ for $P$ a good partition using only these relations.
\begin{defn}
Denote by $\on{Part}(d,n)$ the set of partitions of $[n]$ into $d$ parts. Let $\on{Sq}(d,n)$ be the subgroup of the free abelian group $\mb{Z}^{\on{Part}(d,n)}$ generated by formal square relations $P_{i,j}-P_{j,k}+P_{k,l}-P_{l,i}$ for $P \in \on{Part}(d+1,n)$ and $i,j,k,l\in\{1,\ldots,n\}$ indices in different parts of $P$, where $P_{x,y}$ denotes the partition formed by merging the parts of $P$ containing $x$ and $y$.
\end{defn}
\begin{thm}[{\Cref{3equal}}]
\label{basisthmintro}
For $d \ge 2$, the map
$$\mb{Z}^{\on{Part}(d,n)}/\on{Sq}(d,n) \to A^{n-d}_{PGL_2}((\mb{P}^1)^n)$$
sending $P \mapsto \Delta_P$ is an isomorphism.
\end{thm}
In particular, since every square relation between the $\Delta_P$ classes comes from an explicit $PGL_2$-invariant degeneration in $(\mb{P}^1)^n$ (see \Cref{WDVVintro}), \Cref{basisthmintro} implies that all linear relations between the $\Delta_P$ classes can be realized by a sequence of $PGL_2$-invariant degenerations in $(\mb{P}^1)^n$.

Non-equivariantly, there are relations between the classes $\Delta_P\in A^\bullet((\mb{P}^1)^n)$ not generated by these square relations. For example, if $n=4$ we have
\begin{align*}
\Delta_{\{\{1,2,3\},\{4\}\}}+\Delta_{\{\{1,2,4\},\{3\}\}}+\Delta_{\{\{1,3,4\},\{2\}\}}+\Delta_{\{\{2,3,4\},\{1\}\}}&\\=\Delta_{\{\{1,2\},\{3,4\}\}}+\Delta_{\{\{1,3\},\{2,4\}\}}+\Delta_{\{\{1,4\},\{2,3\}\}}&
\end{align*}
in $A^{2}((\mb{P}^1)^4)$.

\begin{rmk}
\label{psi-intro}
All of our theorems can be extended to $n=1,2$ if we include the classes $\psi_i=\pi_i^*c_1(T^{\vee}\mb{P}^1)\in A^\bullet_{PGL_2}((\mb{P}^1)^n)$ pulled back from the $i$th projection $\pi_i$, which for $n \ge 3$ can be written in terms of the $\Delta_{j,k}$-classes via $\psi_i=\Delta_{j,k}-\Delta_{i,j}-\Delta_{i,k}$ for any $j,k \ne i$. They correspond to $-(2H_i+u+v)$ under the map from item (3) of \Cref{bigthmintro} (see \Cref{psidelta}) and their definition is analogous to the $\psi$-classes on $\overline{M}_{0,n}$ \cite[Section 2]{P02}.
\end{rmk}

\subsubsection{Relation of the square relation to the $WDVV$ relation}\label{WDVVintro}
The $WDVV$ relation in $A^\bullet(\overline{\mc{M}}_{0,4})$ says two points in $\overline{\mc{M}}_{0,4}\cong \mb{P}^1$ corresponding to reducible curves have the same class \cite[Section 0.1]{LP11}. It was shown by Keel \cite{Keel} that $A^\bullet(\overline{\mc{M}}_{0,n})$ is generated as a ring by its boundary divisors, and the only nontrivial relations come from pulling back the $WDVV$ relation under forgetful maps $\overline{\mc{M}}_{0,n} \to \overline{\mc{M}}_{0,4}$.
%says that the classes of the points corresponding to the reducible curves $\mc{C}_{12,34}$ and $\mc{C}_{14,23}$ are equal. 
%The interest in this relation is that if we have the flat map $\overline{\mc{M}}_{0,n} \to \overline{\mc{M}}_{0,4}$ forgetting all but $4$ points and stabilizing, then we can pull back the $WDVV$ relation to a non-trivial relation in $A^\bullet(\overline{\mc{M}_{0,n}})$. 
The square relations relate to the $WDVV$ relations as follows. Consider the diagram
\begin{center}
\begin{tikzcd}
\overline{\mc{M}}_{0,4}(\mb{P}^1,1) \ar[r, "\text{ev}"] \ar[d,"\pi"] & (\mb{P}^1)^4\\ \overline{\mc{M}}_{0,4} &
\end{tikzcd}
\end{center}
where $\text{ev}$ is the ($PGL_2$-equivariant) total evaluation map from the Kontsevich mapping space \cite[Section 1]{FP} and $\pi$ remembers only the source of the stable map and stabilizes. The square relation is $\on{ev}_{*}\pi^{*}$ applied to  the $WDVV$ relation. 
%The square relation is $\text{ev}_{*}$ applied to the $WDVV$ relation equating $[\pi^{-1}(\mc{C}_{12,34})]$ and $[\pi^{-1}(\mc{C}_{14,23})]$ in $A^{\bullet}(\overline{\mc{M}}_{0,4}(\mb{P}^1,1))\cong A^{\bullet}(\mb{P}^1)$ as a consequence of the flatness of $\pi$.

Equivalently, for any closed point $a\in \mb{P}^1\cong \overline{\mc{M}}_{0,4}$, we can consider the locus $A_{a}\subset (\mb{P}^1)^n$ consisting of the quadruples of points with cross ratio $a$. The square relation comes from equating the classes of $A_0$ and $A_\infty$. %Since \Cref{basisthmintro} implies the additive relations between the classes $\Delta_P$ are given by square relations, all additive relations in degree $\le n-2$ between the $\Delta_P$ can be obtained geometrically via $PGL_2$-invariant degenerations in $(\mb{P}^1)^n$. 
%Where $\overline{\mc{M}}_{0,4}\cong \mb{P}^1$ is the standard compactification of $\mc{M}_{0,4}$ and $\overline{\mc{M}}_{0,4}(\mb{P}^1,1)$ is the Konsevich mapping space of degree $1$ stable maps from pre-stable curves with $4$ marked points to $\mb{P}^1$. Then as $\pi$ is flat, $[ev_*\pi^{-1}(\mc{C})]$ is independent of the choice of $\mc{C}$. Taking $\mc{C}$ to be the $3$ reducible curves in $\overline{\mc{M}}_{0,4}$ yields $$[P_{1,2}]+[P_{3,4}]=[P_{1,3}]+[P_{2,4}]=[P_{1,4}]+[P_{2,3}].$$
%$$[\{\{1,2\}\}]+[\{\{3,4\}\}]=[\{\{2,3\}\}]+[\{\{4,1\}\}].$$
%Add picture of a square

\subsubsection{Relation to other moduli spaces}
%Figure out the pullback map from PGL_2 ring to GL_2 ring...
%Since $A^{\bullet}_{PGL_2}((\mb{P}^1)^n)\otimes\mb{Q}\to A^{\bullet}_{GL_2}((\mb{P}^1)^n)\otimes\mb{Q}$
%So you don't forget: it feels like PGL_2 cohomology is stronger, because it has that extra torsion element. However, we're projectivizing a vector space, so hopefully the formulas for the diagonals don't involve that extra torsion element? 

If we pick a linearization of the $PGL_2$-action on $(\mb{P}^1)^n$ and there are no strictly semistable points, then excising the unstable locus and applying \cite[Theorem 3]{EG98} gives the rational Chow ring of the GIT quotient. In this case, the ideal given by excision is generated by the classes of the excised strata. See \cite{quiver} for an approach via quiver representations. These GIT quotients are Hassett spaces with total weight $2+\epsilon$ \cite[Section 8]{Hassett} and receive maps from $\overline{\mc{M}}_{0,n}$ via reduction morphisms \cite[Theorem 4.1]{Hassett}, as induced maps between GIT quotients \cite[Theorem 3.4]{HK00}, or by viewing $\overline{\mc{M}}_{0,n}$ as a Chow quotient \cite{Kapranov}. 
\begin{comment}

---------

\begin{thm}
$A^{\le n-2}_{PGL_2}((\mb{P}^1)^n)$ is $\mb{Z}$-linearly spanned by the $[P]$-classes. $A^{\ge n-1}_{PGL_2}((\mb{P}^1)^n)$ is generated by $[Eul(Z_\{1,\ldots,n\})]^k$, and $[Eul'(Z_P)]^k$, where $Eul'$ is the euler class pulled back from either of the two projections $Z_P \to \mb{P}^1$ (they are negations of each other).
\end{thm}
\begin{rmk}
If $\pi_1,\pi_2$ are the two projections $Z_P \to (\mb{P}^1)^2$, then 
\end{rmk}
\begin{proof}
We proceed by induction on $n$. The result is trivially true for $n=2$, so suppose now $n>2$. Then letting $P(i,j)$ be the partition of $\{1,\ldots,n\}$ with $\{i,j\}$ as the only non-singleton part. Then we have the excision exact sequence
$$\bigoplus_{i,j}A^{\bullet-1}_{PGL_2}(Z_{P(i,j)}) \to A^\bullet_{PGL_2}((\mb{P}^1)^n)\to A^\bullet_{PGL_2}((\mb{P}^1)^n \setminus \cup \{Z_{P(i,j)}\})\to 0.$$
But $PGL_2$ acts freely on $(\mb{P}^1)^n \setminus \cup \{Z_{P(i,j)}\}$ with quotient $M_{0,n}$, which is an open subset of $(\mb{A}^{1})^{n-3}$, so $$A^\bullet_{PGL_2}((\mb{P}^1)^n \setminus \cup \{Z_{P(i,j)}\})=A^\bullet(M_{0,n})=0.$$
Now, each $Z_{P(i,j)} \cong (\mb{P}^1)^{n-1}$, and the pushforward of strata in $Z_{P(i,j)} \to (\mb{P}^1)^n$ are $[P]$ classes, so we're done by induction.
\end{proof}

\begin{thm}
$A^{\ge n-1}_{PGL_2}((\mb{P}^1)^n)$is generated by 
\end{thm}
\end{comment}

\subsection{Unordered strata in $[{\rm Sym}^n\mb{P}^1/PGL_2]$}
\label{Unorderedintro}
The $PGL_2$-action on $\mb{P}^1$ induces an action on the symmetric power ${\rm Sym}^n\mb{P}^1\cong \mb{P}^n$, which parameterizes degree $n$ divisors on $\mb{P}^1$. For each partition $\lambda=\{\lambda_1,\ldots,\lambda_d\}$ of $n$, we have the $PGL_2$-invariant subvariety $Z_{\lambda}\subset \mb{P}^n$ consisting of divisors that can be written in the form $\sum_{i=1}^{d}{\lambda_i p_i}$ where $p_i\in \mb{P}^1$. For convenience we often write $\lambda=a_1^{e_1}\ldots a_k^{e_k}$ to be the partition of $n$ where $a_i$ appears $e_i$ times. 

\subsubsection{Integral classes of strata}
We compute the class of $[Z_\lambda]$ in $A^\bullet_{PGL_2}(\mb{P}^n)$. The class of $[Z_\lambda]$ in $A^\bullet_{GL_2}(\mb{P}^n)$ was given in \cite{FNR06}, and we will give a quick independent proof and more compact form in \Cref{FNRformula}. If $n$ is odd, the map $A^\bullet_{PGL_2}(\mb{P}^n)\to A^\bullet_{GL_2}(\mb{P}^n)$ induced by the projection $GL_2\to PGL_2$ is injective (see \Cref{PGLunordered}). Therefore, all of the difficulty lies in computing $[Z_\lambda]$ in $A^\bullet_{PGL_2}(\mb{P}^n)$ for $n$ even. It turns out (see \Cref{integralunordered}) that it suffices to compute the class in $A^\bullet_{PGL_2}(\mb{P}^n)\otimes \mb{Z}/2\mb{Z}$, which takes on a particularly simple form.

%By combining the class in $A^\bullet_{GL_2}(\mb{P}^n)$ given in \cite{FNR06} (which we quickly rederive in a more compact form, see \Cref{FNRformula}), and a computation of the class in $A^\bullet_{PGL_2}(\mb{P}^n)\otimes\mb{Z}/2\mb{Z}$ for $n$ even, we calculate $[Z_\lambda]$ in $A^\bullet_{PGL_2}(\mb{P}^n)$ explicitly (see \Cref{integralunordered}). 
%Given the identification of $A^{\bullet}_{PGL_2}(\mb{P}^n)$ \cite[Corollary 6.3]{FV11}, this determines the class of $[Z_\lambda]$ in $A^\bullet_{PGL_2}(\mb{P}^n)$. 
%One of our main results is the computation of $[Z_\lambda] \in A^\bullet_{PGL_2}(\mb{P}^n)$ explicitly. We will see that computing the classes rationally is straightforward (and the essential computation has appeared in \cite{FNR06}), though we derive a simpler presentation with an alternative derivation). The challenging part is the integral computation due to the presence of $2$-torsion in $A^\bullet_{PGL_2}(\mb{P}^n)$ for $n$ even.

%We include here the computation of $[Z_\lambda]$ in $A^\bullet_{GL_2}(\mb{P}^n)$ and in $A^\bullet_{PGL_2}(\mb{P}^n)\otimes \mb{Z}/2\mb{Z}$, from which we will later see how to easily reconstruct the class in $A^\bullet_{PGL_2}(\mb{P}^n)$.

\begin{thm}
Let $n$ be even and $\lambda=a_1^{e_1}\ldots a_k^{e_k}$ be a partition of $n$ into $d=e_1+\ldots+e_k$ parts. The class of $[Z_\lambda]\in A^\bullet_{PGL_2}(\mb{P}^n)\otimes \mb{Z}/2\mb{Z}\cong \mb{F}_2[c_2,c_3,H]/(q_n(H))$
where $$q_n(t)=\begin{cases}t^{(n+4)/4}(t^3+c_2t+c_3)^{n/4}&n \equiv 0\text{ mod }4\text{, and}
\\
t^{(n-2)/2}(t^3+c_2t+c_3)^{(n+2)/4}&n \equiv 2\text{ mod }4\end{cases}$$
is non-zero precisely when all $a_i$ and $\frac{d!}{e_1!\ldots e_k!}$ are odd and all $e_i$ are even, in which case it is equal to $(\frac{q_n}{q_d})(H)$.
\begin{comment}
\begin{enumerate}
\item Letting $u,v$ be the standard torus characters of diagonal matrices in $GL_2$ and $H=c_1(\ms{O}(1)) \in A^\bullet_{GL_2}(\mb{P}^n)$, we have inside the torsion-free ring $A^\bullet_{GL_2}(\mb{P}^n)=\mb{Z}[u,v]^{S_2}[H]/\prod_{i=0}^n (H+iu+(n-i)v)$ that $(\prod e_i!)[Z_\lambda]$ is the result of expanding out $$\prod_{i=1}^k (z^{a_i}-1)^{e_i}$$ and replacing each monomial $$z^k \mapsto \frac{1}{(u-v)^d}\prod_{j \in \{1,\ldots,n\}\setminus \{k\}}(H+jv+(n-j)u).$$

\item For $n$ even, letting $c_2,c_3\in A^\bullet_{PGL_2}(\mb{P}^n)$ be the second and third equivariant chern classes of the $PGL_2$ representation of $\text{\rm Sym}^2 K^2$ and $H=[\ms{O}(1)]\in A^\bullet_{PGL_2}(\mb{P}^n)$, we have the class of $[Z_\lambda]\in A^\bullet_{PGL_2}(\mb{P}^n)\otimes \mb{Z}/2\mb{Z}=\mb{F}_2[c_2,c_3,H]/q_n(H)$
with $$q_n(t)=\begin{cases}t^{(n+4)/4}(t^3+c_2t+c_3)^{n/4}&n \equiv 0\text{ mod }4\text{, and}
\\
t^{(n-2)/2}(t^3+c_2t+c_3)^{(n+2)/4}&n \equiv 2\text{ mod }4\end{cases}$$
is non-zero precisely when all $a_i$ and $\frac{d!}{e_1!\ldots e_k!}$ are odd and all $e_i$ are even, in which case we have it equal to $(\frac{q_n}{q_d})(H)$.
\end{enumerate}
\end{comment}
\end{thm}

\begin{comment}
\begin{prop}(\Cref{PGLunordered})
For $n$ odd there is no torsion in $A^{\le n}_{PGL_2}(\mb{P}^n)$ and for $n$ even there is $2$-torsion in codimension $\ge 3$. For all $n$ we have $$A^{\le n}_{PGL_2}(\mb{P}^n)/\text{2-torsion}\hookrightarrow A^{\le n}_{GL_2}((\mb{P}^n)),$$ the latter having no torsion.
\end{prop}
%\begin{rmk}
%The entire ring $A^\bullet_{PGL_2}(\mb{P}^n)$ was computed in \cite{FV11} (see \Cref{wholering}).
%\end{rmk}
For unordered strata we will be working with $[Z_\lambda] \in A^{\le n}_{GL_2}((\mb{P}^n))$, so all relations have the caveat that they descend to relations in $A^\bullet({\rm Sym}^n \mc{F})$ for a $\mb{P}^1$-bundle $\mc{F}$ if either $n$ is odd or $\mc{F}$ is the projectivization of a rank $2$ vector bundle, and are true up to $2$-torsion otherwise.
\end{comment}

\begin{comment}As equivariant classes in $A_{PGL_2}^{\bullet}(\mb{P}^n)$, we will work with the normalization 
\begin{align*}
[\lambda]&:= (\prod_{i=1}^{n}{e^\lambda_i!})[Z_{\lambda}] 
\end{align*}
where $e^\lambda_i=\# \{j\mid \lambda_j=i \}$,
\end{comment}

\subsubsection{Relations between strata}
If $\lambda=\{\lambda_1,\ldots,\lambda_d\}=a_1^{e_1}\ldots a_k^{e_k}$ is a partition of $n$, then taking $\Phi:(\mb{P}^1)^n \to {\rm Sym}^n\mb{P}^1$ to be the multiplication map, if $P=\{A_1,\ldots,A_d\}$ is any partition of $[n]$ with $|A_i|=\lambda_i$, we have $$\Phi_*\Delta_P=(\prod e_i!)[Z_\lambda].$$
In particular, every square relation between the classes of ordered strata induces a relation between $[Z_\lambda]$ classes by pushing forward along $\Phi$. %Despite the presence of $2$-torsion in $A_{PGL_2}^{\bullet}(\mb{P}^n)$ for $n$ even, \Cref{unorderedrelations} says rational linear combinations of pushforwards of square relations are in fact still relations between the classes $[Z_\lambda]$, and that they $\mb{Z}$-linearly generate all relations. 

\begin{thm}
\label{unorderedrelations}
(\Cref{integralunordered})
Fix $n$ and choose $a_\lambda\in \mb{Z}$ for each partition of $n$. The following are equivalent:
\begin{enumerate}
\item
$\sum a_\lambda [Z_\lambda]=0$ in $A^\bullet_{PGL_2}(\mb{P}^n)$
\item
$\sum a_\lambda [Z_\lambda]=0$ in $A^\bullet_{GL_2}(\mb{P}^n)$
\item
$\sum a_\lambda [Z_\lambda]$ is formally a rational linear combination of pushforwards of square relations from $A^\bullet_{PGL_2}((\mb{P}^1)^n)$
\item The following identity holds in $\mb{Q}[z]$:
$$\sum_{\lambda=a_1^{e_1}\ldots a_k^{e_k}} \frac{a_\lambda}{\prod_{i=1}^ke_i!}\prod_{i=1}^k (z^{a_i}-1)^{e_i}=0.$$
\end{enumerate}
\end{thm}
\begin{cor}
Every $\mb{Z}$-linear relation that holds between Chow classes of relative $Z_\lambda$-cycles in $A^{\bullet}({\rm Sym}^n\mb{P}(V))$ for every rank 2 vector bundle $V\to B$ and base $B$ holds in $A^{\bullet}({\rm Sym}^n\mc{P})$ for every $\mb{P}^1$-bundle $\mc{P}\to B$ and base $B$.
\end{cor}

We remark that there is $2$-torsion in $A_{PGL_2}^{\bullet}(\mb{P}^n)$ for $n$ even, but \Cref{unorderedrelations} implies that if each $a_\lambda$ is even and $\sum a_\lambda [Z_\lambda]$ is zero in $A^\bullet_{PGL_2}(\mb{P}^n)$, then in fact the same is true for $\sum \frac{a_\lambda}{2} [Z_\lambda]$.
%Also, the equivalence of (1) and (2) in \Cref{unorderedrelations} implies that every linear relation that holds in $A^{\bullet}({\rm Sym}^n\mc{P})$ for every $\mc{P}\cong \mb{P}(V)$ for $V\to B$ a rank 2 vector bundle in fact holds in $A^{\bullet}({\rm Sym}^n\mc{P})$ for all $\mb{P}^1$-bundles $\mc{P}\to B$.
%This implies in particular that every universal linear relation that holds between relative classes of unordered incidence strata in $\text{\rm Sym}^n \mb{P}$ of every projectivization of a rank two vector bundle also holds universally between relative classes in $\text{\rm Sym}^n$ for arbitrary $\mb{P}^1$-bundles.

Rather than search for linear relations between $[Z_\lambda]$ classes using \Cref{unorderedrelations} (4), the following corollary identifies certain partitions whose corresponding strata are a $\mb{Q}$-linear basis for $A^{\leq n-2}_{PGL_2}(\mb{P}^n)\otimes\mb{Q}$, and gives an explicit formula for writing every such class in this basis. Every part of \Cref{unorderedintro} can be deduced from \Cref{unorderedrelations} except that the strata that we choose span $A^{\leq n-2}_{PGL_2}(\mb{P}^n)\otimes \mb{Q}$.

For $\lambda=a_1^{e_1}\ldots a_k^{e_k}$, denote by $[\lambda]$ the normalization $$[\lambda]=(\prod e_i!)[Z_\lambda].$$

\begin{cor}[\Cref{basis} and \Cref{ab1formula}]
\label{unorderedintro} 
For fixed $d\ge 2$, the classes $[\{a,b,1^{d-2}\}]$ form a $\mb{Q}$-basis for $A^{n-d}_{PGL_2}(\mb{P}^n)\otimes\mb{Q}\subset A^{n-d}_{GL_2}(\mb{P}^n)\otimes\mb{Q}$. Writing the polynomial
\begin{align*}
-\frac{1}{(z-1)^{d-2}}\prod_{i=1}^{d} (z^{a_i}-1)&=\sum_{\substack{0 \le k_1\leq k_2\\k_1+k_2=n-d+2}}{\alpha_{k_1}(z^{k_1}+z^{k_2})}, 
\end{align*}
we have $\alpha_i \in \mb{Z}$ and
\begin{align*}
[\{a_1,\ldots,a_d\}]=\sum_{\substack{1 \le k_1\leq k_2\\k_1+k_2=n-d+2}}\alpha_{k_1}[\{k_1,k_2,1^{d-2}\}].
\end{align*}
\end{cor}

Each relation between classes $[Z_{\lambda}]$ in the equivariant Chow ring $A^{\bullet}_{PGL_2}(\mb{P}^n)$ gives relations between enumerative problems. 

\begin{exmp}
\label{exampleintro}
Suppose $n=6$, then \Cref{unorderedintro} implies
\begin{align*}
[Z_{\{4,1,1\}}]+3[Z_{\{2,2,2\}}]=[Z_{\{3,2,1\}}]. \label{examplerelation}
\end{align*}
Consider the following two instances:
\begin{enumerate}
\item
Let $C_t\subset\mb{P}^2$ be a general pencil of degree 6 plane curves. Then, as we vary $C_t$ over $t\in \mb{P}^1$, the number of hyperflex lines plus thrice the number of tritangent lines is equal to the number of lines that are both flex and bitangent. 
\item
Let $X\subset \mb{P}^3$ be a general degree 6 surface. Then in $\mb{G}(1,3)$, the class of the curve of lines that meet $X$ to order $4$ at a point plus three times the class of the curve of tritangent lines to $X$ is equal to the class of the curve of lines that meet $X$ at three points with multiplicities $1,2,3$. 
\end{enumerate}
Note that in both examples, in the absence of a transversality argument, the equalities need to be taken with appropriate multiplicities. 
\end{exmp}

\begin{rmk}
Lines with prescribed orders of contact with a hypersurface were also studied in \cite[Section 5]{Vainsencher}. Counts of these lines are also related to counting line sections of a hypersurface with fixed moduli \cite{Laza, FML}. For the surface $X\subset \mb{P}^3$ in \Cref{exampleintro}, the points $p\in X$ for which a line meets $X$ at $p$ to order 4 is the \emph{flecnode curve}, which is always of expected dimension 1 if $X$ is not ruled by lines by the Cayley-Salmon theorem \cite[Theorem 6]{Katz}, which is a primary tool for bounding the number of lines on a smooth surface in $\mb{P}^3$ (see \cite{Segre} and \cite[Appendix]{Starr}). 
%For the surface $X\subset \mb{P}^3$ in \Cref{exampleintro}, the points $p\in X$ for which a line meets $X$ at $p$ to order 4 is the \emph{flecnode curve}, which is always of expected dimension 1 if $X$ is not ruled by lines by the Cayley-Salmon theorem \cite[Theorem 6]{Katz}, generalized by Landsberg to arbitrary dimensions \cite{Landsberg}. For a smooth surface $X\subset \mb{P}^3$ of degree $d\geq 3$, the flecnode curve yields a bound of $11d^2-24d$ for maximum number of lines in $X$ \cite[Appendix]{Starr} that is close to bound of $11d^2 - 28 d + 12$ obtained by Segre \cite{Segre}, which is still the best known bound for general $d$ \cite[Section 2]{lines}. 

Also, there is no reason not to consider a general variety $X\subset\mb{P}^N$ other than the difficulty of finding a projective variety of higher codimension that has at least a 3-dimensional family of 6-secant lines. 
\end{rmk}

\subsection{Excision}
\label{excisionintro}
As an application of our results, we compute the rational equivariant Chow ring of the complement of a union of unordered strata $A^\bullet_{PGL_2}(\mb{P}^n\setminus \cup_\lambda Z_\lambda)\otimes \mb{Q}=(A^\bullet_{PGL_2}(\mb{P}^n)\otimes\mb{Q})/(\sum_\lambda I_\lambda\otimes \mb{Q})$, where $I_\lambda$ is the ideal of excision for $Z_\lambda$.

We show that $I_{\lambda}\otimes\mb{Q}$ is generated by the classes of strata contained in $Z_{\lambda}$. 
\begin{thm}[\Cref{pushforwardlemma}]
\label{pushforwardintro}
Given a partition $\lambda$ of $n$, $I_{\lambda}\otimes\mb{Q}$ is generated by $[Z_{\lambda'}]$ for all $\lambda'$ that can be obtained from $\lambda$ by merging parts.
\end{thm}
\begin{rmk}
\Cref{pushforwardintro} is false if we replace $I_{\lambda}\otimes\mb{Q}$ with $I_{\lambda}$. This already fails nonequivariantly in the case $n=4$ and $\lambda=\{2,1,1\}$. Indeed, $\Phi: \mb{P}^1\times\mb{P}^2\to \mb{P}^4$ maps birationally onto $Z_{\lambda}$. Let $H_1$ and $H_2$ be the hyperplane classes in the factors of $\mb{P}^1\times\mb{P}^2$ and $H$ be the hyperplane class of $\mb{P}^4$. Then $\Phi_{*}H_1=H^2$, while $[Z_{\{2,2\}}]=8H^2$, $[Z_{\{3,1\}}]=6H^2$, and $[Z_{\{2,1,1\}}]=6H$. 
%Considering the following nonequivariant example. Suppose $n=4$ and let $\lambda =\{2,1,1\}$. Then, $Z_{\lambda}$ is the birational image of a map $\mb{P}^1\times\mb{P}^2\to \mb{P}^4$. Let $H_1,H_2$ be the hyperplane classes on $\mb{P}^1\times\mb{P}^1$ and $H$ be the hyperplane class on $\mb{P}^4$. By taking test classes, 1 pushes forward to $6H$, $H_1$ pushes forward to $H^2$. Similarly, $[Z_{\{3,1\}}]=6H^2$ and $[Z_{\{2,2\}}]=8H^2$ nonequivariantly. No $\mb{Z}$-linear combination of $[Z_{\{3,1\}}]=6H^2$ and $[Z_{\{2,2\}}]$ will give $H^2$, which in $I_{\{2,1,1\}}$. 
\end{rmk}

We typically don't need to use every merged partition $\lambda'$ for dimension reasons by \Cref{unorderedintro}. When $\lambda=\{a,1^{n-a}\}$ is a partition with only one part of size greater than 1, we in fact show that $I_{\lambda}\otimes\mb{Q}$ is generated by just two generators. 

\begin{thm}[\Cref{2generators}]
\label{twogenerators}
Given the partition $\lambda=\{a,1^{n-a}\}$ of $n$, $I_{\lambda}\otimes\mb{Q}$ is generated by $[Z_\lambda]$ and $[Z_{\lambda'}]$, where 
\begin{align*}
\lambda'=
\begin{cases}
\{a+1,1^{n-a-1}\}&\qquad\text{if }a\neq \frac{n}{2}\\
\{a,2,1^{n-a-2}\}&\qquad\text{if }a=\frac{n}{2}. 
\end{cases}
\end{align*}
\end{thm}

In fact we will also show the analogous results with $A^{\bullet}_{GL_2}(\mb{P}^n\setminus \cup_{\lambda} Z_\lambda)\otimes \mb{Q}$, and if we further replace $\mb{P}^n\setminus \cup_{\lambda} Z_\lambda$ with its affine cone $\mb{A}^{n+1}\setminus \cup_{\lambda} \wt{Z_\lambda}$ and consider $A^{\bullet}_{GL_2}(\mb{A}^{n+1}\setminus \cup_{\lambda} \wt{Z_\lambda})$ (see \Cref{GL2andaffine}).

%In many situations, one would like to compute the Chow ring $A^{\bullet}_{GL_2}(\mb{P}^n\backslash \bigcup_{\lambda} Z_{\lambda})$ after removing some strata $Z_{\lambda}$. By excision, this is equivalent to computing the ideal $I_{\lambda}$ given by the pushforward
%$$A^{\bullet}_{GL_2}(Z_{\lambda})\to A^{\bullet}_{GL_2}(\mb{P}^n).$$
In the special case $\lambda=\{2,1^{n-2}\}$, computing $I_{\lambda}$ is the technical heart of the computation of Edidan and Fulghesu of the Chow ring of the stack of hyperelliptic curves of even genus \cite{Fulghesu}. 

For $n$ odd and $Z_{\lambda}$ the unstable locus, i.e with $\lambda=\{\frac{n+1}{2},1^{\frac{n-1}{2}}\}$, the rational Chow ring $A^\bullet_{GL_2}(\mb{P}^n \setminus Z_\lambda)\otimes \mb{Q}$ equals $A^\bullet(\mb{P}^n//GL_2)\otimes \mb{Q}$, the rational Chow ring of the GIT quotient \cite[Theorem 3]{EG98}. For all $n$ and $Z_{\lambda}\subset \mb{P}^n$ the locus of unstable and strictly semistable points, Feh\'er, N\'emethi, and Rim\'anyi computed $A^{\bullet}_{GL_2}(\mb{P}^n\backslash Z_{\lambda})\otimes\mb{Q}$ using a spectral sequence and used the result to compute the rational Chow ring of the GIT quotient \cite[Theorems 4.3 and 4.10]{FNR06}. They actually work with the affine space ${\rm Sym}^n K^2$ instead of $\mb{P}^n$, but the two settings are essentially the same (see \Cref{divideleray}). 

%\begin{thm}[\Cref{pushforwardlemma}]
%\label{pushforwardintro}
%Given a partition $\lambda$ of $n$, $I_{\lambda}\otimes\mb{Q}$ is generated by $[\lambda']$ for all $\lambda'$ that can be obtained from $\lambda$ by merging parts. 
%\end{thm}

\begin{rmk}The affine analogue of \Cref{twogenerators} as given in \Cref{GL2andaffine} in the special case $a=\lceil \frac{n}{2}\rceil$ recovers the $GL_2$-equivariant Chow rings of the stable locus computed in \cite[Theorems 4.3 and 4.10]{FNR06} as described above. The Chow ring of the semistable locus required a separate argument.
\end{rmk}

\subsubsection{Multiplicative relations of affine analogues}
We conclude in \Cref{multiplicative} by describing a combinatorial branching rule for multiplying the affine analogue of the class of a strata $[\wt{Z_\lambda}] \in A^\bullet_{GL_2}({\rm Sym}^n K^2)\cong \mb{Z}[u,v]^{S_2}$ by a generator $u+v$ or $uv$. This generalizes \cite[Remark 3.9 (1)]{FNR06}. 
\begin{comment}
\Cref{twogenerators} implies that if $\lambda=\{a,1^{n-a}\}$, then all the classes $[\lambda'']$ can be generated by $[\lambda]$ and $[\lambda']$ over $A_{GL_2}^{\bullet}(\mb{P}^r)$, meaning that there are multiplicative relations. 

Generalizing this, we give a combinatorial way to explicitly generate multiplicative relations in \Cref{multiplicative} and give formulas to generate all of them in terms of the basis given in \Cref{unorderedintro}. 
\end{comment}

\subsection{Acknowledgements}
The authors would like to thank Mitchell Lee and Anand Patel for helpful conversations during the project. The authors would like to thank Jason Starr for helpful comments and references.

\section{Background and conventions}
\noindent Conventions:
\begin{enumerate}
\item 
The base field $K$ is algebraically closed of arbitrary characteristic
\item
$GL_2$ acts linearly on $\mb{P}^1$ and hence on all products $(\mb{P}^1)^n$, symmetric powers ${\rm Sym^n}\mb{P}^1\cong \mb{P}^n$, and their duals
\item
$T\subset GL_2$ is the standard maximal torus with standard characters $u$ and $v$
\item
$[n]$ denotes the set $\{1,\ldots,n\}$
\item
$\Phi: (\mb{P}^1)^n\to {\rm Sym^n}\mb{P}^1\cong \mb{P}^n$ denotes the multiplication map, where $n$ will be clear from context. 
\end{enumerate}
\subsection{Universal relations and equivariant intersection theory}
\label{universalrelations}
Equivariant intersection theory was formalized in \cite{EG98} and will be used to help us analyze the following situation. See also \cite{Anderson} for an exposition. 

Suppose we have a group $G$ (typically $G=T,GL_2,PGL_2$) acting on a variety $X$ (typically $(\mb{P}^1)^n, \text{\rm Sym}^n\mb{P}^1=\mb{P}^n$), and $G$-invariant subvarieties $Y_i$ (typically incidence strata in $(\mb{P}^1)^n$ or $\mb{P}^n$). Given a principal $G$-bundle $\mc{P} \to B$, we have the $X$-bundle $X_\mc{P}\to B$, where $X_\mc{P}:=\mc{P} \times^G X$. Inside $X_{\mc{P}}$, we have the cycles $$(Y_i)_{X_\mc{P}}:=(Y_i)_\mc{P}\subset X_\mc{P}$$ restricting to $Y_i$ in each fiber $X$, inducing classes $[Y_i]_{X_\mc{P}} \in A_\bullet(X_{\mc{P}})$. We would like to understand what ``universal'' linear relations exist between these classes (i.e. which don't depend on $B$ or $\mc{P}$).

For example, if we take $G=PGL_2$, then we are seeking universal relations between classes $[Z_P]_{\mc{F}^n}$ and between classes $[Z_\lambda]_{\text{\rm Sym}^n\mc{F}}$ for $\mc{F}\to B$ a $\mb{P}^1$-bundle. If we use $G=GL_2$ instead the relations hold a priori only for $\mc{F}$ the projectivization of a rank $2$ vector bundle on $B$.

As we will see in \Cref{eqinttheory}, there is a universal group $A_\bullet^G(X)$ approximated by certain $A_\bullet(X_{\mc{P}'})$ which is equipped with maps $A_\bullet^G(X) \to A_\bullet(X_\mc{P})$ for all $\mc{P}$ and there are classes $[Y_i] \in A_\bullet^G(X)$ such that $[Y_i] \mapsto [Y_i]_{\mc{P}}$, so any relations in $A_\bullet^G(X)$ between the $[Y_i]$ descend to relations between the $[Y_i]_{\mc{P}}$. Conversely, we will see by construction that any relation between the $[Y_i]_{\mc{P}}$ for all $\mc{P}$ induces a relation between the $[Y_i]$. 
%Conversely, we will see that by its construction, any relation between the $[Y_i]$ occurs between $[Y_i]_\mc{P}$ for $\mc{P}$ the principal $G$-bundle $U \to U/G$ for $U$ a subset of affine space with a large codimension complement on which $G$ acts freely.

%Furthermore, for us $X$ will always be smooth, so $A^\bullet_G(X)=A_{\dim(X)-\bullet}^G(X)$ \cite[Proposition 4]{EG98} is a ring.% which is  which is frequently easy to describe in terms of generators and relations, and we will exploit this structure frequently throughout the paper.

\subsection{Equivariant intersection theory}
\label{eqinttheory}
The equivariant Chow group $A^G_\bullet(X)$ is defined as follows. Suppose $G$ acts linearly on a vector space $V$ with an open subset $U$ of codimension $c$ on which it acts freely. Then for any $k < c$, we define $A^G_{\dim(X)-k}(X):=A_{\dim(X\times^G U)-k}(X \times^G U)$. Note that $X \times^G U = X_\mc{P}$ where $\mc{P}$ is the principal $G$-bundle $U \to U/G$. This does not depend on the choice of $V$ \cite[Definition-Proposition 1]{EG98}.

For $\mc{P}\to B$ a principal $G$-bundle over an equidimensional base $B$, we have a map $$
A^G_\bullet(X) \to A_{\dim(B)+\bullet}(P\times^G X)$$ via the composition
\begin{align*}
A^G_{> \dim(X)-c}(X) &\cong A_{> \dim(X\times^G U)-c}(X \times^G U)\\
& \to A_{> \dim((P\times X)\times^G U)-c}((P\times X)\times^G U)\\
& \cong A_{> \dim((P\times X)\times^G U)-c}((P \times X)\times^G V)\\
& \cong A_{> \dim(P \times^G X)-c}(P \times^G X)
\end{align*}
where the second map is induced by flat pullback from the projection, the third map follows from excising $(P \times X)\times^G (V\setminus U)$, and the last map follows from the Chow groups of a vector bundle \cite[Theorem 3.3(a)]{Fulton}. 
%\cite[Lemma 2.2]{Affine}.

Now, we define $A^\bullet_G(X)$ to be the ring of operational $G$-equivariant Chow classes on $X$, i.e. $A^i_G(X)$ is all assignments $$(Y \to X) \mapsto (A^G_{\bullet}(Y) \to A^G_{\bullet-i}(Y))$$
for every $G$-equivariant map $Y\to X$, compatible with the standard operations on Chow groups \cite[Section 2.6]{EG98}. %\marginpar{Make Better}
In our case $X$ is always smooth, and we have the Poincar\'e duality isomorphism $A^\bullet_G(X) = A_{\dim(X)-\bullet}^G(X)$ \cite[Proposition 4]{EG98}, and the identification $$A^\bullet([X/G])\cong A^\bullet_G(X),$$ where $[X/G]$ is the quotient stack \cite[Section 5.3]{EG98}. %Because of the ring structure, $A^\bullet_G(X)$ will be especially convenient to work with, and in all cases we consider $X$ will be smooth.

\subsection{$GL_2$ and $T$-equivariant Chow rings of $(\mb{P}^1)^n$ and $\mb{P}^n$}
\label{ringdescriptions}
%For standard facts on equivariant intersection theory, we refer the reader to \cite{Anderson, EG98}.
We will postpone discussing $PGL_2$-equivariant intersection rings to \Cref{PGL2}. The equivariant Chow rings $A_T^{\bullet}((\mb{P}^1)^n),A^\bullet_T(\mb{P}^n)$, (respectively $A_{GL_2}^{\bullet}((\mb{P}^1)^n), A^\bullet_{GL_2}(\mb{P}^n)$) can be approximated by the ordinary Chow rings of $(\mb{P}^1)^n$ and $\mb{P}^n$ bundles over $\mb{P}^N\times \mb{P}^N$ (respectively the Grassmannian of lines $\mb{G}(1,N)$) for $N>>0$.%, as the base approximates the classifying space $BT$ (respectively $BGL_2$) when $K=\mb{C}$. 

%In addition to being the ring of operational Chow classes on maps from test schemes into the quotient stack \cite[Proposition 19]{EG98}, equivariant chow rings are universal in the following sense. Given a principal $T$-bundle (respectively $GL_2$-bundle) $P$ over a quasiprojective variety $X$, $T$-equivariant (respectively $GL_2$-equivariant) chow classes on $(\mb{P}^1)^n$ or $\mb{P}^n$ can be pulled back to $P\times^{T}(\mb{P}^1)^n$ and $P\times^{T}\mb{P}^n$ (respectively $P\times^{GL_2}(\mb{P}^1)^n$ and $P\times^{GL_2}\mb{P}^n$), after replacing $X$ with an affine bundle $X'\to X$. See \cite[Lemma 1.6]{Totaro} for a precise statement and proof using Jouanolou’s trick \cite{Jouanolou}.% and proof using Jouanolou’s trick \cite{Jouanolou} and see \cite[Lemma 1]{EG97} on how to remove the quasiprojective hypothesis.

Let $u,v$ be the standard characters of $T$. If $\pt$ is a point with trivial $GL_2$ action, then
\begin{align*}
A_{T}^{\bullet}(\pt)=\mb{Z}[u,v], \qquad A_{GL_2}^{\bullet}(\pt)=\mb{Z}[u,v]^{S_2}
\end{align*}
where $S_2$ acts on $\mb{Z}[u,v]$ by swapping $u,v$. By the Chow ring of a vector bundle \cite[Theorem 3.3(a)]{Fulton}, the $T$ (respectively $GL_2$) equivariant Chow ring of an affine space is isomorphic to the equivariant Chow ring of a point. By the projective bundle theorem \cite[Theorem 9.6]{3264}, we have
\begin{align*}
A_T^{\bullet}((\mb{P}^1)^n)=\mathbb{Z}[u,v][H_1,\ldots,H_n]/(F(H_i)),\qquad &A^\bullet_T(\mb{P}^n)=\mathbb{Z}[u,v][H]/(G(H)),\\
A_{GL_2}^{\bullet}((\mb{P}^1)^n)=\mathbb{Z}[u,v]^{S_2}[H_1,\ldots,H_n]/(F(H_i)),\qquad &A^\bullet_{GL_2}(\mb{P}^n)=\mathbb{Z}[u,v]^{S_2}[H]/(G(H))
\end{align*}
where $H_i$ is $c_1(\ms{O}_{\mb{P}^1}(1))$ pulled back to $(\mb{P}^1)^n$ under the $i$th projection and $H$ is $c_1(\ms{O}_{\mb{P}^n}(1))$, and we define 
\begin{align*}
F(z)=(z+u)(z+v),\qquad G(z)=\prod_{k=0}^n(z+ku+(n-k)v)
\end{align*}
for the rest of the document. Even though one might want to use $GL_2$-equivariant Chow rings for applications, $GL_2$-equivariant Chow rings inject into $T$-equivariant Chow rings, so it suffices to only consider $T$-equivariant Chow rings.

The formula for the class of the projectivization of a subbundle \cite[Proposition 9.13]{3264} shows the $i$th coordinate hyperplane in $\mb{P}^n$ has class $H+iu+(n-i)v$. This gives the formula for any torus fixed linear space (for example the torus-fixed points) in $(\mb{P}^1)^n$ or $\mb{P}^n$ by multiplying a subset of these classes.

\subsection{Ordered and unordered strata of $n$ points on $\mb{P}^1$}
\begin{defn}
\label{orderedloci}
Given a collection $P=\{A_1, \ldots, A_d\}$ of disjoint subsets of $[n]$, let $\Delta_P\subset (\mb{P}^1)^n$ denote the $d$-dimensional locus of points $(p_1,\ldots,p_n)$ where $p_i=p_j$ whenever $i,j$ are in the same set $A_k$ of $P$.
\end{defn}

\begin{exmp}
If $P=\{\{1,2,4\},\{3,6\}\}$ and $A=[6]$, then $Z_P\subset (\mb{P}^1)^6$ consists of points $(p_1,\ldots,p_6)$ such that $p_1=p_2=p_4$ and $p_3=p_6$. 
\end{exmp}

\begin{defn}
\label{unorderedloci}
Given a partition $\lambda=\{\lambda_1,\ldots,\lambda_d\}$ of a positive integer $n$, we define the $d$-dimensional subvariety $Z_{\lambda}\subset {\rm Sym^n}\mb{P}^1\cong \mb{P}^n$ to be the image of $\Delta_P$ under the multiplication map $\Phi: (\mb{P}^1)^n\to \mb{P}^n$, where $P=\{A_1, \ldots, A_d\}$ is any partition of $[n]$ with $|A_i|=\lambda_i$. 
\end{defn}

\begin{rmk}
If we view  ${\rm Sym^n}\mb{P}^1\cong \mb{P}^n$ as binary degree $n$ forms on the dual of $\mb{P}^1$, then $Z_{\lambda}$ is the closure of the degree $n$ forms with multiplicity sequence given by $\lambda$, whose equivariant Chow classes were studied by Feh\'er, N\'emethi, and Rim\'anyi \cite{FNR06}.
\end{rmk}

\noindent In order to compactify notation, we make the following definitions.
\begin{defn}
\label{bracket}
Given $P$ a partition of $[n]$ and $\lambda$ a partition of $n$, we let
\begin{align*}
\Delta_P&:=[\Delta_P]\in H_G^{\bullet}((\mb{P}^1)^n)\\
[\lambda]&:= (\prod_{i=1}^{n}{e^\lambda_i!})[Z_{\lambda}] \in H_G^{\bullet}(\mb{P}^n),
\end{align*}
where $G$ is $T$, $GL_2$ or $PGL_2$, depending on the context and $e^\lambda_i=\# \{j\mid \lambda_j=i \}$. For $\lambda=\{a_1,\ldots,a_d\}$, we will often write $[a_1,\ldots,a_d]$ or $[1^{e_1^\lambda},\ldots,n^{e_n^{\lambda}}]$ for $[\lambda]$.
\end{defn}

\begin{rmk}
For any such partition $P$ and $\lambda$ as in \Cref{unorderedloci}, then $\Phi$ maps $\Delta_P$ onto $Z_{\lambda}$ with degree $\prod_{i=1}^{n}{e^\lambda_i!}$, so $\Phi_{*}\Delta_P=[\lambda]$.
\end{rmk}

\subsection{Affine and projective Thom polynomials}
\label{reconstructionsection}
\begin{defn}
\label{cone}
Given a $T$-invariant subvariety $V\subset \mb{P}^n$, let $\wt{V}\subset  \mb{A}({\rm Sym}^n K^2)$ denote the cone of $
V\subset\mb{P}^n$ in $\mb{A}({\rm Sym}^n K^2)\cong \mb{A}^{n+1}$. 
\end{defn}

Given a $T$-invariant subvariety $V\subset \mb{P}^n$, its class $[V]\in A_T^{\bullet}(\mb{P}^n)$ is a polynomial $p(H,u,v)$ of degree at most $n$. The degree 0 term in $H$, $p_0(u,v)$, is $[\wt{V}]\in A_T^{\bullet}(\mb{A}^{n+1})\cong \mb{Z}[u,v]$. This is seen by considering the diagram
\begin{align*}
A_{T}^{\bullet}(\mb{P}^n)\stackrel{\sim}{\leftarrow}A_{T\times \Gm}^{\bullet}(\mb{A}^{n+1}\backslash\{0\})\to A_{T}^{\bullet}(\mb{A}^{n+1}\backslash\{0\})
\end{align*}
and noting that $A_{T}^{k}(\mb{A}^{n+1}\backslash\{0\})\cong A_{T}^{k}(\mb{A}^{n+1})$ for $k\leq n$. 

It turns out $p_0(u,v)$ determines all of $p$. 
\begin{lem}
[{\cite[Theorem 6.1]{FNR05}}]
\label{reconstruction}
We have $p(u,v)=p_0(u+\frac{H}{d},v+\frac{H}{d})$. 
\end{lem}

\begin{proof}[Proof sketch.]As $(\mb{A}^{n+1}\backslash\{0\})/\Gm\cong \mb{P}^n$, $p$ can be computed from $[\wt{V}]\in A_{T\times \Gm}^{\bullet}(\mb{A}^{n+1})$ by mapping to $A_T^{\bullet}(\mb{P}^n)$ via
$$A_{T\times \Gm}^{\bullet}(\mb{A}^{n+1})\to A_{T\times \Gm}^{\bullet}(\mb{A}^{n+1}\backslash\{0\})\cong A_T^{\bullet}(\mb{P}^n).$$ However, the diagonal action of $\Gm$ on $\mb{A}^{n+1}$ actually factors through the action of $T$ on $\mb{A}^{n+1}$, so $A_{T\times \Gm}^{\bullet}(\mb{A}^{n+1})$ contains no more information than $A_T^{\bullet}(\mb{A}^{n+1})$. Taking the class $p_0$ and following it from $A_T^{\bullet}(\mb{A}^{n+1})$ to $A_{T\times \Gm}^{\bullet}(\mb{A}^{n+1})$ and finally to $A_T^{\bullet}(\mb{P}^n)$ yields \Cref{reconstruction}. This argument is written down precisely and in its natural generality in \cite[Theorem 6.1]{FNR05}.
\end{proof}

\section{$PGL_2$ and $GL_2$-equivariant Chow rings}
\label{PGL2}
In this section we compare certain $PGL_2$-equivariant Chow rings to their $GL_2$-equivariant counterparts, which are easier to work with because $GL_2$ is special, so restricting to the maximal torus is an injection on equivariant Chow rings \cite[Proposition 6]{EG98}.

In particular, we show in \Cref{PGL2injective} that $A^{\bullet}_{PGL_2}((\mb{P}^1)^n)\to A^{\bullet}_{GL_2}((\mb{P}^1)^n)$ is injective and identify its image. For the unordered case, we show in \Cref{PGLunordered} that $A^{\bullet}_{PGL_2}(\mb{P}^n)\to A^{\bullet}_{GL_2}(\mb{P}^n)$ is injective for $n$ odd and injective up to $2$-torsion when $n$ is even.

%Hence as we will primarily work with $T$ and $GL_2$-equivariant Chow rings for unordered strata, the relations we find between unordered strata have the caveat that they descend to relations in $A^\bullet({\rm Sym}^n \mc{F})$ for a $\mb{P}^1$-bundle $\mc{F}$ if either $n$ is odd or $\mc{F}$ is the projectivization of a rank $2$ vector bundle, and are true up to $2$-torsion otherwise. There is no such caveat in the ordered case.

To start, we recall a lemma. 
\begin{lem}
[{\cite[Lemma 2.1]{classical}}]
\label{quotientlemma}
Given a linear algebraic group $G$ acting on a smooth variety $X$, let $H$ be a normal subgroup of $G$ that acts freely on $X$ with quotient $X/H$. Then, there is a canonical isomorphism of graded rings
$$A_{G}^{\bullet}(X)\cong A_{G/H}^\bullet(X/H).$$
\end{lem}

\begin{rmk}
\Cref{quotientlemma} was proven in \cite[Lemma 2.1]{classical} directly from the definitions, but it can also be seen as a consequence of the fact that the ring $A^{\bullet}_G(X)$ depends only on the quotient stack $[X/G]$ \cite[Proposition 16]{EG98} and $[[X/H]/(G/H)]\cong [X/G]$ (see \cite[Remark 2.4]{Romagny} or \cite[Lemma 4.3]{Shamil}). 
\end{rmk}
\begin{comment}
\begin{lem}[{\cite[Remark 2.4]{Romagny} \cite[Lemma 4.3]{Shamil}}]
\label{doublequotient}
Suppose $G$ is an algebraic group with a normal subgroup $H$ and $G$ acts on a scheme $X$. If the stack $[X/H]$ is an algebraic space then $G/H$ acts on $[X/H]$ and $[[X/H]/(G/H)]\cong [X/G]$. 
\end{lem}

\begin{lem}[{\cite[Proposition 16]{EG98}}]
\label{WTFEG98}
Let $G$ and $H$ be algebraic groups acting on $X$ and $Y$ respectively. If $[X/G]\cong [Y/H]$ as quotient stacks, then $A^{G}_{\bullet+g}(X)\cong A^{H}_{\bullet+h}(Y)$, where $g=\dim(G)$ and $h=\dim(H)$. 
\end{lem}
\end{comment}

\begin{thm}
\label{PGL2injective}
For $n\ge 1$, the ring homomorphism
\begin{align*}
A^{\bullet}_{PGL_2}((\mb{P}^1)^n)&\to A^{\bullet}_{GL_2}((\mb{P}^1)^n)
%A^{\bullet}_{PGL_2}(\mb{P}^n)&\to A^{\bullet}_{GL_2}(\mb{P}^n) False is n is even
\end{align*}
induced by the quotient map $GL_2\to PGL_2$ is an injection, and the image is generated by the classes $-(2H_i+u+v)$ and $\Delta_{i,j}=H_i+H_j+u+v$.
%If $n\geq 3$, the image is precisely the subring of $A^{\bullet}_{GL_2}((\mb{P}^1)^n)$ generated by the classes $[\{i,j\}]$ for $1\leq i<j\leq n$. 
\end{thm}
\begin{rmk}
\label{redundantrmk}
We will show in \Cref{psidelta} that $\psi_i:=\pi_i^*c_1(T^{\vee}\mb{P}^1)=-(2H_i+u+v)$, as mentioned in \Cref{psi-intro}. For $n \ge 3$ this class is redundant as $$-(2H_i+u+v)=\Delta_{j,k}-\Delta_{i,j}-\Delta_{i,k}.$$
\end{rmk}

\begin{comment}
\begin{rmk}
One can interpret \Cref{PGL2injective} as saying $A^{\bullet}_{PGL_2}((\mb{P}^1)^n)$ is precisely the subring of $A^{\bullet}_{GL_2}((\mb{P}^1)^n)$ generated by the classes of the strata in \Cref{orderedloci} for $n\geq 3$. 

For all $n\ge 1$, the subring is generated by $\{H_i-H_1\mid 2\leq i\leq n\}\cup \{2H_1+u+v\}$. 
\end{rmk}
\end{comment}

\begin{proof}
We show the injectivity of $A^{\bullet}_{PGL_2}((\mb{P}^1)^n)\xrightarrow{\iota} A^{\bullet}_{GL_2}((\mb{P}^1)^n)$ using the commutativity of the diagram
\begin{center}
\begin{tikzcd}
A^{\bullet}_{PGL_2}((\mb{P}^1)^n) \ar[r,"\sim", "q_1"']  \ar[d, "\iota"] & A^{\bullet}_{GL_2}((\mb{A}^2\backslash 0)\times (\mb{P}^1)^{n-1}) \ar[d,"f"] \\
A^{\bullet}_{GL_2}((\mb{P}^1)^n) \ar[r,"\sim", "q_2"'] & A^{\bullet}_{GL_2\times \Gm}((\mb{A}^2\backslash\{0\})\times(\mb{P}^1)^{n-1}) 
\end{tikzcd}
\end{center}
with $f$ induced by the multiplication map $GL_2 \times \Gm \to GL_2$.

%The isomorphisms $q_1$ and $q_2$ are from the fact that given an algebraic group $G$ acting on a scheme $X$, the equivariant chow $A_{\bullet}^G(X)$ depends only on the stack $[X/G]$. 
%We have the isomorphism $[(\mb{A}^2\backslash 0)\times (\mb{P}^1)^{n-1}/\Gm]\cong (\mb{P}^1)^n$ by \Cref{doublequotient}, so we have the isomorphisms $q_1$ and $q_2$ by \Cref{WTFEG98}. %\marginpar{WTF is EG98\\Dennis: fixed}
We have the isomorphisms $q_1$ and $q_2$ by \Cref{quotientlemma}. 

To prove commutativity of the diagram, we can identify each of the rings $A^\bullet_G(X)$ with $A^\bullet([X/G])$ as in \Cref{eqinttheory}, so it suffices to show the following diagram of stacks is commutative.%will use the fact that given an algebraic group $G$ acting on a scheme $X$, $A_{G}^{i}(X)$ is equivalent to maps $A_{\bullet}(S)\to A_{\bullet-i}(S)$ for each map $S\to [X/G]$ compatible with proper pushforward, flat pullbacks, and intersection products for maps of schemes into $[X/G]$ \cite[Proposition 19]{EG98}. Therefore, it suffices to show the following diagrams of stacks is commutative
\begin{center}
\begin{tikzcd}
\left[(\mb{P}^1)^n/PGL_2\right]  & \left[(\mb{A}^2\backslash\{0\})\times(\mb{P}^1)^{n-1}/GL_2\right] \ar[l,"\sim",swap] \\
\left[(\mb{P}^1)^n/GL_2\right]  \ar[u] & \left[(\mb{A}^2\backslash\{0\})\times(\mb{P}^1)^{n-1}/(GL_2\times \Gm)\right] \ar[u] \ar[l,"\sim",swap]
\end{tikzcd}
\end{center}
Suppose we start with a principal $GL_2\times \Gm$-bundle $P\to S$ together with a $GL_2\times \Gm$-equivariant map $P\to (\mb{A}^2\backslash\{0\})\times (\mb{P}^1)^{n-1}$, giving a map $S\to [(\mb{A}^2\backslash\{0\})\times(\mb{P}^1)^{n-1}/(GL_2\times \Gm)]$. Following the diagram around clockwise or counterclockwise, we get a map $S\to [(\mb{P}^1)^n/PGL_2]$ given by a $PGL_2$-equivariant morphism
$$P\times^{GL_2\times \Gm} GL_2\times^{GL_2}PGL_2\cong P\times^{GL_2\times \Gm}PGL_2\to (\mb{P}^1)^n.$$ When going counterclockwise, the product $P\times^{GL_2\times \Gm} GL_2$ is taken with respect to the multiplication map $GL_2\times \Gm\to GL_2$, while when going clockwise, the product is taken with respect to the projection map $GL_2\times \Gm\to GL_2$. However, the resulting principal $PGL_2$-bundle is the same as the compositions with the quotient $GL_2\to PGL_2$ are identical. 

Now, we will find the induced map 
$$A^{\bullet}_{GL_2}((\mb{A}^2\backslash 0)\times (\mb{P}^1)^{n-1})\to A^{\bullet}_{GL_2}((\mb{P}^1)^n)$$
in terms of generators and show it is injective. Consider the diagram
\begin{center}
\begin{tikzcd}
A^{\bullet}_{GL_2}((\mb{A}^2\backslash 0)\times (\mb{P}^1)^{n-1}) \ar[r,"f"] \ar[d,"\sim"]& A^{\bullet}_{GL_2\times \Gm}((\mb{A}^2\backslash\{0\})\times(\mb{P}^1)^{n-1})\ar[d,"\sim"]  \ar[r,"q_2","\sim"']& A^{\bullet}_{GL_2}((\mb{P}^1)^n)\ar[d,"\sim"]\\
A^{\bullet}_{GL_2\times (\Gm)^{n-1}}((\mb{A}^2\backslash 0)^{n}) \ar[r, "f'"] & A^{\bullet}_{GL_2\times (\Gm)^{n}}((\mb{A}^2\backslash\{0\})^n) \ar[r, "q_2'","\sim"'] & A^{\bullet}_{GL_2\times (\Gm)^{n}}((\mb{A}^2\backslash\{0\})^n)
\end{tikzcd}
\end{center}
where $GL_2$ acts in the standard way in all cases. In the middle term of the top row, $\Gm$ acts by scaling $\mb{A}^2\setminus\{0\}$. In the last term of the second row, $(\Gm)^n$ acts by having the $i$th copy of $\Gm$ scale the $i$th copy of $\mb{A}^2\setminus\{0\}$. In the middle term of the second row, $(\Gm)^n$ acts by having the first copy of $\Gm$ act by scaling all copies of $\mb{A}^2\setminus\{0\}$ and the $i$th copy of $\Gm$ with $2 \le i \le n$ acting by scaling the $i$th copy of $\mb{A}^2\setminus\{0\}$. In the first term of the second row, the $i$th copy of $\Gm^{n-1}$ scales the $i+1$st copy of $\mb{A}^2\setminus\{0\}$. 

%every map in the bottom row is induced by group maps, the map $f'$ is induced by the multiplication map $GL_2\times \Gm\to GL_2$, the last $n-1$ $\Gm$'s in the bottom three scale rings act by scaling the respective copies of $\mb{A}^2\setminus\{0\}$, the first $\Gm$ in the domain of $q_2'$ scales all coordinates in $(\mb{A}^2\setminus\{0\})^n$ and the first $\Gm$ in the codomain of $q_2'$ scales the first copy of $\mb{A}^2\setminus\{0\}$.
To compute $f'$, we let $H_1$ be the standard character on the first factor of $\Gm$ in $GL_2\times(\Gm)^{n}$ and let $H_2,\ldots,H_n$ be the standard characters on the remaining $n-1$ factors and the $n-1$ factors of $\Gm$ in $GL_2\times(\Gm)^{n-1}$.  The induced map $T\times (\Gm)^n\to T\times (\Gm)^{n-1}$ of tori induces $u\mapsto u+H_1$ and $v\mapsto v+H_1$. Therefore, 
$$f': \frac{\mb{Z}[u,v]^{S_2}[H_2,\ldots,H_n]}{(uv,F(H_2),\ldots, F(H_n))}\to \frac{\mb{Z}[u,v]^{S_2}[H_1][H_2,\ldots,H_n]}{(uv,F(H_2+H_1),\ldots, F(H_n+H_1))},$$
where  $u\mapsto u+H_1$, $v\mapsto v+H_1$, and $H_i\mapsto H_i$. 
%The map $f: A^{\bullet}_{GL_2}((\mb{A}^2\backslash 0)\times (\mb{P}^1)^{n-1})\to A^{\bullet}_{GL_2\times \Gm}((\mb{A}^2\backslash\{0\})\times(\mb{P}^1)^{n-1})$ is given by the multiplication map $GL_2\times \Gm\to GL_2$, where $\Gm$ is identified with the scalar matrices in $GL_2$. If we let $t$ be the standard character on $\Gm$, then the induced map $T\times \Gm\to T$ of tori induces $u\mapsto u+t$ and $v\mapsto v+t$. Therefore, $f$ is the injective ring map
%$$\frac{\mb{Z}[u,v]^{S_2}[H_2,\ldots,H_n]}{(uv,F(H_2),\ldots, F(H_n))}\to \frac{\mb{Z}[u,v]^{S_2}[t][H_2,\ldots,H_n]}{(uv,F(H_2+t),\ldots, F(H_n+t))},$$
%where  $u\mapsto u+t$, $v\mapsto v+t$, and $H_i\mapsto H_i$. 

For $q_2'$, the induced map $T\times (\Gm)^n\to T\times (\Gm)^{n}$ of tori induces $H_1\mapsto H_1$, $H_i\mapsto H_i-H_1$ for $2\leq i\leq n$ and $u\mapsto u$, $v\mapsto v$, and gives the map $$q_2': 
\frac{\mb{Z}[u,v]^{S_2}[H_1][H_2,\ldots,H_n]}{(uv,F(H_2+H_1),\ldots, F(H_n+H_1))}\to \frac{\mb{Z}[u,v]^{S_2}[H_1,\ldots,H_n]}{(F(H_1),\ldots, F(H_n))}.$$ The composite 
$$q_2'\circ f': \frac{\mb{Z}[u,v]^{S_2}[H_2,\ldots,H_n]}{(uv,F(H_2),\ldots, F(H_n))}\to \frac{\mb{Z}[u,v]^{S_2}[H_1,\ldots,H_n]}{(F(H_1),\ldots, F(H_n))}$$
is given by $u\mapsto u+H_1$, $v\mapsto v+H_1$, $H_i\mapsto H_i-H_1$ for $2\leq i\leq n$. The image is therefore generated by $2H_1+u+v$ and $H_i-H_1$ for $2\leq i\leq n$. If $n\geq 3$, then this is generated by the collection 
$$\{H_i+H_j+u+v\mid 1\leq i<j\leq n\}=\{\Delta_{i,j}\mid 1\leq i<j\leq n\}$$
(see \Cref{psidelta}).
\end{proof}

\begin{rmk}
\label{equivariantcohomology}
Suppose our base field is $\mb{C}$. We have a commutative diagram
\begin{center}
\begin{tikzcd}
A^{\bullet}_{PGL_2}((\mb{P}^1)^{n}) \ar[r, hook,"q_1"]\ar[d] & A^{\bullet}_{GL_2}((\mb{P}^1)^{n}) \ar[d]\\
H^{\bullet}_{PGL_2}((\mb{P}^1)^{n}) \ar[r,"q^H_1"'] & H^{\bullet}_{GL_2}((\mb{P}^1)^{n})
\end{tikzcd}
\end{center}
The map $A^{\bullet}_{GL_2}((\mb{P}^1)^{n})\to H^{\bullet}_{GL_2}((\mb{P}^1)^{n})$ is an isomorphism by the Leray-Hirsch theorem applied to $\mb{P}^1_{\mb{C}}$-bundles. Running the proof of \Cref{PGL2injective} for the map $q_1^H$ shows $q_1^H$ is injective. Here we replace the projective bundle theorem in algebraic geometry by the Leray-Hirsch theorem applied to $\mb{P}^1_{\mb{C}}$-bundles and the application of \Cref{quotientlemma}
%\Cref{doublequotient,WTFEG98} 
with the fact that if $G$ acts on $X$ and $H$ is a normal subgroup which acts freely, then $(X\times EG)/G\cong ((X\times EG)/H)/(G/H)$, and $(X\times EG)/H$ is homotopy equivalent to $X/H$ and has a free action by $G/H$. 

This implies $A^{\bullet}_{PGL_2}((\mb{P}^1)^{n})\to H^{\bullet}_{PGL_2}((\mb{P}^1)^{n}) $ is an isomorphism.
\end{rmk}

By \cite[Theorem 1]{Pandharipande}, the injection $SO(3)\to GL_3$ induces a surjection $A^{\bullet}_{GL_3}(\pt)\to A^{\bullet}_{SO(3)}(\pt)$ expressing $A^\bullet_{SO(3)}(\pt)\cong \mb{Z}[c_1,c_2,c_3]/(c_1,2c_3)$, where $c_1,c_2,c_3$ are the generators of $A^{\bullet}_{GL_3}(\pt)$. \Cref{PGLGL} expresses the map $A^{\bullet}_{PGL_2}(\pt)\to A^{\bullet}_{GL_2}(\pt)$ in terms of this presentation. 
\begin{lem}
\label{PGLGL}
Under the composition, 
$$A^{\bullet}_{GL_3}(\pt)\to A^{\bullet}_{SO(3)}(\pt)\cong A^{\bullet}_{PGL_2}(\pt)\to A^{\bullet}_{GL_2}(\pt)\to A^{\bullet}_T(\pt),$$
we have $c_1\mapsto 0$, $c_2\mapsto -(u-v)^2$, $c_3\mapsto 0$. 
\end{lem}

\begin{proof}
\Cref{PGLGL} amounts to finding the map $$T\to GL_2\to PGL_2\cong SO(3)\to GL_3$$ 
inducing the maps of rings.

To describe the isomorphism $SO(3)\cong PGL_2$, recall that $GL_2$ acts on the space $K^{2\times 2}$ of 2 by 2 matrices by conjugation. There is a pairing $\langle \bullet , \bullet \rangle$ on $K^{2\times 2}$ given by $\langle A,B\rangle=\on{Tr}(AB)$ that restricts to a nondegenerate form on the three-dimensional vector space of trace zero matrices $V\subset K^{2\times 2}$. Since the action of $GL_2$ preserves $\langle \bullet , \bullet \rangle$ and the scalar matrices inside $GL_2$ act trivially, we have an injection $PGL_2\to SO(3)$, which is an isomorphism for dimension reasons. 

Under this isomorphism, $\begin{pmatrix} u & 0 \\ 0 & v\end{pmatrix}\in T$ maps into diagonal matrices in $GL_3$ and acts on $V$ with characters $u-v$, $v-u$ and 0 (written additively). Therefore, 
\begin{align*}
c_1&\mapsto (u-v)+(v-u)=0\\
c_2&\mapsto (u-v)(v-u)=-(u-v)^2\\
c_3&\mapsto 0(u-v)(v-u)=0. 
\end{align*}
\end{proof}

\begin{prop}
\label{PGLunordered}
We have
\begin{align*}
A^{\bullet}_{PGL_2}(\mb{P}^n) &\cong 
\begin{cases}
\mb{Z}[u,v]^{S_2}/(\prod_{i=0}^{n}((\frac{n+1}{2}-i)u+(\frac{-n+1}{2}+i)v)) &\qquad\text{if $n$ is odd}\\
\mb{Z}[c_2,c_3,H]/(2c_3,p_n(H))&\qquad\text{if $n$ is even}
\end{cases}
\end{align*}
where $p_n(t)\in A^{\bullet}_{PGL_2}(\pt)[t]$ is defined as
\begin{align*}
p_n(t) &=
\begin{cases}
t\prod_{k=1}^{\frac{n}{2}}(t^2+k^2c_2)+t^{\frac{n}{4}+1}\sum_{k=1}^{\frac{n}{4}}{\binom{\frac{n}{4}}{k}(t^3+c_2t)^{\frac{n}{4}-k}c_3^k}&\quad n\equiv 0\pmod{4}\\
t\prod_{k=1}^{\frac{n}{2}}{(t^2+k^2c_2)}+t^{\frac{n-2}{4}}\prod_{k=1}^{\frac{n+2}{4}}{\binom{\frac{n+2}{4}}{k}(t^3+c_2t)^{\frac{n+2}{4}-k}c_3^k}&\quad n\equiv 2\pmod{4}.
\end{cases}
\end{align*}
The map 
\begin{align*}
A^{\bullet}_{PGL_2}(\mb{P}^n)\to A^{\bullet}_{GL_2}(\mb{P}^n)
\end{align*}
induced by $GL_2\to PGL_2$ is given by 
\begin{alignat*}{2}
&u\mapsto H+\frac{n+1}{2}u+\frac{n-1}{2}v \quad v\mapsto H+\frac{n-1}{2}u+\frac{n+1}{2}v &\qquad\text{if $n$ is odd, and}\\
&c_2\mapsto -(u-v)^2 \quad c_3\mapsto 0\quad H\mapsto H+\frac{n}{2}(u+v) &\qquad\text{if $n$ is even.}
\end{alignat*}
Finally, $A^{\bullet}_{PGL_2}(\mb{P}^n)\to A^{\bullet}_{GL_2}(\mb{P}^n)$ is injective for $n$ odd and injective up to $2$-torsion when $n$ is even.
\end{prop}

\begin{proof}
The injectivity statements immediately follow from the explicit descriptions of all of the rings maps in the statement of \Cref{PGLunordered}, we omit the verification.

We do the cases $n$ is odd and even separately. First suppose $n$ is odd. Consider the commutative diagram
\begin{center}
\begin{tikzcd}
 & A^\bullet_{PGL_2}(\mb{P}^n) \ar[r]\ar[d,"\sim"] & A^\bullet_{GL_2}(\mb{P}^n)\ar[d,"\sim"]\\
A^\bullet_{GL_2}(\mb{A}^{n+1}\setminus\{0\}) \ar[r,"\phi_1"',"\sim"] & A^\bullet_{GL_2/\mu_{n}}(\mb{A}^{n+1}\setminus\{0\}) \ar[r, "\phi_2"] & A^\bullet_{GL_2\times \Gm}(\mb{A}^{n+1}\setminus \{0\})
\end{tikzcd}
\end{center}
The map $\phi_1$ is induced by the isomorphism $GL_2/\mu_n \to GL_2$ given by $A\mapsto (\det A)^{\frac{n-1}{2}}A$ \cite[Proposition 4.4]{Vistoli}. To determine $A^\bullet_{GL_2}(\mb{A}^{n+1}\setminus\{0\}) $ it suffices to check how the maximal torus $T\subset GL_2$ acts on $\mb{A}^{n+1}$. Since the inverse of $GL_2\to GL_2/\mu_n$ is given by $A\mapsto (\det A)^{\frac{1-n}{2n}} A$, $\begin{pmatrix} \lambda_1 & \\ & \lambda_2\end{pmatrix}$ maps to $\begin{pmatrix} \lambda_1^{\frac{n+1}{2n}}\lambda_2^{\frac{1-n}{2n}} & \\ & \lambda_1^{\frac{1-n}{2n}}\lambda_2^{\frac{n+1}{2n}}\end{pmatrix}$ in $GL_2/\mu_{n}$ and acts on $\mb{A}^{n+1}$ with characters $\{(\frac{n+1}{2}-i)u+(\frac{-n+1}{2}+i)v\mid 0\leq i\leq n\}$. This shows  
\begin{align*}
A^{\bullet}_{PGL_2}(\mb{P}^n) &= \mb{Z}[u,v]^{S_2}/(\prod_{i=0}^{n}((\frac{n+1}{2}-i)u+(\frac{-n+1}{2}+i)v))
\end{align*}
in this case. 

To find the map $A^\bullet_{GL_2}(\mb{A}^{n+1}\setminus\{0\})\to A^\bullet_{GL_2\times \Gm}(\mb{A}^{n+1}\setminus \{0\})$, we consider the map $GL_2\times \Gm\to GL_2$ and find it maps the pair $\begin{pmatrix} \lambda_1 & \\ & \lambda_2\end{pmatrix},\lambda$ to $\lambda^{\frac{1}{n}} \begin{pmatrix} \lambda_1 & \\ & \lambda_2\end{pmatrix}$ in $GL_2/\mu_n$ and $\begin{pmatrix} \lambda \lambda_1^{\frac{n+1}{2}}\lambda_2^{\frac{n-1}{2}} & \\ & \lambda \lambda_1^{\frac{n-1}{2}}\lambda_2^{\frac{n+1}{2}}\end{pmatrix}$ in $GL_2$. This shows the map 
\begin{align*}
\mb{Z}[u,v]^{S_2}/(\prod_{i=0}^{n}((\frac{n+1}{2}-i)u+(\frac{-n+1}{2}+i)v))\to \mb{Z}[u,v]^{S_2}[H]/(\prod_{i=0}^{n}{(H+iu+(n-i)v)})
\end{align*}
giving $A^\bullet_{GL_2}(\mb{A}^{n+1}\setminus\{0\})\to A^\bullet_{GL_2\times \Gm}(\mb{A}^{n+1}\setminus \{0\})$ is given by 
\begin{align*}
u\mapsto H+\frac{n+1}{2}u+\frac{n-1}{2}v \quad v\mapsto H+\frac{n-1}{2}u+\frac{n+1}{2}v.
\end{align*}
Now, we do the case $n$ is even. Let $V\cong K^2$ be a 2-dimensional vector space with the standard representation of $GL_2$. Let $D\cong K$ be a 1-dimensional vector space where $GL_2$ acts by multiplication by the determinant. Then, $({\rm Sym}^nV)\otimes (D^{\vee})^{\otimes n}$ is a $GL_2$ representation that descends to a $PGL_2$ representation. 

To determine 
$$A^\bullet_{PGL_2}(\mb{P}^n)\cong A^\bullet_{PGL_2}(\mb{P}(({\rm Sym}^nV)\otimes (D^{\vee})^{\otimes n/2}))$$
it suffices to find the chern classes of the $PGL_2$ representation $({\rm Sym}^nV)\otimes (D^{\vee})^{\otimes n/2}$ regarded as a $PGL_2$-equivariant vector bundle over a point. These chern classes are given in \cite[Corollary 6.3]{FV11}. The reader should also note that \cite{FV11} contains mistakes elsewhere in the document (see \cite[Introduction]{Lorenzo}). As a result, we have $A^{\bullet}_{PGL_2}(\mb{P}^n)$ is $\mb{Z}[c_2,c_3,H]/(2c_3,p_n(H))$, where $p_n(t)\in A_{PGL_2}(\pt)[t]$ is given as in the statement of the proposition. 

Therefore, we have 
$$A^\bullet_{PGL_2}(\mb{P}(({\rm Sym}^nV)\otimes (D^{\vee})^{\otimes n/2}))\to A^\bullet_{GL_2}(\mb{P}(({\rm Sym}^nV)\otimes (D^{\vee})^{\otimes n/2}))$$
given by $c_2\mapsto -(u-v)^2$ and $c_3\mapsto 0$ by \Cref{PGLGL}. Also, the $\mathscr{O}_{\mb{P}(({\rm Sym}^nV)\otimes (D^{\vee})^{\otimes n/2})}(1)$ class in $A^\bullet_{PGL_2}(\mb{P}(({\rm Sym}^nV)\otimes (D^{\vee})^{\otimes n/2})$ maps to the $\mathscr{O}_{\mb{P}(({\rm Sym}^nV)\otimes (D^{\vee})^{\otimes n})}(1)$ class in $A^\bullet_{GL_2}(\mb{P}(({\rm Sym}^nV)\otimes (D^{\vee})^{\otimes n/2})$ by the projective bundle formula. 

Finally, since $({\rm Sym}^nV)\otimes (D^{\vee})^{\otimes n/2}$ is a twist of ${\rm Sym}^nV$ by a $GL_2$-equivariant line bundle, the $\mathscr{O}_{\mb{P}(({\rm Sym}^nV)\otimes (D^{\vee})^{\otimes n/2})}(1)$ class in $A^\bullet_{GL_2}(\mb{P}(({\rm Sym}^nV)\otimes (D^{\vee})^{\otimes n/2})$ maps to $\mathscr{O}_{\mb{P}({\rm Sym}^nV)}(1)+c_1^{GL_2}(D^{\otimes \frac{n}{2}})$ in $A^\bullet_{GL_2}(\mb{P}({\rm Sym}^nV)^{\otimes n/2})$. Since $c_1^{GL_2}(D^{\otimes n/2})=\frac{n}{2}(u+v)$, we find the composite map 
$$A^\bullet_{PGL_2}(\mb{P}(({\rm Sym}^nV)\otimes (D^{\vee})^{\otimes n/2})\to A^\bullet_{GL_2}(\mb{P}({\rm Sym}^nV))$$
is given by 
$$c_2\mapsto -(u-v)^2 \quad c_3\mapsto 0\quad H\mapsto H+\frac{n}{2}(u+v).$$
%In this case, we note that the $PGL_2$ action on $\mb{P}^n$ lifts to a representation on $\mb{A}^{n+1}$ (regarded as an $n+1$-dimensional vector space), as, under the natural $SL_2$ action on $\mb{A}^{n+1}$ induced by the inclusion $SL_2\subset GL_2$, the subgroup $\{\pm I_2\}\subset SL_2$ acts trivially. 

%Therefore, $A^\bullet_{PGL_2}(\mb{P}^n)$ is given by the projective formula once we know the chern classes of the vector bundle over a point given by the $PGL_2$ action on $\mb{A}^{n+1}$. These chern classes are given in \cite[Corollary 6.3]{FV11}. The reader should also note that \cite{FV11} contains mistakes (see \cite[Introduction]{Lorenzo}). As a result, we have $A^{\bullet}_{PGL_2}(\mb{P}^n)$ is $\mb{Z}[c_2,c_3,H]/(2c_3,p_n(H))$, where $p_n(t)\in A_{PGL_2}(\pt)[t]$ is given as in the statement of the proposition. 
\end{proof}

\section{Formulas and initial reductions}

In this section we express the $\Delta_P$ and $[\lambda]$ classes in terms of our equivariant Chow ring presentations. 

After this, we compute formulas for $\Delta_P\in A^\bullet_T((\mb{P}^1)^n)$ and give a quick, alternative computation of the classes $[Z_\lambda]\in A^\bullet_T(\mb{P}^n)$ given in \cite[Theorem 3.4]{FNR06}. The simple presentation for the class of the diagonal in $(\mb{P}^1)^n$ works especially well with the formula for the pushforward $\Phi_*:A^\bullet_T((\mb{P}^1)^n) \to A^\bullet_T(\mb{P}^n)$ via the classes of torus fixed points, and appears not to have been previously exploited in this fashion.
%\Cref{PGL2injective} shows that we lose no information in the ordered case by expressing these classes with $T$ and $GL_2$-equivariant Chow classes/rings.%We were not able to find the presentation in \Cref{diagonal} in the literature, which yields especially compact presentations for \Cref{FormulaP} and \Cref{FNRformula} that are crucial for all subsequent results in this paper.

\subsection{Class of the diagonal in $(\mb{P}^1)^n$}
We now compute the $T$-equivariant class of the diagonal $\Delta_{\{[n]\}} \subset (\mb{P}^1)^n$. This formula would also follow from localization to the torus fixed points, but the derivation below is simpler.
\begin{prop}
\label{diagonal}
The class of $\Delta_{\{[n]\}}$ in $A^\bullet_T((\mb{P}^1)^n)$ is given by
$$\Delta_{\{[n]\}}=\frac{1}{u-v}\left(\displaystyle\prod_{i=1}^{n} (H_i+u)-\displaystyle\prod_{i=1}^{n} (H_i+v)\right).$$
\end{prop}
\begin{proof}
This is a result of the fact that $\Delta_{\{[n]\}}$ intersected with $\{[0:1]\}\times (\mb{P}^1)^{n-1}$ and $\{[1:0]\}\times (\mb{P}^1)^{n-1}$ are the torus-fixed points $[0:1]^n$ and $[1:0]^n$ respectively, so
\begin{align*}
((H_1+u)-(H_1+v))\Delta_{\{[n]\}} = \prod_{i=1}^{n}{(H_i+u)}-\prod_{i=1}^{n}{(H_i+v)}. 
\end{align*}
\end{proof}

\subsection{Formula for $\Delta_P$}
When two strata $\Delta_P$ and $\Delta_{P'}$ intersect transversely in $(\mb{P}^1)^n$, it is easy to describe their intersection as another stratum. 
\begin{prop}\label{trivialfacts}
The class $\Delta_P \in A^{n-d}_{PGL_2}((\mb{P}^1)^n)$ for $P$ a partition of $[n]$ into $d$ parts is given by the product $\prod_{\{i,j\} \in \text{Edge}(T)} \Delta_{i,j}$, where $T$ is any forest with vertex set $[n]$ consisting of one spanning tree for each part of $P$. In particular,
\begin{enumerate}
\item If $i,j$ are in distinct parts of $P$, then if $P_{ij}$ is the partition merging the parts containing $i$ and $j$, we have $\Delta_{i,j}\Delta_P=\Delta_{P_{ij}}$.
\item If $i,j,i',j'$ are in the same part of $P$, we have $\Delta_{i,j}\Delta_P=\Delta_{i',j'}\Delta_P$.
\end{enumerate}
\begin{comment}
Given disjoint sets $A,B$, suppose we have partitions $P_A,P_B$ of $A,B$ respectively. Then in $(\mb{P}^1)^{A\sqcup B}$ we have
$$\Delta_{P_A \sqcup P_B}=\Delta_{P_A}\Delta_{P_B}.$$
Given sets $C,D$ such that $|C \cap D|=1$ and partitions $P_C,P_D$ of $C,D$ respectively, if $P_{C\cup D}$ is the partition of $C \cup D$ formed by merging the parts of $P_C$ and $P_D$ containing $C \cap D$, then in $(\mb{P}^1)^{C\cup D}$, we have
$$\Delta_{P_{C \cup D}}=\Delta_{P_C}\cdot \Delta_{P_D}.$$
\end{comment}
\end{prop}
\begin{proof}
Item (1) follows from the transversality of the intersection $\Delta_{i,j}\cap \Delta_P$, from which we can deduce the first part of the proposition, and item (2) then follows from the first part and repeated applications of the diagonal relation $\Delta_{i,j}\Delta_{i,k}=\Delta_{i,j}\Delta_{j,k}$.
\begin{comment}
\Cref{trivialfacts} follows from the fact that the corresponding intersections
\begin{align*}
\Delta_{P_A}\cap \Delta_{P_B} = \Delta_{P_A\sqcup P_B}\text{ and }\Delta_{P_C}\cap \Delta_{P_D} = \Delta_{P_C\cup P_D}.
\end{align*}
are transverse.
\end{comment}
\end{proof}

\begin{prop}
\label{FormulaP}
Let $P=\{V_1, \ldots ,V_d\}$ be a partition of $[n]$, then
\begin{align*}
\Delta_P&=\frac{1}{(u-v)^d}\prod_{i=1}^{d}{\left(\prod_{j\in V_i}{(H_j+u)}-\prod_{j\in V_i}{(H_j+v)}\right)}.
\end{align*}
\end{prop}
\begin{proof}
From \Cref{trivialfacts}, $\Delta_P=\prod_{i=1}^{d}\Delta_{\{V_i\}}$. Now apply \Cref{diagonal}. 
\end{proof}
\subsection{The $\psi_i$ and $\Delta_{i,j}$ classes}
At this point, we can prove the formula for $\Delta_{i,j}$ in item (3) of \Cref{bigthmintro} and for $\psi_i$ as mentioned in \Cref{psi-intro}.
\begin{prop}
\label{psidelta}
We have
\begin{align*}
\Delta_{i,j}&=H_i+H_j+u+v\\
\psi_i&=-(2H_i+u+v).
\end{align*}
\end{prop}
\begin{proof}
The formula for $\Delta_{i,j}$ is an immediate consequence of \Cref{FormulaP}.

To compute $\psi_i$, it suffices to show that $c_1(T_{\mb{P}^1})\in A^\bullet_{GL_2}(\mb{P}^1)$ is $2H+u+v$, where $H=c_1(\ms{O}(1))$. We note that $c_{{\rm top}}(T_X)$ for any smooth $X$ is the pullback of the diagonal under the diagonal map $X\to X \times X$. The pullback $A^{\bullet}_{GL_2}((\mb{P}^1)^2)\to A^{\bullet}_{GL_2}(\mb{P}^1)$ under the inclusion $\mb{P}^1\cong \Delta_{1,2}\hookrightarrow \mb{P}^1\times \mb{P}^1$ is given by $H_1,H_2 \mapsto H$. Under this map, $\Delta_{1,2}$ pulls back to $2H+u+v$ as desired.
\end{proof}

\subsection{Pullback and Pushforward under $\Phi$}
\label{pushforward}
The pullback map $\Phi^*:A^\bullet_T(\mb{P}^n)\to A^\bullet_T((\mb{P}^1)^n)$ is induced by $$\Phi^*(H)=\sum_{i=1}^n H_i.$$

We now consider $\Phi_{*}$. By considering the classes of the torus-fixed loci, we have for any $A\subset [n]$, $$\Phi_*\left(\prod_A (H_i+u)\prod_{[n]\setminus A} (H_j+v)\right) = \prod_{k \in [n]\setminus\{|A|\}} (H+kv+(n-k)u).$$
This in fact uniquely characterizes $\Phi_*$, which can be seen either from localization \cite[Theorem 2]{EG98b} or because
\begin{align*}
\frac{\prod_A (H_i+u)\prod_{[n]\setminus A} (H_j+v)}{\prod_A (-v+u)\prod_{[n]\setminus A} (-u+v)}
\end{align*}
is a Lagrange interpolation basis for polynomials in $H_1,\ldots,H_n$ modulo $F(H_i)$ for each $i$.

\subsection{Formula for $[\lambda]$}
Feh\'er, N\'emethi, and Rim\'anyi computed the class of $[\lambda]$ for $\lambda$ a partition of $n$ \cite[Theorem 3.4]{FNR06}. We can give a quick self-contained computation from \Cref{pushforward,diagonal} as follows.

\begin{thm}[{\cite[Theorem 3.4]{FNR06}}]
\label{FNRformula}
The class $[a_1,\ldots,a_d]$ is the result of first expanding the polynomial $$\prod_{i=1}^d (z^{a_i}-1)=\sum_{k\geq 0}{c_k z^k}\qquad (c_k\in \mb{Z}),$$ and then replacing each monomial $$z^k\mapsto \frac{1}{(u-v)^d}\prod_{j \in [n]\setminus \{k\}}(H+jv+(n-j)u).$$
\end{thm}
\begin{proof}
Let $P=\{V_1, \ldots ,V_d\}$ be a partition of $[n]$ with $|V_i|=a_i$. We expand the formula from \Cref{FormulaP}
\begin{align*}
\Delta_P &=\frac{1}{(u-v)^d}\prod_{i=1}^{d}{\left(\prod_{j\in V_i}{(H_i+u)}-\prod_{j\in V_i}{(H_i+v)}\right)}
\end{align*}
to a sum of terms of the form $\prod_{i \in A} (H_i+u) \prod_{j \in [n]\setminus A}(H_j+v)$. Then, \Cref{pushforward} implies that each such term pushes forward to $\prod_{j \in [n]\setminus \{|A|\}}(H+jv+(n-j)u)$. The result follows immediately.
%Then $\frac{1}{(u-v)^r}\prod (z^{a_i}-1)$ is obtained by replacing each $\prod_{i \in A} (H_i+u) \prod_{j \in [n]\setminus A} (H_j+v)$ with $z^{|A|}$. The result then follows from the formula for $\Phi_*\left(\prod_A (H_i+u)\prod_{[n]\setminus A} (H_j+v)\right)$.
\end{proof}

\section{Strata in $[(\mb{P}^1)^n/PGL_2]$}
\label{orderedstratasection}
In this section we prove all of our results on ordered point configurations in $\mb{P}^1$. Up to \Cref{algandexmp}, the only result that we use is \Cref{PGL2injective}, and in particular the identification of $\Delta_{i,j}$ in $A^\bullet_{GL_2}((\mb{P}^1)^n)$ as $H_i+H_j+u+v$.

\begin{rmk}
Whenever we write $\Delta_{i,j}$ in any context, we will always treat $\{i,j\}$ as an unordered tuple, so that implicitly $$\Delta_{i,j}:=\Delta_{j,i}$$
for $i>j$. 
\end{rmk}

Recall from \Cref{PGL2injective} and \Cref{ringdescriptions}, we have the inclusions
$$A^\bullet_{PGL_2}((\mb{P}^1)^n)\subset A^\bullet_{GL_2}((\mb{P}^1)^n) \subset A^\bullet_{T}((\mb{P}^1)^n).$$
We first consider the square relation in $(\mb{P}^1)^4$. 
\begin{prop}
\label{P14}
In $A_{PGL_2}^{\bullet}((\mb{P}^1)^4)$, we have the square relation
\begin{align*}
\Delta_{1,2}+\Delta_{3,4}=\Delta_{2,3}+\Delta_{4,1}.
\end{align*}
\end{prop}
\begin{proof}
Both sides are equal to $H_1+H_2+H_3+H_4+2(u+v)$ by \Cref{psidelta}. This can also be shown using the fact that the diagonal $\Delta \subset \mb{P}^1\times\mb{P}^1$ has a torus-equivariant deformation to $\{0\}\times \mb{P}^1\cup \mb{P}^1\times\{\infty\}$. It also holds by \Cref{WDVVintro}. 
\end{proof}

\begin{defn}
Let $R(n)$ be the ring 
$$R(n)=\mb{Z}[\{\Delta_{i,j}\mid 1\leq i<j\leq n\}]/\text{relations},$$
generated by the symbols $\Delta_{i,j}$ together with the relations
\begin{enumerate}
\item $\Delta_{i,j}+\Delta_{k,l}=\Delta_{i,k}+\Delta_{j,l}$ for distinct $i,j,k,l$ \hfill (square relations)
\item $\Delta_{i,j}\Delta_{i,k}=\Delta_{i,j}\Delta_{j,k}$ for distinct $i,j,k$. \hfill (diagonal relations)
\end{enumerate}
given in \Cref{bigthmintro} (1). If $n$ is clear from context or irrelevant, we will let $R:=R(n)$. If we let each $\Delta_{i,j}$ have degree 1, then the ideal of relations is homogenous, so $R$ is a graded ring, and we will denote by $R_k$ the $k$th graded part of $R$.
\end{defn}

By \Cref{PGL2injective}, we can identify $A^\bullet_{PGL_2}((\mb{P}^1)^n)$ as a subring of $A^\bullet_{GL_2}((\mb{P}^1)^n)$, where the image 
$$A^\bullet_{PGL_2}((\mb{P}^1)^n) \hookrightarrow A^\bullet_{GL_2}((\mb{P}^1)^n)=\mb{Z}[u,v]^{S_2}[H_1,\ldots,H_n]/(F(H_1),\ldots,F(H_n))$$
is generated by $\Delta_{i,j}=H_i+H_j+u+v$ for $n\geq 3$. If $n\leq 2$, we also have to add the classes $\psi_i = -(2H_i+u+v)$ (see \Cref{psidelta}). Therefore for $n \ge 3$ by \Cref{P14}, we have a surjective map 
\begin{equation}
R \twoheadrightarrow A^\bullet_{PGL_2}((\mb{P}^1)^n), \label{Rsurject}
\end{equation}
sending each symbol $\Delta_{i,j}\in R$ to $\Delta_{i,j}\in A^\bullet_{PGL_2}((\mb{P}^1)^n)$. To show \Cref{bigthmintro} (1), we need to show this surjection is an isomorphism for $n\ge 3$. 

As $A^\bullet_{GL_2}((\mb{P}^1)^n)$ is free as an abelian group, $A^k_{PGL_2}((\mb{P}^1)^n)$ is a free abelian group for each $k$. We first compute the rank of these groups for varying $k$.

\begin{lem}
\label{bettinumber}
For every $n\ge 1$, the free abelian group $A^k_{PGL_2}((\mb{P}^1)^n)$ has rank
$$\sum_{\substack{i \le k\\i \equiv k\text{ mod 2}}}\binom{n}{i}.$$
\end{lem}
\begin{proof}
We compute the rank of $A^k_{PGL_2}((\mb{P}^1)^n)$ by working instead with the rational subring 
$$A^\bullet_{PGL_2}((\mb{P}^1)^n) \otimes \mb{Q}\subset A^\bullet_{GL_2}((\mb{P}^1)^n))\otimes \mb{Q},$$
which is generated by the elements $H_i':=H_i +\frac{1}{2}(u+v)$ by \Cref{PGL2injective}.
Noting that $H_i'^2=\frac{1}{4}(u-v)^2$, we see the $\mb{Q}$-vector space $A^k_{PGL_2}((\mb{P}^1)^n)\otimes \mb{Q}$ is spanned by the elements
\begin{align*}
\mc{B}=\{\left(\frac{u-v}{2}\right)^{n-d-|B|}\prod_{i \in B} H_i'\mid B\subset [n],\ |B|\leq n-d,\ |B|\equiv n-d\pmod{2}\},
\end{align*}
which has size $$|\mc{B}|=\displaystyle\sum_{\substack{i \le k\\i \equiv k\text{ mod 2}}}\binom{n}{i}.$$
To finish, it suffices to show that the elements of $\mc{B}$ are linearly independent. Indeed, the elements of $\mc{B}$ become distinct monomials in the $H_i'$ after setting $u=1$ and $v=-1$ (after which the defining relations $F(H_i)=0$ become $H_i^{\prime 2}=1$ for each $i$). 
\end{proof}
\begin{defn}
Let $\on{Part}(d,n)$ denote the set of partitions of $[n]$ into $d$ parts. For $P \in \on{Part}(d,n)$, for any forest $T$ with vertex set $[n]$ consisting of one spanning tree for each part of $P$, we define $$\Delta_P=\prod_{\{i,j\}\in \text{Edge}(T)}\Delta_{i,j}\in R$$
Note that by the diagonal relations this is independent of the choice of $T$, and $\Delta_P \mapsto \Delta_P$ under the map $R \to A^\bullet_{PGL_2}((\mb{P}^1)^n)$ by \Cref{trivialfacts}.
\end{defn}
\begin{rmk}
\label{trivialremark}
The two items (1), (2) in \Cref{trivialfacts} are also true for the elements $\Delta_P \in R$ as the proof only uses the diagonal relations in $A^{\bullet}_{PGL_2}((\mb{P}^1)^n)$.
\end{rmk}
%We now show that for $k \le n-2$, $R_k$ is generated by the elements $\Delta_P$.
\begin{lem}
\label{Pgenerate}
For $k \le n-2$, $R_k$ is generated by $\{\Delta_{P}\mid P\in \on{Part}(n-k,n)\}$.
\end{lem}
\begin{proof}
Given a product $\prod_{\ell=1}^k \Delta_{i_\ell,j_\ell}$, we will produce an algorithm for rewriting this product in terms of $\Delta_P$ with $P$ a partition of $[n]$ into $n-k$ parts.

By induction, we can write $\prod_{\ell=1}^{k-1} \Delta_{i_\ell,j_\ell}$ as $\sum_{P'\in \on{Part}(n-k+1,n)} a_{P'}\Delta_{P'}$, so it suffices to show that $\Delta_{i_k,j_k}\Delta_{P'}$ for $P'\in \on{Part}(n-k+1,n)$ can be written as a $\mb{Z}$-linear combination $\sum_{P\in\on{Part}(n-k,n)}{a_P \Delta_P}$.

If $i_k,j_k$ are in different parts of $P'$, then $\Delta_{i_k,j_k}\Delta_P'=\Delta_{P}$ where $P$ merges the parts containing $i_k$ and $j_k$, and we are done. Otherwise, if they are in the same part $A_1$, let $A_2,A_3$ be two parts of $P'$ distinct from $A_1$ (which exist as $n-k+1 \ge 3$), with elements $x_2 \in A_2$ and $x_3 \in A_3$. By applying a square relation, we have
$$\Delta_{i_k,j_k}\Delta_{P'}=(\Delta_{i_k,x_2}-\Delta_{x_2,x_3}+\Delta_{x_3,j_k})\Delta_{P'},$$
and each of the three terms on the right is some $\Delta_P$ with $P\in \on{Part}(n-k,n)$.
\end{proof}
%Hence, for $k \le n-2$ we have a surjective map of abelian groups $$\mb{Z}[\Delta_P]_{P \in \on{Part}(n-k,n)}\twoheadrightarrow R^k.$$

\begin{defn}
Given a partition $P$ of $[n]$ and $i,j\in [n]$ in distinct parts of $P$, let $P_{i,j}$ be the partition of $[n]$ obtained by merging the parts in $P$ containing $i$ and $j$.
\end{defn}
From \Cref{trivialremark}, the following relations hold in $R(n)$ (and hence also in $A^\bullet_{PGL_2}((\mb{P}^1)^n)$).
\begin{defn}
\label{squarerelation}
For $i_1,i_2,i_3,i_4$ in distinct parts of a partition $P$ of $[n]$, define the \emph{square relation} for $P$ associated to $i_1,i_2,i_3,i_4$ to be the relation
\begin{align*}
\Delta_{P_{i_1,i_2}}-\Delta_{P_{i_2,i_3}}+\Delta_{P_{i_3,i_4}}-\Delta_{P_{i_4,i_1}}=0.
\end{align*}
\end{defn}

\begin{defn}
Inside the free abelian group $\mb{Z}^{\on{Part}(d,n)}$, denote by $\on{Sq}(d,n)$ the subgroup generated by formal square relations $P_{i,j}-P_{j,k}+P_{k,l}-P_{l,i}$ for $P \in \on{Part}(d+1,n)$ and $i,j,k,l$ distinct. Then we define $$\mathcal{A}(d,n):=\mb{Z}^{\on{Part}(d,n)}/\on{Sq}(d,n).$$
\end{defn}
\Cref{Pgenerate} shows for $d \ge 2$ we have a surjection
$$\mathcal{A}(d,n)\twoheadrightarrow R_{n-d}$$ 
that sends $P\mapsto \Delta_P$. We will in fact show this is an isomorphism. 

\begin{defn}
Say a partition $P \in \on{Part}(d,n)$ for $d \ge 2$ is \emph{good} if $P$ can be written as $P=\{A_1,\ldots,A_n\}$ with $A_1 \sqcup A_2$ a partition of an initial segment of $[n]$, and $A_3,\ldots,A_n$ all contiguous intervals. Denote
$$\on{Good}(d,n):=\{P \in \on{Part}(d,n) \mid P\text{ good}\}.$$ 
\end{defn}

\begin{lem}
\label{generatePlow}
For $d \le n-2$, $\mathcal{A}(d,n)$ is generated by the set of $P\in \on{Good}(d,n)$.
\end{lem}
\begin{proof}
We use induction on $n$ and $d$. For $d=2$ every partition is good, and for $n=2$ the result is trivial. Suppose now we have $n,d>2$. Take $Q \in \on{Part}(d,n)$.

If $n-1$ and $n$ are in the same part, then $Q':=Q \setminus n \in \on{Part}(d,n-1)$, and by the induction hypothesis applied to $\mathcal{A}(d,n-1)$ we can write $Q'=\sum_{P'\in \on{Good}(d,n-1)} a_{P'}P'$. There is a map
$$\mathcal{A}(d,n-1)\to \mathcal{A}(d,n)$$ mapping each $P'$ for $P'\in \on{Part}(d,n-1)$ to $P$, where $P$ is obtained by adding $n$ to the same part as $n-1$ in $P'$. Furthermore, under this map $P\in \on{Good}(d,n)$ if $P'\in \on{Good}(d,n-1)$, so we get $Q$ as a $\mb{Z}$-linear combination of $P$ for $P\in \on{Good}(d,n)$.

If $n$ is isolated in $Q$, then let $Q'=Q \setminus n \in \on{Part}(d-1,n-1)$. By the induction hypothesis applied to $\mc{A}(d-1,n-1)$, we can write $Q=\sum_{P'\in \on{Good}(d-1,n-1)} a_{P'}P'$. There is a map
$$\mc{A}(d-1,n-1)\to \mc{A}(d,n)$$ mapping each $P'$ for $P'\in \on{Part}(d-1,n-1)$ to $P$, where $P$ is obtained by adding $n$ as an isolated part. Furthermore, under this map $P\in \on{Good}(d,n)$ if $P'\in \on{Good}(d-1,n-1)$, so we get $Q$ as a $\mb{Z}$-linear combination of $P$ for $P\in \on{Good}(d,n)$.

If neither of the above two cases hold, then $n-1$ and $n$ are not in the same part and $n$ is not isolated in $Q$. Let $x\in [n]$ be another element in the same part as $n$, and let $y\in [n]$ be in a different part as $n-1$ and $n$ (which exists as $d>2$). Then if we let $\wt{Q}\in \on{Part}(d+1,n)$ be the result of taking $Q$ and isolating $n$ into its own part, the square relation for $\wt{Q}$ associated to $n-1,n,x,y$ yields $Q$ as a combination of $3$ terms, each of which either has $n$ isolated or $n-1,n$ in the same group.
\end{proof}

\begin{lem}
\label{howgood}
For $d\ge 2$, 
\begin{align*}
\# \on{Good}(d,n) &= \sum_{\substack{i \le n-d\\i\equiv n-d\text{ mod 2}}} \binom{n}{i}.
\end{align*}
%the number of good partitions into $d$ parts is given by $$\sum_{\substack{i \le n-d\\i\equiv n-d\text{ mod 2}}} \binom{n}{i}.$$ %Equivalently, $\mc{A}_{[n],n-2}\to A_{T}^{n-2}((\mb{P}^1)^n) \otimes \mathbb{Q}$ is an isomorphism onto the rational span of the $[P]$ classes over all partitions of $[n]$ into $d$ parts.
\end{lem}

\begin{proof}
%\Cref{basisthm} says that the span of the $\Delta_P$ over all $P\in \on{Part}(d,[n])$ is a free $\mb{Z}$-module with $\mb{Z}$-basis given by $\{\Delta_P\mid P \in \mc{G}(d,n)\}$. 
From the definition of $\on{Good}(d,n)$,
\begin{align*}
%|\mc{G}(d,n)| = \sum_{k=1}^{n-d+2}{(2^{k-1}-1)\binom{n-k-1}{d-3}}. 
\# \on{Good}(d,n) = \sum_{k=1}^{n-d+2}{(2^{k-1}-1)\binom{n-k-1}{n-k-d+2}}. 
\end{align*}
Let 
\begin{align*}
G_{d,n}&=\sum_{k=1}^{n-d+2}{(2^{k-1}-1)\binom{n-k-1}{n-k-d+2}}\\
G'_{d,n} &= \sum_{\substack{i \le n-d\\i\equiv n-d\text{ mod 2}}} \binom{n}{i}. 
\end{align*}

We will show $G_{d,n}=G'_{d,n}$ for all $n\geq 2$ and $d\geq 2$ by induction on $n$. For the base case if $n=2$ and $d\geq 2$ arbitrary, we have two cases: if $d=2$, $|\mc{G}(2,2)|=G_{2,2}=1$ and if $d>2$, $|\mc{G}(d,2)|=G_{d,2}=0$. If $d=2$ and $n\geq 2$ arbitrary, then $G'_{d,n}=2^{n-1}-1$ by the binomial theorem, and $G_{d,n}=2^{n-1}-1$ because only the $k=n$ term $(2^{n-1}-1)\binom{-1}{0}$ contributes. 

Now, assume we know $G_{d,n}=G'_{d,n}$ for some $n$ and all $d\geq 2$. For the induction step, 
\begin{align*}
G_{d,n}+G_{d+1,n} &= \sum_{k=1}^{n-d+2}{(2^{k-1}-1)\left(\binom{n-k-1}{n-k-d+2}+\binom{n-k-1}{n-k-d+1}\right)}\\
%&= \sum_{k=1}^{n-d+2}{(2^{k-1}-1)\binom{n-k}{n-k-d+2}}\\
&= \sum_{k=1}^{n-d+2}{(2^{k-1}-1)\binom{n-k}{n-k-d+2}}= G_{d+1,n+1},
\end{align*}
and similarly applying Pascal's identity, $G_{d,n}'+G_{d+1,n}'=G'_{d+1,n+1}$.
\end{proof}

\begin{cor}
\label{3equal}
For $d\geq 2$, and $n \ge 3$ we have the isomorphisms
\begin{align*}
\mb{Z}^{\on{Good}(d,n)}\xrightarrow{\sim} \mc{A}(d,n)\xrightarrow{\sim} R_{n-d}\xrightarrow{\sim} A^{n-d}_{PGL_2}((\mb{P}^1)^d). 
\end{align*}
\end{cor}

\begin{proof}
By \Cref{Pgenerate,generatePlow} and \eqref{Rsurject}, we have
$$\mb{Z}^{\on{Good}(d,n)}\twoheadrightarrow \mc{A}(d,n) \twoheadrightarrow  R^{n-d} \twoheadrightarrow A^{n-d}_{PGL_2}((\mb{P}^1)^d).$$
Since $A^{n-d}_{PGL_2}((\mb{P}^1)^d)$ is a finitely generated, free $\mb{Z}$-module of rank equal to the rank of $\mb{Z}^{\on{Good}(d,n)}$ by \Cref{bettinumber,howgood}, the composite $\mb{Z}^{\on{Good}(d,n)}\to A^{n-d}_{PGL_2}((\mb{P}^1)^d)$ is an isomorphism. 
\end{proof}

%Hence, we have for $d \ge 2$, $$\mb{Z}[\Delta_P]_{P \in \on{Part}(d,n)}/\text{relations} \twoheadrightarrow  R^{n-d} \twoheadrightarrow A^{n-d}_{PGL_2}((\mb{P}^1)^d)$$ and as we have a generating set for the first group with the same cardinality as the rank of the last free ableian group, these are in fact all isomorphic.

We now find an explicit basis for $R_k$ for $k>n-2$ of size $2^{n-1}$. 
%Now consider $R^k$ for $k>n-2$. We will show that $R^k$ is generated by $2^{n-1}$ elements of a certain form.

\begin{lem}
\label{generatePhigh}
For each partition $P\in \on{Part}(d,n)$ for $d\leq 2$, arbitrarily choose $i_P,j_P$ that lie in the same part. Then for $k>n-2$, $R_k$ is generated by the $2^{n-1}$ elements
\begin{align*}
S_k:=\{\Delta_{\{[n]\}}\Delta_{i_{\{[n]\}},j_{\{[n]\}}}^{k-n+1}\}\cup \{\Delta_{P}\Delta_{i_P,j_P}^{k-n+2}\mid P\in\on{Part}(2,n)\}.
\end{align*}
%$\Delta_{\{[n]\}}\Delta_{i_{\{[n]\}},j_{\{[n]\}}}^{n-k+1}$ and for $P$ a partition of $[n]$ into two parts, $\Delta_{P}\Delta_{i_P,j_P}^{n-k+2}$, where for a partition $Q$, $i_Q,j_Q$ are fixed elements of $[n]$ chosen to lie in the same part of $Q$.
\end{lem}
\begin{proof}
Let $P=\{A,B\}\in \on{Part}(2,n)$. By \Cref{Pgenerate}, it suffices to show $\Delta_P \prod_{a=1}^{k-n+2}\Delta_{i_a,j_a}$ is generated by $S_k$ for any choices of $i_a\ne j_a$. We proceed by induction on $k>n-2$. For the base case $k=n-1$, it suffices to show $\Delta_{i,j}\Delta_{P}$ is generated by $S_k$ for any $i\neq j$. If $k>n-1$, then by the induction hypothesis, it suffices to show $\Delta_{i,j}\Delta_{P}\Delta_{i_P,j_P}^{k-n+1}$ and $\Delta_{i,j}\Delta_{\{[n]\}}\Delta_{i_{\{[n]\}},j_{\{[n]\}}}^{k-n}$ are generated by $S_k$. Both the base case and the induction step will work in the same way.

First, $\Delta_{i,j}\Delta_{\{[n]\}}\Delta_{i_{\{[n]\}},j_{\{[n]\}}}^{k-n}=\Delta_{\{[n]\}}\Delta_{i_{\{[n]\}},j_{\{[n]\}}}^{k-n+1}$ by \Cref{trivialremark} (2). To deal with $\Delta_{i,j}\Delta_{P}\Delta_{i_P,j_P}^{k-n+1}$, we have two cases. 

\begin{enumerate}
\item
If $\{i,j\}$ is not contained in $A$ or $B$, then $\Delta_{i,j}\Delta_P$ is the diagonal $\Delta_{\{[n]\}}$ by \Cref{trivialremark} (1). Then, by \Cref{trivialremark} (2), $\Delta_{i,j}\Delta_{P}\Delta_{i_P,j_P}^{k-n+1}=\Delta_{i_{\{[n]\}},j_{\{[n]\}}}^{k-n+1}\Delta_{\{[n]\}}$. 
\item
Suppose now each $\{i,j\}$ is in $A$ or $B$, and that without loss of generality, $i_P,j_P \in A$. If $i,j \in A$, then using \Cref{trivialremark} (2) we may replace $\Delta_{i,j}$ with $\Delta_{i_P,j_P}$. If $i,j \in B$, we can use a square relation to replace it with $\Delta_{i,i_P}-\Delta_{i_P,j_P}+\Delta_{j_P,i}$. We then have $\Delta_{i,i_P}\Delta_P=\Delta_{\{[n]\}}=\Delta_{j_P,i}\Delta_P$, so
\begin{align*}
\Delta_{i,j}\Delta_{P}\Delta_{i_P,j_P}^{k-n+1}= 2\Delta_{i_{\{[n]\}},j_{\{[n]\}}}^{k-n+1}\Delta_{\{[n]\}}-\Delta_{P}\Delta_{i_P,j_P}^{k-n+2}
\end{align*}
by \Cref{trivialremark} (2).
%Then as in (1), multiplying the remaining $\Delta_{i,j}$'s with either of the $\Delta_{\{[n]\}}$ yields $\Delta_{i_{\{[n]\}},j_{\{[n]\}}}^{k-n+1}\Delta_{\{[n]\}}$ by repeated applications of the diagonal relation again. For the term involving $\Delta_{i_P,j_P}$, we can apply induction on $k$.
\end{enumerate}
\end{proof}

%The lemma provides us with a generating set of $R^k$ of the same cardinality as the rank of the free abelian group $A^k_{PGL_2}((\mb{P}^1)^n)$ onto which it surjects. Hence $R^k \cong A^k_{PGL_2}((\mb{P}^1)^n)$ for $k>n-2$. Putting this together with the previous result yields
%$$R \cong A^\bullet_{PGL_2}((\mb{P}^1)^n).$$

\begin{thm}
\label{mainthm}
For $n \ge 3$, the natural surjection $R\twoheadrightarrow A^\bullet_{PGL_2}((\mb{P}^1)^n)$ is an isomorphism. Furthermore, $R_k$ has $\mb{Z}$-basis given by
\begin{enumerate}
\item
$\{\Delta_P\mid P\in \on{Good}(n-k,n)\}$ for $k\leq n-2$
\item 
$S_k=\{\Delta_{\{[n]\}}\Delta_{i_{\{[n]\}},j_{\{[n]\}}}^{k-n+1}\}\cup \{\Delta_{P}\Delta_{i_P,j_P}^{k-n+2}\mid P\in\on{Part}(2,n)\}$,
where for each partition $P\in \on{Part}(d,n)$ for $d\leq 2$, arbitrarily choose $i_P,j_P$ that lie in the same part.
\end{enumerate}
\end{thm}

\begin{proof}
If $k\leq n-2$, we have $R_k\twoheadrightarrow A^k_{PGL_2}((\mb{P}^1)^n)$ is an isomorphism with $\mb{Z}$-basis given by $\{\Delta_P\mid P\in \on{Good}(n-k,n)\}$ by \Cref{3equal}. Now, we consider the case $k>n-2$. 

The $S_k$ span $R_k$ by \Cref{generatePhigh}, so applying \eqref{Rsurject} yields 
\begin{align*}
\mb{Z}^{S_k}\twoheadrightarrow R_k\twoheadrightarrow {A}^{k}_{PGL_2}((\mb{P}^1)^n),
\end{align*}
whose composite is a surjection of free $\mb{Z}$-modules of the same rank $2^{n-1}$ by \Cref{generatePhigh,bettinumber}, so it is an isomorphism. This proves $R_k\twoheadrightarrow A^k_{PGL_2}((\mb{P}^1)^n)$  is an isomorphism and identifies $S_k$ as a basis. 
\end{proof}

\subsection{Algorithm and Example}
\label{algandexmp}
We can describe an algorithm for writing arbitrary classes in $A^\bullet_{PGL_2}((\mb{P}^1)^d)$ in terms of our $\mb{Z}$-basis. The key fact is that if $\on{pr}^n:(\mb{P}^1)^n \to (\mb{P}^1)^{n-1}$ is projection by forgetting the last factor, then by definition of the pushforward of a cycle
$$\on{pr}^n_*\Delta_P=\begin{cases}\Delta_{P \setminus n}&\text{if $n$ is not isolated, and}\\0&\text{if $n$ is isolated.}\end{cases}$$
%Indeed, if $n$ is isolated then $\Delta_P=\pi^*\Delta_{P \setminus n}$ and if $n$ is not isolated then we may write $\Delta_P=\Delta_{n,k}\Delta_{P'}$ with $P'$ a partition of $[n]\setminus \{k\}$ for some $k$. 
At the level of formulae, if we write our class as a polynomial in the $H_i,u,v$ with each $H_i$ appearing to degree at most $1$, then $\on{pr}^n_*$ extracts the $H_n$-coefficient. Also, if we have a $\Delta_P$ and we know that either $n$ is isolated or $n-1,n$ are in the same part, then as $(H_n-H_{n-1})\cap \Delta_{n-1,n}=0$ we also have
$$\on{pr}^n_*(\Delta_P\cap(H_n-H_{n-1}))=\begin{cases}0&\text{if $n-1,n$ are in the same group, and}\\\Delta_{P\setminus n}&\text{if $n$ is isolated.}\end{cases}$$

Suppose we have a class
$$\alpha=\sum_{P \in \on{Good}(d,n)}a_P\Delta_P=\sum_{\substack{P \in \on{Good}(d,n)\\n\text{ isolated}}}a_P\Delta_P+\sum_{\substack{P \in \on{Good}(d,n)\\n-1,n\text{ together}}}a_P\Delta_P$$ and we want to find the coefficients $a_P$.

We first show how to reduce down to the case $d=2$. By the above, we have
$$\on{pr}^n_*\alpha=\sum_{\substack{P \in \on{Good}(d,n)\\n-1,n\text{ together}}}a_P\Delta_{P\setminus n}, \qquad \on{pr}^n_*(\alpha\cap(H_{n}-H_{n-1}))=\sum_{\substack{P \in \on{Good}(d,n)\\n\text{ isolated}}}a_P\Delta_{P\setminus n}.$$
In the first case each $P \setminus n\in \on{Good}(d-1,n)$, and in the second case each $P\setminus n \in \on{Good}(d-1,n-1)$ so we can apply induction to determine all of these coefficients.

Once we have reduced down to the case $d=2$, we can now identify each $a_P$ separately for $P=\{A,B\}$ a partition of $[n]$ into two parts by evaluating at $H_i=-u$ for $i \in A$ and $H_i=-v$ for $i \in B$ (which is localization at a torus-fixed point). By \Cref{FormulaP}, this evaluates to $a_{\{A,B\}}(u-v)^{n-2}(-1)^{|A|-1}$.

The same method for $d=2$ works for elements $\alpha \in A^k((\mb{P}^1)^n)$ with $k>n-2$. Applying the same substitution to $$\alpha=\sum a_P\Delta_{i_P,j_P}^{k-n+2}\Delta_P+a_{\{[n]\}}\Delta_{i_{\{[n]\}},j_{\{[n]\}}}^{n-k+1}\Delta_{\{[n]\}}$$
extracts the $a_P$-coefficient for $P=\{A,B\}$ a partition of $[n]$ into two parts as this is the only term that does not vanish under this substitution. Then, we subtract off all of these terms to recover $a_{[n]}$.

\begin{exmp}
As a simple example, consider the $PGL_2$-orbit closure of a generic point in $(\mb{P}^1)^5$. The formula computed in \cite[Corollary 4.8]{FML} shows that the class of this orbit is
$$\alpha=e_2(H_1,H_2,H_3,H_4,H_5)+2(u+v)(H_1+H_2+H_3+H_4+H_5)+(3u^2+4uv+3v^2),$$
where $e_2$ is the second elementary symmetric polynomial. We have
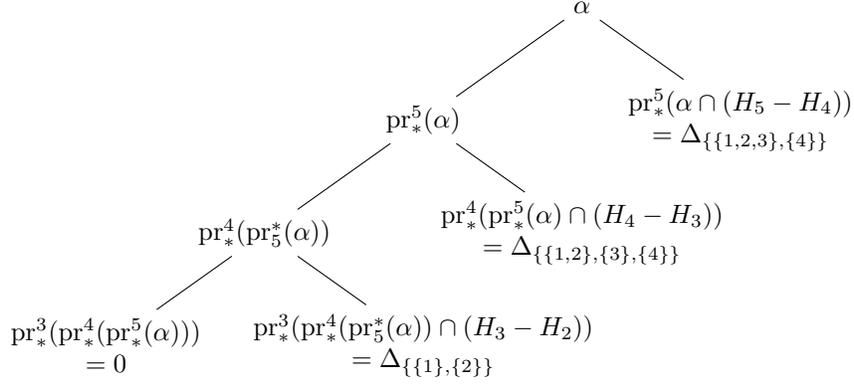
\begin{center}
\begin{tikzpicture}[sibling distance=12em, every node/.style = {align=center}]
  \node {$\alpha$}
    child
    {
    	node {$\on{pr}_*^5(\alpha)$}
        child
        {
        	node {$\on{pr}_*^4(\on{pr}_5^*(\alpha))$}
            child
            {
            	node{$\on{pr}_*^3(\on{pr}_*^4(\on{pr}^5_*(\alpha)))$\\$=0$}
            }
            child
            {
            	node{$\on{pr}_*^3(\on{pr}_*^4(\on{pr}^*_5(\alpha))\cap (H_3-H_2))$\\$=\Delta_{\{\{1\},\{2\}\}}$}
            }
        }
        child
        {
        	node {$\on{pr}^4_*(\on{pr}^5_*(\alpha)\cap (H_4-H_3))$\\$=\Delta_{\{\{1,2\},\{3\},\{4\}\}}$}
        }
    }
    child { node {$\on{pr}^5_*(\alpha \cap (H_5-H_4))$\\ $=\Delta_{\{\{1,2,3\},\{4\}\}}$}};
\end{tikzpicture}
\end{center}

\begin{align*}
&\on{pr}^5_* \alpha = (H_1+H_2+H_3+H_4)+2(u+v)\\
&\on{pr}^5_*(\alpha \cap (H_5-H_4))=e_2(H_1,H_2,H_3)+(u+v)(H_1+H_2+H_3)+(u^2+uv+v^2)\\
&\on{pr}^4_*(\on{pr}^5_*\alpha)=1\\
&\on{pr}^4_*(\on{pr}^5_*\alpha \cap (H_4-H_3))=H_1+H_2+u+v\\
&\on{pr}_*^3(\on{pr}_*^4(\on{pr}^5_*\alpha))=0\\
&\on{pr}_*^3(\on{pr}_*^4(\on{pr}^5_*\alpha)\cap (H_3-H_2))=1.
\end{align*}
The only non-trivial identification was  $\on{pr}^5_*(\alpha \cap (H_5-H_4))=\Delta_{\{\{1,2,3\},\{4\}\}}$, which we can identify as follows. Substitute $-u$'s and $-v$'s for the $H_i$ corresponding to all nontrivial partitions $\{A,B\}$ of $[4]$ into two parts. We find the only choice that gives a nonzero result is $A=\{1,2,3\},B=\{4\}$, yielding $(u-v)^2$, which is the same as for $\Delta_{\{\{1,2,3\},\{4\}\}}$ by \Cref{FormulaP}.
Putting this together yields
\begin{align*}
\alpha &= \Delta_{\{\{1\},\{2\},\{3,4,5\}\}}+\Delta_{\{\{1,2\},\{3\},\{4,5\}\}}+\Delta_{\{\{1,2,3\},\{4\},\{5\}\}}.
\end{align*}
\begin{comment}
We have
\begin{align*}
\on{pr}^5_*\alpha &=(H_1+H_2+H_3+H_4)+2(u+v).\\
\on{pr}^4_*(\on{pr}^5_*\alpha)&=1.\\
\on{pr}_*^3(\on{pr}_*^4(\on{pr}^5_*\alpha))&=0.\\
\on{pr}_*^3(\on{pr}_*^4(\on{pr}^5_*\alpha)\cap (H_3-H_2))=1&=\Delta_{\{\{1\},\{2\}\}}\\
\on{pr}^4_*(\on{pr}^5_*\alpha \cap (H_4-H_3))=H_1+H_2+u+v&=\Delta_{\{\{1,2\},\{3\}\}}\\
\on{pr}^5_*(\alpha \cap (H_5-H_4))=H_1H_2+H_1H_3+H_2H_3&+(u+v)(H_1+H_2+H_3)+(u^2+uv+v^2)\\
&=\Delta_{\{\{1,2,3\}\},\{4\}\}}.
\end{align*}
Hence, we get the decomposition
$$\alpha=\Delta_{\{\{1\},\{2\},\{3,4,5\}\}}+\Delta_{\{\{1,2\},\{3\},\{4,5\}\}}+\Delta_{\{\{1,2,3\},\{4\},\{5\}\}}.$$
\end{comment}
We remark that the $PGL_2$-orbit closure $X_n\subset(\mb{P}^1)^n$ of a general point in $(\mb{P}^1)^n$ decomposes into good incidence strata as
\begin{equation}
[X_n]=\sum_{a=1}^{n-2} \Delta_{\{\{1,\ldots,a\},a+1,\{a+2,\ldots,n\}\}} \label{orbit}
\end{equation}
which can be geometrically explained as follows. Consider the diagram
\begin{center}
\begin{tikzcd}
\overline{\mc{M}}_{0,n}(\mb{P}^1,1) \ar[r,"\text{ev}"]\ar[d,"\pi"] & (\mb{P}^1)^n\\
\overline{\mc{M}}_{0,n}
\end{tikzcd}
\end{center}
(see \Cref{WDVVintro} for notation). The left and right hand side of \eqref{orbit} can both be described as $\text{ev}_*\pi^*(\pt)$ for $\pt \in \overline{\mc{M}}_{0,n}$ being a general point and the point corresponding to a chain of $n-2$ rational curves (respectively), and the result follows from the flatness of $\pi$. See \cite[Section 4]{FML} for a generalization of this degeneration to $PGL_{r+1}$ orbits closures of general points in $(\mb{P}^r)^n$.
\end{exmp}
%\begin{exmp}
%Take the $PGL_2$-orbit $X$ of a general point in $(\mb{P}^1)^n$. There is an explicit series of $PGL_2$-invariant degenerations of $X$ into the sum $$[X]=\sum_{a=1}^{n-2} \Delta_{\{1,\ldots,a\},a+1,\{a+2,\ldots,n\}},$$ and each of these $\Delta_P$'s has $P$ good.
%\end{exmp}

\section{$GL_2$-equivariant classes of strata in ${\rm Sym}^n\mb{P}^1$}
Recall from \Cref{bracket} that $[\lambda]\in A_{GL_2}^{\bullet}(\mb{P}^n)$ for $\lambda$ a partition of $n$ is the pushforward of $\Delta_P\in  A_{GL_2}^{\bullet}((\mb{P}^1)^n)$ under the multiplication map $(\mb{P}^1)^n\to \mb{P}^n$ for $P$ a partition of $[n]$ into subsets with cardinalities given by $\lambda$. Up to a constant factor given in \Cref{bracket}, this is the class of the closure $Z_{\lambda}$ given in \Cref{unorderedloci} of degree $n$ forms on $(\mb{P}^{1})^{\vee}$ whose roots have multiplicities given by $\lambda$  as studied by Feh\'er, N\'emethi, and Rim\'anyi \cite{FNR06}. 

\begin{defn}
Denote by $[a,b,1^c]:=[\{a,b,1,1,\ldots,1\}]$ where there are $c$ $1$'s.
\end{defn}

From writing the expressions for $[\lambda]$ in \Cref{FNRformula} using generating functions, we find the following new Corollary. 

\begin{cor}
\label{ab1formula}
For $d\geq 2$, consider the polynomial
\begin{align*}
-\frac{1}{(z-1)^{d-2}}\prod_{i=1}^{d} (z^{a_i}-1)&=\sum_{\substack{0 \le k_1\leq k_2\\k_1+k_2=n-d+2}}{\alpha_{k_1}(z^{k_1}+z^{k_2})}. 
\end{align*}
Then $\alpha_i \in \mb{Z}$ and
\begin{align*}
[a_1,\ldots,a_d]=\sum_{\substack{1 \le k_1\leq k_2\\k_1+k_2=n-d+2}}\alpha_{k_1}[k_1,k_2,1^{d-2}]
\end{align*}
\end{cor}

\begin{proof}
Clearly all $\alpha_i \in \mb{Z}$ except possibly $\alpha_{\frac{n-d+2}{2}}$, which a priori only lies in $\mb{Z}[\frac{1}{2}]$. But plugging in $z=1$ to both sides shows the integrality.

By \Cref{FNRformula}, it suffices to show
\begin{align*}
\prod_{i=1}^{d}{(z^{a_i}-1)}=\sum_{\substack{1 \le k_1\leq k_2\\k_1+k_2=n-d+2}}\alpha_{k_1}(z^{k_1}-1)(z^{k_2}-1)(z-1)^{d-2}.
\end{align*}
or equivalently
\begin{align*}
\frac{1}{(z-1)^{d-2}}\prod_{i=1}^{d}{(z^{a_i}-1)}=\sum_{\substack{k_1\leq k_2\\k_1+k_2=n-d+2}}\alpha_{k_1}(z^{k_1}-1)(z^{k_2}-1).
\end{align*}
%Indeed, dividing both sides by $(z-1)^{d-2}$, 
By definition of $\alpha_k$, the coefficients of both sides agree except possibly the $z^0$ and $z^{n-(d-2)}$-coefficient. Also, the coefficients of $z^0$ and $z^{n-(d-2)}$ are equal to each other on the left hand side, and the same is true on the right side. To see they agree between the left and right sides, we note both sides are $0$ after substituting $z=1$.
\end{proof}

\begin{lem}
\label{linindeplem}
The rational $GL_2$-equivariant classes in $\mb{P}^n$ of the torus fixed points $$\prod_{j \in [n]\setminus \{k\}} (H+jv+(n-j)u)\in A^\bullet_{T}(\mb{P}^n)\otimes \mb{Q}$$ are linearly independent.
\end{lem}
\begin{proof}
For fixed $k$, $H\mapsto -ku-(n-k)v$ maps $\prod_{j\in [n]\setminus\{k'\}}(H+jv+(n-j)u)$ to $0$ if and only if $k' \ne k$
\end{proof}

\begin{thm}
\label{basis}
For fixed $c\geq 0$, the classes $[a,b,1^c]$ with $a+b=n-c$ and $a \ge b$ form a $\mb{Q}$-basis for
$A^{n-c-2}_{PGL_2}(\mb{P}^n)\otimes \mb{Q}\subset A^{n-c-2}_{GL_2}(\mb{P}^n)\otimes \mb{Q}$.
%$all $[V]\in A_{GL_2}^{\bullet}(\mb{P}^n)$ for any $GL_2$-invariant subvariety $V\subset \mb{P}^n$ of dimension $c+2$ over $\mb{Q}$.
\end{thm}
\begin{proof}
%\Cref{AB1s} shows that all the unsymmetrized classes $[P]$ for $P$ a partition of $[n]$ are generated over $\mb{Z}$ by the classes for which $P$ contains at most two sets with size greater than 1. The statement about the classes $[a,b;c]$ generated all symmetrized classes $[\lambda]$ follow from pushing forward via the multiplication map $\Phi: (\mb{P}^1)^n\to \mb{P}^n$. 

To show the linear independence, first note that $\prod_{j\in [n]\setminus\{k\}}(H+jv+(n-j)u)$ are linearly independent in $A^\bullet_T(\mb{P}^n)\otimes \mb{Q}$ by \Cref{linindeplem}. Therefore, it suffices to show for fixed $c$ that the polynomials $(z^a-1)(z^b-1)(z-1)^c$ with $a \ge b$ and $a+b=n-c$ are linearly independent. Indeed, dividing out by $(z-1)^c$, we note that $(z^a-1)(z^b-1)$ is the only such polynomial which contains either of the monomials $z^a$ or $z^b$.

To see that the $\mb{Q}$-linear span of the classes $[a,b,1^c]$ is precisely $A^{n-c-2}_{PGL_2}(\mb{P}^n)\otimes \mb{Q}$, we note that we have just shown that the dimension of the $\mb{Q}$-linear span of the $[a,b,1^c]$ is precisely $\lfloor \frac{n-c}{2} \rfloor$ by linear independence, which we can check is the same as the dimension of $A^{n-c-2}_{PGL_2}(\mb{P}^n)\otimes \mb{Q}$ by \Cref{PGLunordered}. 

\end{proof}
%\begin{rmk}
%The only closed $GL_2$-invariant subvariety of $\mb{P}^n$ of dimension $1$ is $[\{n\}]$, so the $[\{a_1,\ldots,a_r\}]$ generate all $[V] \in A^\bullet_{GL_{2}}(\mb{P}^n)\otimes \mb{Q}$.
%\end{rmk}

\section{Integral classes of unordered strata in $[\text{\rm Sym}^n\mb{P}^1/PGL_2]$}
\label{integralunordered}
In this section, we compute the integral classes of $[Z_{\lambda}]\in A^\bullet_{PGL_2}(\mb{P}^n).$ By \Cref{PGLunordered}, if $n$ is odd, then $A_{PGL_2}^{\bullet}(\mb{P}^n)\to A_{GL_2}^{\bullet}(\mb{P}^n)$ is injective and we know the image of the $[Z_{\lambda}]$ in $A_{GL_2}^{\bullet}(\mb{P}^n)$ by \Cref{FNRformula}, so it suffices to consider the case $n$ is even, which we assume for the remainder of this section. 

Recall the polynomials $p_n(t)\in A^{\bullet}_{PGL_2}(\pt)[t]$ defined in \Cref{PGLunordered} for even $n$ and let $q_n$ be the image of $p_n$ in $A^{\bullet}_{PGL_2}(\pt)/(2)[t]\cong \mb{F}_2[c_2,c_3,t]$. It is easy to see by the binomial theorem or directly from \cite[Lemma 6.1]{FV11} that
$$q_n(t)= \begin{cases}t^{(n+4)/4}(t^3+c_2t+c_3)^{n/4}&\text{if $n\equiv 0$ mod $4$, and}\\t^{(n-2)/4}(t^3+c_2t+c_3)^{(n+2)/4}&\text{if $n \equiv 2$ mod $4$,} \end{cases}$$
and $q_n(t) \mid q_{n+k}(t)$ for $k=0$ or $k \ge 4$ for any $n$. 

By \Cref{PGLunordered}, for $n$ even,
\begin{align*}
A_{PGL_2}^{\bullet}(\mb{P}^n)\cong \mb{Z}[c_2,c_3,H]/(2c_3,p_n(H)),
\end{align*}
which is isomorphic to 
\begin{align*}
\left(\bigoplus_{i=0}^n\mb{Z}[c_2]H^i\right)\oplus \left(\bigoplus_{i=0}^nc_3\mb{F}_2[c_2,c_3]H^i\right)
\end{align*}
as abelian groups. So to determine the class $[Z_{\lambda}]\in A^\bullet_{PGL_2}(\mb{P}^n)$, it suffices to find its image in $\bigoplus_{i=0}^n\mb{Z}[c_2]H^i$ and $\bigoplus_{i=0}^nc_3\mb{F}_2[c_2,c_3]H^i$. Equivalently, if we write the class of $[Z_{\lambda}]$ as a polynomial in $c_2$, $c_3$, and $H$ with degree at most $n$ in $H$, then it suffices to consider the terms not containing $c_3$ and the terms containing $c_3$ separately. Under the map $A_{PGL_2}^{\bullet}(\mb{P}^n)\to A_{GL_2}^{\bullet}(\mb{P}^n)$, \Cref{PGLunordered} shows that the first factor maps injectively and the second factor maps to zero. 

We can determine the image of $[Z_{\lambda}]$ in the first factor using \Cref{FNRformula}, so it suffices to determine the image of $[Z_{\lambda}]$ in the second factor to identify its class. To do this, we will work modulo $2$ and determine $[Z_{\lambda}]\in A^\bullet_{PGL_2}(\mb{P}^n)\otimes \mb{Z}/2\mb{Z}$. Discarding those monomials not containing $c_3$ then yields the image of $[Z_\lambda]$ in the second factor.

\begin{defn}
We say a partition $\lambda=a_1^{e_1}\ldots a_k^{e_k}$ of $n$ into $d=\sum_{i=1}^k e_i$ parts is \emph{special} if all $a_i$ and $\frac{d!}{e_1!\cdots e_k!}$ are odd, and all $e_i$ are even. 
\end{defn}

\begin{thm}
\label{PGLmod2}
Let $d$ and $n$ be integers with $n$ even. The class of $[Z_\lambda]\in A^\bullet_{PGL_2}(\mb{P}^n)\otimes \mb{Z}/2\mb{Z}$ for $\lambda$ a partition of $n$ into $d$ parts is given by
$$\begin{cases}\frac{q_n}{q_{d}}(H)&\text{if $\lambda$ is special, and}
\\ 0&\text{otherwise.}\end{cases}$$
\end{thm}
\begin{rmk}
If $[Z_\lambda]\in A^\bullet_{PGL_2}(\mb{P}^n)\otimes \mb{Z}/2\mb{Z}$ is zero, then the component in $\bigoplus_{i=0}^n\mb{Z}[c_2]H^i$ is a multiple of $2$, and the component in $\bigoplus_{i=0}^nc_3\mb{F}_2[c_2,c_3]H^i$ is zero.

Furthermore, given the statement of the theorem, if $[Z_\lambda]\in A^\bullet_{PGL_2}(\mb{P}^n)\otimes \mb{Z}/2\mb{Z}$ is non-zero, then the component in $\bigoplus_{i=0}^nc_3\mb{F}_2[c_2,c_3]H^i$ is non-zero and is given by discarding anything with a $c_3^0$-coefficient in $\frac{q_n}{q_{d}}(H)$.
\end{rmk}

\begin{lem}
\label{diagonalclass}
Given a ring $R[H]/(P(H))$ for $P$ a monic polynomial of degree $n+1$, define the $R$-linear map $\int:R[H]/(P(H)) \to R$ given by taking a polynomial $f(H)$, and outputting the $H^{n}$-coefficient of the reduction $\wt{f}(H)$ of $f(H)\pmod{P(H)}$ to a polynomial of degree $\le n$. Then letting $t$ be an indeterminate, we have
$$\int \frac{P(H)-P(t)}{H-t} f(H)=\wt{f}(t).$$
\end{lem}
\begin{proof}
We have
%\begin{align*}
%\int \frac{P(H)-P(t)}{H-t} (f(H)-\wt{f}(t))&= \int (P(H)-P(t))\frac{f(H)-\wt{f}(t)}{H-t}\\
%\int \frac{P(H)-P(t)}{H-t} (\wt{f}(H)-\wt{f}(t))&= \int (P(H)-P(t))\frac{\wt{f}(H)-\wt{f}(t)}{H-t}\\
%\int \frac{P(H)-P(t)}{H-t} \wt{f}(H)-\int \frac{P(H)-P(t)}{H-t}\wt{f}(t)&= \int (-P(t))\frac{\wt{f}(H)-\wt{f}(t)}{H-t}\\
%\end{align*}
\begin{align*}
&\int \frac{P(H)-P(t)}{H-t} f(H)\\
=&\int \frac{P(H)-P(t)}{H-t}\wt{f}(H)\\
=&\int P(H)\frac{\wt{f}(H)-\wt{f}(t)}{H-t}-\int P(t)\frac{\wt{f}(H)-\wt{f}(t)}{H-t}+\int \frac{P(H)-P(t)}{H-t}\wt{f}(t)\\
=&0+0+\wt{f}(t)=\wt{f}(t).
\end{align*}
Where in the second last equality, the first term is zero because the integrand is a multiple of $P(H)$, the second term is zero because $\frac{\wt{f}(H)-\wt{f}(t)}{H-t}$ is a polynomial of degree at most $n-1$, and the last term is $\wt{f}(t)$ because $\frac{P(H)-P(t)}{H-t}$ is monic of degree $n$.
\end{proof}
\begin{rmk}
Let $G$ be a linear algebraic group and $V$ be a representation. Then, 
\begin{align*}
A^{\bullet}_G(\mb{P}(V))&\cong A^{\bullet}_G(\pt)[H]/(P(H))\\
A^{\bullet}_G(\mb{P}(V)\times \mb{P}(V))&\cong A^{\bullet}_G(\pt)[H_1,H_2]/(P(H_1),P(H_2)),
\end{align*}
where $P\in A^{\bullet}_G[T]$ is $T^{\dim(V)}+c_1^G(V)T^{\dim(V)-1}+\cdots+c^G_{\dim(V)}(V)$ by the projective bundle theorem and the class of the diagonal in $\mb{P}(V)\times \mb{P}(V)$ is $(P(H_1)-P(H_2))/(H_1-H_2)$, giving a geometric interpretation of \Cref{diagonalclass}. This can be proven, for example, by first noting that it suffices to consider the case $G=GL(V)$. Then, we can restrict to a maximal torus \cite[Proposition 6]{EG98} and use the fact that the diagonal in $\mb{P}(V)\times \mb{P}(V)$ admits a torus-equivariant deformation into a union of products of coordinate linear spaces \cite[Theorem 3.1.2]{Brion}. 
\end{rmk}

\begin{proof}[Proof of \Cref{PGLmod2}]
Note that when all $a_i$ are odd and all $e_i$ are even then $n=\sum a_ie_i$ is either equal to $\sum e_i$, or exceeds it by at least $4$, so $q_{e_1+\ldots+e_k} \mid q_n$ and the claimed expression for $[Z_\lambda]$ is well-defined.

We resolve $Z_\lambda$ birationally with the map $$\Psi:\prod_{i=1}^k\mb{P}^{e_i} \to \mb{P}^n$$
taking $(D_1,\ldots,D_k) \mapsto a_1D_1+\ldots+a_kD_k$ (treating $P^r=\text{\rm Sym}^r\mb{P}^1$ for all $r$).

If at least one $e_i$ is odd, then we claim $c_3[Z_{\lambda}]=0$. Indeed, $$c_3[Z_\lambda]=\Psi_* c_3,$$ and $c_3\in A^{\bullet}_{PGL_2}(\pt)$ maps to $0$ in $A^\bullet_{PGL_2}(\prod_{i=1}^k \mb{P}^{e_i})$ as the projection $\prod_{i=1}^k \mb{P}^{e_i}\to \pt$ can be factored as the composite $\prod_{i=1}^k \mb{P}^{e_i} \to \mb{P}^{e_i} \to \pt$, and if $e_i$ is odd then $c_3$ pulls back to zero in $A^\bullet_{PGL_2}(\mb{P}^{e_i})$ by \Cref{PGLunordered}.

Hence, as $c_3[Z_\lambda]=0$, we must have $[Z_\lambda]$ is zero in $A^{\bullet}_{PGL_2}(\mb{P}^n)\otimes\mb{Z}/2\mb{Z}$.

%To compute $[Z_\lambda]$, we note that the expression $\frac{q_n(t)-q_n(H)}{t-H}$ for indeterminate $t$ has the property that under the map
%$$\int: A^\bullet_{PGL_2}(\mb{P}^n)\otimes \mb{Z}/2\mb{Z} \to A^\bullet_{PGL_2}(\pt)\otimes \mb{Z}/2\mb{Z},$$ we have $\int [Z_\lambda] \frac{q_n(t)-q_n(H)}{t-H} = [Z_\lambda](t)$, where $[Z_\lambda](t)$ is the polynomial formed by taking the reduced expression for $[Z_\lambda]$ mod $q_n(t)$ and replacing all $H$'s with the $t$. 

Now, suppose that all $e_i$ are even. For the remainder of the proof all integrals are in Chow rings after tensoring with $\mb{Z}/2\mb{Z}$. By \Cref{diagonalclass}, it suffices to show
\begin{align*}
\int_{\mb{P}^n}\frac{q_n(t)-q_n(H)}{t-H}\cap \Psi_{*}1=\begin{cases}\frac{q_n}{q_{d}}(t)&\text{if all $a_i$ and $\frac{d!}{e_1!\ldots e_k!}$ are odd and}
\\ 0&\text{otherwise.}\end{cases}
\end{align*}
By the projection formula applied to $\Psi$, we have
$$\int_{\mb{P}^n}\frac{q_n(t)-q_n(H)}{t-H}\cap \Psi_*1=\int_{\prod_{i=1}^k \mb{P}^{e_i}}\frac{q_n(t)-q_n(\sum a_iH_i)}{t-\sum a_iH_i}.$$

Now, if any $a_i$ is even, then as we are working modulo $2$, $\frac{q_n(t)-q_n(\sum a_iH_i)}{t-\sum a_iH_i}$ will not contain $H_i$, so the integral is clearly zero. Hence we may assume from now on that all $a_i$ are odd, so that $\sum a_iH_i = \sum H_i$ mod 2.

We claim that $q_{d}(\sum H_i)=0$ and that $$\int_{\prod_{i=1}^k \mb{P}^{e_i}}\frac{q_{d}(t)-q_{d}(\sum H_i)}{t-\sum H_i}=\frac{d!}{e_1!\cdots e_k!}.$$ The first of these follows from pulling back $q_d(H)$ under the multiplication map $\prod_{i=1}^k \mb{P}^{e_i} \to \mb{P}^{d}$, and the second of these follows from applying \Cref{diagonalclass} to $1\in A^{\bullet}_{PGL_2}(\mb{P}^d)$ together with the projection formula as the multiplication map has degree $\frac{d!}{e_1!\cdots e_k!}$.

From the vanishing of $q_d(\sum H_i)$, we have 
\begin{align*}
\frac{q_n(t)-q_n(\sum H_i)}{t-\sum H_i}&=\frac{q_{n}}{q_{d}}(t)\frac{q_{d}(t)-q_{d}(\sum H_i)}{t-\sum H_i}+q_d(\sum H_i)\frac{\frac{q_n(t)}{q_d(t)}-\frac{q_n(\sum H_i)}{q_d(\sum H_i)}}{t-\sum H_i}\\
&=\frac{q_{n}}{q_{d}}(t)\frac{q_{d}(t)-q_{d}(\sum H_i)}{t-\sum H_i},
\end{align*}
and the result now follows from the second claim after applying $\int_{\prod_{i=1}^{k}\mb{P}^{e_i}}$ to both sides.
\end{proof}

We now prove surprisingly that despite the presence of occasional $2$-torsion, integral relations between $[Z_\lambda]$ classes in $A^\bullet_{GL_2}(\mb{P}^n)$ are equivalent to integral relations between $[Z_{\lambda}]$-classes in $A^\bullet_{PGL_2}(\mb{P}^n)$.

\begin{thm}
\label{noextrarelations}
Let $n,d$ be integers. A linear combination $\sum a_\lambda [Z_\lambda]$ with  $a_\lambda \in \mb{Z}$ and each $\lambda$ a partition of $n$ into $d$ parts is zero in $A^\bullet_{PGL_2}(\mb{P}^n)$ if and only if it is zero in $A^\bullet_{GL_2}(\mb{P}^n)$. In particular, $\sum a_\lambda[Z_\lambda]=0$ if and only if $$\sum_{\lambda=a_1^{e_1}\ldots a_n^{e_k}} a_\lambda \prod_{i=1}^k \frac{(z^{a_i}-1)^{e_i}}{e_i!}=0.$$
\end{thm}
\begin{proof}
One direction is trivial, as we have the map $A^\bullet_{PGL_2}(\mb{P}^n) \to A^\bullet_{GL_2}(\mb{P}^n)$ induced by $GL_2\to PGL_2$, so if a linear relation holds in $A^\bullet_{PGL_2}(\mb{P}^n)$, then it also holds in $A^\bullet_{GL_2}(\mb{P}^n)$. Conversely, suppose that we have $\sum a_\lambda [Z_\lambda]=0$ in $A^\bullet_{GL_2}(\mb{P}^n)$. We only have to care about the case that $n$ is even, because when $n$ is odd, $A^{\bullet}_{PGL_2}(\mb{P}^n) \hookrightarrow A^{\bullet}_{GL_2}(\mb{P}^n)$ is an injection by \Cref{PGLunordered}. 

For $n$ even, suppose we have a sum $\sum a_\lambda [Z_\lambda]$, which is $0$ in $A^\bullet_{GL_2}(\mb{P}^n)$. Then since the kernel of $A^{\bullet}_{PGL_2}(\mb{P}^n)\to A^{\bullet}_{GL_2}(\mb{P}^n)$ is $2$-torsion by \Cref{PGLunordered}, we know $\sum a_\lambda [Z_\lambda]$ is $2$-torsion in $A^\bullet_{PGL_2}(\mb{P}^n)$. By \Cref{PGLmod2}, the class $[Z_\lambda]$ in $A^{\bullet}_{PGL_2}(\mb{P}^n)\otimes\mb{Z}/2\mb{Z}$ is either $0$ or $\frac{q_n}{q_d}(H)$, and the second possibility occurs precisely when $\lambda$ is special. Hence to prove \Cref{noextrarelations}, by \Cref{FNRformula} and \Cref{linindeplem} it suffices to show that if 
\begin{equation}
\sum_{\lambda=a_1^{e_1}\ldots a_n^{e_k}} a_\lambda \prod_{i=1}^k \frac{(z^{a_i}-1)^{e_i}}{e_i!}=0,
\label{paleq}
\end{equation}
then
\begin{align*}
\sum_{\lambda\text{ special}}a_\lambda\equiv 0\pmod{2}. 
\end{align*}

Note first that if no special $\lambda$ appears we are done, so we may assume that at least one special $\lambda$ appears. As $d=\sum_{i=1}^k e_i$ for any partition $\lambda=a_1^{e_1}\ldots a_n^{e_k}$ appearing, we must have $d$ is even if a special $\lambda$ appears. Multiplying \eqref{paleq} by $\frac{d!}{(z-1)^d}$ and plugging in $z=1$, we have

$$\sum_{\lambda=a_1^{e_1}\ldots a_n^{e_n}} a_\lambda \frac{d!}{e_1!\cdots e_k!}\prod_{i=1}^k {a_i}^{e_i}=0.$$
Now we claim that $\frac{d!}{e_1!\cdots e_k!}$ is even if any $e_i$ is odd. Indeed, as $d$ is even, if not all $e_i$ are even, then at least two of the $e_i$ are odd. If $e_i,e_j$ are both odd, then replacing $e_i!e_j!$ in $\frac{d!}{e_1!\cdots e_k!}$ with $(e_i-1)!(e_j+1)!$ yields an integer with a smaller power of $2$ dividing it.

Hence, $\frac{d!}{e_1!\cdots e_k!}\prod_{i=1}^k {a_i}^{e_i}$ is odd precisely when $\lambda$ is special. Taking the equality $\pmod{2}$ then yields the desired result.
\end{proof}

We complete the proof of \Cref{unorderedrelations}. 

\begin{proof}[Proof of \Cref{unorderedrelations}]
We have (1), (2) and (4) are equivalent by \Cref{noextrarelations}. Also (3) implies (2) is clear as $A^{\bullet}_{GL_2}(\mb{P}^n)$ is free as an abelian group, so $A^{\bullet}_{GL_2}(\mb{P}^n)\hookrightarrow A^{\bullet}_{GL_2}(\mb{P}^n) \otimes \mb{Q}$.

To finish, it suffices to show (2) implies (3). Let $\lambda=(\lambda_1,\ldots,\lambda_d)$ for $\lambda_1\geq \cdots\geq \lambda_d$. 
\begin{clm*}
Suppose $\lambda_3>1$. Then using pushforwards of square relations in $A^{\bullet}_{PGL_2}((\mb{P}^1)^n)$, we can express $[\lambda]\in A^{\bullet}_{PGL_2}(\mb{P}^n)$ in terms of classes $[\lambda']$ where $\lambda'=(\lambda_1',\ldots,\lambda_d')$ where $\lambda_1+\lambda_2 > \lambda_1'+\lambda_2'$. 
\end{clm*}

\begin{proof}[Proof of Claim]
Pick a partition $P=\{A_1,\ldots,A_d\}$ of $[n]$ with $|A_i|=\lambda_i$. Since $|A_3|>1$, we can partition it as $A_3=A_3'\sqcup A_3''$ into nonempty parts. Now, applying the square relation associated to $P'=\{A_1,A_2,A_3',A_3'',\ldots,A_d\}$ of $[n]$ into $d+1$ parts and the parts $A_1,A_2,A_3,A_3''$ shows
\begin{align*}
[\lambda]=[\lambda_1]+[\lambda_2]-[\lambda_3],
\end{align*}
where $\lambda_3'=|A_3'|$ and $\lambda_3''=|A_3''|$ and
\begin{align*}
\lambda_1&=\{\lambda_1+\lambda_3',\lambda_2,\lambda_3'',\ldots,\lambda_d\}\\
\lambda_2 &= \{\lambda_1,\lambda_2+\lambda_3'',\lambda_3',\ldots,\lambda_d\}\\
\lambda_3 &= \{\lambda_1+\lambda_2,\lambda_3',\lambda_3'',\ldots,\lambda_d\}.
\end{align*}
\end{proof}
Returning to the proof of \Cref{unorderedrelations}, iterating the claim shows that the pushforward of square relations allow us to rewrite any $[\lambda]$ in terms of the $\mb{Q}$-basis found in \Cref{basis}, which shows (2) implies (3).
\end{proof}

\section{Excision of unordered strata in $[\text{\rm Sym}^n\mb{P}^1/PGL_2]$}
\label{excision}
As an application of our results in the ordered case, we will prove the following result on the $PGL_2$-equivariant Chow ring of $\mb{P}^n$ with strata excised, which we will adapt in the next section to the case of $GL_2$-equivariant Chow rings with strata in both $\mb{P}^n$ and in $\mb{A}^{n+1}$.

\begin{comment}
In \Cref{excision}, we will always let $GL_2$ act on $\mb{A}^{n+1}$ via the symmetric power of the standard representation ${\rm Sym}^{n}K^2\cong \mb{A}^{n+1}$. 

\begin{defn}
Given a partition $\lambda$ of $n$, let $\wt{Z}_{\lambda}\subset \mb{A}^{n+1}$ be the cone of $Z_{\lambda}\subset\mb{P}^n$ under the projection $\mb{A}^{n+1}\dashrightarrow \mb{P}^n$, and let $[\wt{Z}_{\lambda}]\in A_{GL_2}^{\bullet}(\mb{A}^{n+1})$ be its class. 
\end{defn}

\begin{rmk}
\label{constantterm}
By \Cref{reconstruction}, if $[Z_{\lambda}]\in A_{GL_2}^{\bullet}(\mb{P}^{n})$ is written as a polynomial in $H,u,v$  that is degree at most $n$ in $H$, then $[\wt{Z}_{\lambda}]$ is the constant term in $H$. 
\end{rmk}
\end{comment}

\begin{thm}
\label{pushforwardtheorem}
Given a partition $\lambda=\{\lambda_1,\ldots,\lambda_d\}$ of $n$,
\begin{align*}%A^\bullet_{G}(\mathbb{A}^{n+1}\setminus \wt{Z_{(a_1^{e_1}\ldots a_k^{e_k})}})\otimes \mathbb{Q}=
A^\bullet_{PGL_2}(\mathbb{P}^n\setminus Z_{\lambda}) &= A^{\bullet}_{PGL_2}(\mb{P}^n)/I,%\\
%A^\bullet_{GL_2}(\mb{A}^{n+1}\setminus \wt{Z}_{\lambda})\otimes \mathbb{Q} &= A^{\bullet}_{GL_2}(\mb{A}^{n+1})\otimes\mb{Q}/\wt{I}\otimes \mb{Q},
\end{align*}
where the ideal $I\otimes\mb{Q}\subset A^\bullet_{PGL_2}(\mb{P}^n)\otimes\mb{Q}$ is generated by all $[\lambda']$ for $\lambda'$ a partition formed by merging some of the parts of $\lambda$.
%$[\lambda_{A,B}]$ as $A,B\subset \{1,\ldots,d\}$ vary over all disjoint (possibly empty) subsets and $\lambda_{A,B}$ is 
%$$\{\sum_{i\in A}\lambda_i, \sum_{i\in B}\lambda_i\}\cup\{\lambda_i\mid i\in [d]\backslash (A\cup B)\}$$
%given by merging the parts of $\lambda$ corresponding to $A$ and $B$.
% Similarly, the ideal $\wt{I}\subset A^\bullet_{GL_2}(\mb{A}^{n+1})$ is generated by the classes $[\wt{Z_{\lambda_{A,B}}}]\in A^{\bullet}_{GL_2}(\mb{A}^{n+1})$.
\end{thm}

Even though \Cref{pushforwardtheorem} requires many generators for $I$, in some cases fewer generators suffice.

\begin{thm}
\label{2generators}
Given the partition $\lambda=\{a,1^{n-a}\}$ of $n$, the ideal $I\otimes \mb{Q}$ in \Cref{pushforwardtheorem} is generated by $[\lambda]$ and $[\lambda']$, where 
\begin{align*}
\lambda'=
\begin{cases}
\{a+1,1^{n-a-1}\}&\qquad\text{if }a\neq \frac{n}{2}\\
\{a,2,1^{n-a-2}\}&\qquad\text{if }a=\frac{n}{2}. 
\end{cases}
\end{align*}
%Similarly, the ideal $\wt{I}$ in \Cref{pushforwardtheorem} is generated by $[\wt{Z_{\lambda}}]$ and $[\wt{Z_{\lambda'}}]$. 
\end{thm}
See \Cref{Feherrmk} for the connection to similar results proved in \cite{FNR06}.

By the excision exact sequence \cite[Proposition 1.8]{Fulton}, the ideal $I$ is the same as the pushforward ideal $I_\lambda$ which we define in \Cref{Ilambda}.

%Recall that we used $\wt{\bullet}$ to denote the cone over a $T$-invariant subvariety of $\mb{P}^n$ (\Cref{cone}).

\begin{defn}\label{Ilambda}
Given a partition $\lambda$ of $n$ and for $G=PGL_2$ or $GL_2$, let $I^G_{\lambda}$ be the ideal of $A_G^\bullet(\mb{P}^n)$ given by the pushforward via the inclusion $\iota_{\lambda}: Z_\lambda\hookrightarrow \mb{P}^n$
$$I^G_{\lambda}=(\iota_\lambda)_*A^{G}_{\bullet}(Z_{\lambda})\subset A^G_{\bullet}(\mb{P}^n)$$
and the identification $A^G_{\bullet}(\mb{P}^n)\cong A_G^{n-\bullet}(\mb{P}^n)$ via Poincar\'e duality \cite[Proposition 4]{EG98}. When $G$ is clear from context we will simply write $I_\lambda$.
%and similarly the ideal $\wt{I}_{\lambda}\subset A_{GL_2}^\bullet(\mb{A}^{n+1})$ be the ideal given by the image of the pushforward map
%\begin{align*}
%\iota:A_{GL_2}^{\bullet}(\wt{Z}_{\lambda})\to A_{GL_2}^\bullet(\mb{A}^{n+1}).
%\end{align*}
\end{defn}

Since $Z_\lambda$ is possibly singular, we will want to instead work with a desingularization (as was done in \cite{FNR06}). 

\begin{defn}\label{thePs}
Given a partition $\lambda=\{\lambda_1,\ldots,\lambda_d\}$ of $n$, let $e^\lambda_i=\# \{j\mid \lambda_j=i\}$ and $Y_\lambda=\prod_{i=1}^{n}{\mb{P}^{e^\lambda_i}}$. We have a map $$\hat{\iota}_{\lambda}: Y_\lambda\to \mb{P}^n$$ that is birational onto its image $Z_\lambda$ given by the composition $$Y_\lambda \hookrightarrow \prod_{i=1}^{n}{\mb{P}^{i e^\lambda_i}}\to \mb{P}^n$$ of the $i$th power map on each factor $\mb{P}^{e_i}$ together with the multiplication map. Equivalently, if we view projective space $\mb{P}^n$ as parameterizing degree $n$ divisors on $\mb{P}^1$, then the map is given by $(D_1,\ldots,D_n)\mapsto \sum_{i=1}^{n}{i D_i}$. 
\end{defn}

In particular, $I_\lambda$ is also given by the image of $(\hat{\iota}_\lambda)_{*}$. Since we are working rationally, we can take a finite cover of $Y_\lambda$. 

\begin{defn}
Given a partition $\lambda=\{\lambda_1,\ldots,\lambda_d\}$ of $n$, define the finite map $\Phi_\lambda:(\mb{P}^1)^d \to Y_\lambda$ to be
$$\Phi_\lambda:(\mb{P}^1)^d= \prod_{i=1}^{n}{(\mb{P}^1)^{e^\lambda_i}}\to \prod_{i=1}^{n}{\mb{P}^{e^\lambda_i}}= Y_\lambda $$
given by the multiplication map $(\mb{P}^1)^{e^\lambda_i}\to \mb{P}^{e^\lambda_i}$ on each factor. %Here, $(\mb{P}^1)^d\hookrightarrow (\mb{P}^1)^n$ is an isomorphism onto $\Delta_{\{A_1,\ldots,A_d\}}$, where $\{A_1,\ldots,A_d\}$ is any partition of $[n]$ with $|A_i|=\lambda_i$ and $\Phi$ is the multiplication map. By symmetry it is clear that $\Phi_\lambda$ is independent of the partition of $[n]$ we choose.
\end{defn}

Since $\Phi_\lambda$ is finite, $$(\Phi_\lambda)_*:A^\bullet_{PGL_2}((\mb{P}^1)^d)\otimes \mb{Q} \to A^\bullet_{PGL_2}(Y_\lambda)\otimes \mb{Q}$$ is surjective, so $I_\lambda\otimes \mb{Q}$ is the image of $$(\hat{\iota}_\lambda \circ \Phi_\lambda)_*:A^\bullet_{PGL_2}((\mb{P}^1)^d)\otimes \mb{Q}\to A^\bullet_{PGL_2}((\mb{P}^1)^n)\otimes \mb{Q}.$$ The map $\Phi_\lambda$ has the nice property that given a partition $P$ of $[d]$, the pushforward of the strata $(\hat{\iota}_\lambda \circ \Phi_\lambda)_*\Delta_P$ is $[\lambda']$, where $\lambda'$ is the partition of $n$ given by merging the parts of $\lambda$ according to the partition $P$. From this, we will be able to deduce certain symmetrized strata generate $I_\lambda \otimes \mb{Q}$ based on the generation properties of strata in $(\mb{P}^1)^d$.

\begin{defn}
Given a set of partitions $\mc{P}$ of $[d]$ and $G=PGL_2$ or $GL_2$, let $\Lambda^G_{\mc{P}} \subset A^\bullet_G((\mb{P}^1)^d) \otimes \mb{Q}$ be the submodule over $A^{\bullet}_G(\mb{P}^n)\otimes \mb{Q}$ generated by the classes $\Delta_P$. Explicitly,
$$\Lambda^G_{\mc{P}}=\sum_{P \in \mc{P}}\Delta_P \cap \Phi_\lambda^*\hat{\iota}_\lambda^*(A^\bullet_G(\mb{P}^n)\otimes \mb{Q}).$$
When $G$ is clear from context we will notate $\Lambda^G_{\mc{P}}$ simply by $\Lambda_{\mc{P}}$.
\end{defn}

\begin{lem}\label{twocrit}
Let $\lambda=\{\lambda_1,\ldots,\lambda_d\}$ be a partition of $n$, and let $G=PGL_2$ or $GL_2$. Suppose we have a collection of partitions $\mc{P}$ of $[d]$ such that in $A^\bullet_{G}((\mb{P}^1)^d)\otimes \mb{Q}$
%$$\Phi_\lambda^*(A^\bullet_{G}(Y_\lambda)\otimes \mb{Q})\subset \Lambda^G_{\mc{P}}.$$
$$A^\bullet_{G}((\mb{P}^1)^d)^{\prod_{i=1}^{n}{S_{e^\lambda_i}}}\otimes \mb{Q}\subset \Lambda^G_{\mc{P}}.$$
Then $\{(\hat{\iota}_\lambda \circ \Phi_\lambda)_*\Delta_P\mid P\in \mc{P}\}$ generates $I_\lambda^G\otimes \mb{Q}\subset A^\bullet_G(\mb{P}^n)\otimes \mb{Q}.$
\end{lem}
\begin{proof}
Since 
$$\Phi_\lambda^*(A^\bullet_{G}(Y_\lambda)\otimes \mb{Q})\subset A^\bullet_{G}((\mb{P}^1)^d)^{\prod_{i=1}^{n}{S_{e^\lambda_i}}}\otimes \mb{Q}\subset \Lambda^G_{\mc{P}},$$
we have $$(\Phi_\lambda)_*\Lambda_{\mc{P}}^{G}\supset (\Phi_\lambda)_*(\Phi_\lambda^*(A^\bullet_{G}(Y_\lambda))\otimes \mb{Q})=A^\bullet_{G}(Y_\lambda) \otimes \mb{Q}$$
and by the projection formula, $(\Phi_\lambda)_*\Lambda_{\mc{P}}^{G}$ is
$$(\Phi_\lambda)_*\sum_{P \in \mc{P}}\Delta_P\cap \Phi_\lambda^*\hat{\iota}_\lambda^*(A^\bullet_{G}(\mb{P}^n)\otimes \mb{Q})=\sum_{P \in \mc{P}}(\Phi_\lambda)_*\Delta_P \cap \hat{\iota}_\lambda^*(A^\bullet_G(\mb{P}^n)\otimes \mb{Q}).$$
By the projection formula again, we thus have $$I^G_\lambda\otimes \mb{Q}=(\hat{\iota}_\lambda)_*(A^\bullet_G(Y_\lambda)\otimes \mb{Q})=\sum_{P \in \mc{P}}(\hat{\iota}_\lambda\circ \Phi_\lambda)_*\Delta_P\cap A^\bullet_G(\mb{P}^n) \otimes \mb{Q}$$ as desired.
\end{proof}
\begin{comment}
\begin{lem}\label{twocrit}
Let $\lambda=\{\lambda_1,\ldots,\lambda_d\}$ be a partition of $n$. Suppose we have a collection of partitions $\mc{P}$ of $[d]$ such that $\mb{Q}[u,v]^{S_2}$-span of the collection $$\{\Phi_\lambda^*(\iota_\lambda^* H)^k\cap \Delta_P\mid P\in \mc{P}, k\geq 0\}$$ contains the subring $\Phi_\lambda^*(A^\bullet_{GL_2}(Z_\lambda))\otimes \mb{Q} \subset A^\bullet_{GL_2}((\mb{P}^1)^d) \otimes \mb{Q}$. Then $\{(\iota_\lambda \circ \Phi_\lambda)_*\Delta_P\mid P\in \mc{P}\}$ generates the ideal $I_\lambda \otimes \mb{Q}$.
\end{lem}

\begin{proof}
We have $$(\Phi_\lambda)_*(\Phi_\lambda^*(A^\bullet_{GL_2}(Z_\lambda))\otimes \mb{Q})=A^\bullet_{GL_2}(Z_\lambda) \otimes \mb{Q}$$
and by the projection formula,
$$(\Phi_\lambda)_*(\Phi_\lambda^*(\iota_\lambda^* H)^k\cap \Delta_P)=(\iota_\lambda^* H)^k\cap (\Phi_\lambda)_*\Delta_P,$$ so these classes generate $A^\bullet_{GL_2}(Z_\lambda) \otimes \mb{Q}$ as a $\mb{Q}[u,v]^{S_2}$-module. Hence $I_\lambda$ is the $\mb{Q}[u,v]^{S_2}$-span of
$$(\iota_\lambda)_*((\iota_\lambda^* H)^k\cap (\Phi_\lambda)_*\Delta_P)=H^k \cap (\iota_\lambda \circ \Phi_\lambda)_*\Delta_P.$$
\end{proof}
\end{comment}

\begin{lem}\label{pushforwardlemma}
Let $\lambda=\{\lambda_1,\ldots,\lambda_d\}$ be a partition of $[n]$ and $\mc{P}$ be all partitions of $[d]$. Then

$$
\Lambda^{PGL_2}_{\mc{P}}=
\begin{cases}
A^{\bullet}_{PGL_2}((\mb{P}^1)^2)^{S_2}\otimes\mb{Q}&\text{if $d=2$ and $\lambda_1=\lambda_2$, and}\\
A^{\bullet}_{PGL_2}((\mb{P}^1)^d)\otimes\mb{Q}&\text{otherwise.}
\end{cases}
$$

In particular, given a partition $\lambda=\{\lambda_1,\ldots,\lambda_d\}$ of $n$, $I^{PGL_2}_{\lambda}\otimes \mb{Q}$ is generated by all $[\lambda']$ with $\lambda'$ formed by merging parts of $\lambda$.
\begin{comment}
$[\lambda_{A,B}]$ where we define $\lambda_{A,B}$ by taking $A$ and $B$ a disjoint pair of (possibly empty) subsets $A,B\subset [d]$ and merge the $\{\lambda_i\}_{i \in A}$ and $\{\lambda_j\}_{j \in B}$ in $\lambda$ to get
$$\lambda_{A,B}:=\{\sum_{i\in A}\lambda_i, \sum_{i\in B}\lambda_i\}\cup\{\lambda_i\mid i\in [d]\backslash (A\cup B)\}.$$
\end{comment}
\end{lem}
\begin{proof}
Given the description of $\Lambda_\mc{P}^{PGL_2}$, the result about $I^{PGL_2}_\lambda \otimes \mb{Q}$ follows directly from \Cref{twocrit}. We will now show the description of $\Lambda_{\mc{P}}^{PGL_2}$.

We may identify $A^\bullet_{PGL_2}(\mb{P}^n)\otimes \mb{Q}\subset A^\bullet_{GL_2}(\mb{P}^n)\otimes \mb{Q}$ as the subring generated by $H+\frac{n}{2}(u+v)$ and $(u-v)^2$ by \Cref{PGLunordered}. Define
\begin{align*}
H_i'&=H_i+\frac{1}{2}(u+v)\quad\text{and}\quad H'=H+\frac{n}{2}(u+v).
\end{align*}
Note that with these definitions, we have $$\Phi_\lambda^*\hat{\iota}_\lambda^*(H')=\sum \lambda_iH_i',\qquad H_i'^2=\frac{1}{4}(u-v)^2.$$

We have the $\mb{Q}$-linear span $$\Lambda_{\mc{P}}=\on{Span}_{\mb{Q}}(\{\Delta_P(u-v)^{2k}(\sum \lambda_iH_i')^{\ell}\mid k,l\ge 0, P\in \mc{P}\}).$$

\begin{comment}
the $\mb{Q}[u,v]^{S_2}$-span of $$\{\Phi_\lambda^*(\iota_\lambda^*H)^k\cap \Delta_P \mid P\in \mc{P}, k \ge 0\},$$
where $\mathcal{P}$ is the set of partitions of $[d]$ with at most two non-singleton parts.
\end{comment}
The trivial partition is in $\mc{P}$, so $1$ is automatically in $\Lambda_{\mc{P}}$.

%Showing $\Lambda=A^\bullet_{GL_2}((\mb{P}^1)^d) \otimes \mb{Q}$ is equivalent to producing an expression in $\Lambda$ of the form $$\prod_{i \in C} H_i+\text{lower order terms in the $H_i$}$$ for each $C \subset [d]$. The trivial partition is in $\mc{P}$, so $1$ is automatically in $\Lambda$ and we can restrict our attention to non-empty $C$.

Recall by \Cref{psidelta} that $$\Delta_{i,j}=H_i'+H_j',$$ and that $A^\bullet_{PGL_2}((\mb{P}^1)^d)\otimes \mb{Q}$ is generated by the $H_i'$ and $(u-v)^2$. As $H_i'^2=\frac{1}{4}(u-v)^2$, to show $\Lambda_\mc{P}=A^\bullet_{PGL_2}((\mb{P}^1)^d)\otimes \mb{Q}$ it suffices to show that every monomial $\prod_{i \in C}H_i'$ is in $\Lambda_{\mc{P}}$ for $C \subset [n]$.

For $d=1$, $\lambda=\{[n]\}$, we are done as $H_1'=\frac{1}{\lambda_1}\Phi_\lambda^*\iota_\lambda^*H'$.

For $d=2$ and $\lambda_1 \ne \lambda_2$, 
\begin{align*}
H_1'& =\frac{1}{\lambda_1-\lambda_2}(\Phi_\lambda^*\hat{\iota}_\lambda^*(H')-\lambda_2\Delta_{1,2})\\
H_2'&=\frac{1}{\lambda_2-\lambda_1}(\Phi_\lambda^*\hat{\iota}_\lambda^*(H')-\lambda_1\Delta_{1,2})\\
H_1'H_2'&=\frac{1}{2\lambda_1\lambda_2}(\Phi_\lambda^*\hat{\iota}_\lambda^*(H')^2-\frac{1}{4}(\lambda_1^2+\lambda_2^2)(u-v)^2).
\end{align*}
%and 
%$$H_1'H_2'=\frac{1}{2\lambda_1\lambda_2}(\Phi_\lambda^*\iota_\lambda^*(H')^2-\frac{1}{4}(\lambda_1^2+\lambda_2^2)(u-v)^2).$$
%where for the last expression we note that we can replace $H_i^2$ with $-(u+v)H_i-uv$.

For $d=2$ and $\lambda_1=\lambda_2=a$, we have to show $\Lambda_{\mc{P}}=A^\bullet_{PGL_2}((\mb{P}^1)^2)^{S_2}\otimes \mb{Q}$. As $H_i'^2=\frac{1}{4}(u-v)^2$, %note that $\Phi_\lambda^*(A_{GL_2}^{\bullet}(Z_\lambda))$ is invariant under permuting $H_1$ with $H_2$, so 
it suffices to show $1,H_1'+H_2'$ and $H_1'H_2'$ are in $\Lambda_{\mc{P}}$. We already know that $1\in \Lambda_{\mc{P}}$, and
\begin{align*}
H_1'+H_2'&=\frac{1}{a}\Phi_\lambda^*\hat{\iota}_\lambda^*H',\\
H_1'H_2'&=\frac{1}{2a^2}(\Phi_\lambda^*\hat{\iota}_\lambda^*(H')^2-\frac{1}{2}a^2(u-v)^2).
\end{align*}

\begin{comment}
$(\Phi_\lambda)_*\Phi^*_\lambda$ is multiplication by $\deg(\Phi_\lambda)$, so it suffices to show that the images of $A^\bullet_{GL_2}((\mb{P}^1)^2)$ under $(\iota \circ \Phi_\lambda)_*$ lie in $I_\lambda$. We have
\begin{align*}
(\iota \circ \Phi_\lambda)_*1&=[\lambda]\\
\iota \circ (\Phi_\lambda)_*(H_1)&=\iota \circ (\Phi_\lambda)_*(H_2)=\frac{1}{2a}(\iota \circ \Phi_\lambda)_*(\Phi_\lambda^*(\iota^* H))=\frac{1}{2a}[H][\lambda],\\ \intertext{and}
\iota_* (\Phi_\lambda)_* (H_1H_2)&= \iota_* (\Phi_\lambda)_* (\Phi_\lambda^*\iota^*(\frac{1}{2a^2}H^2+\frac{1}{2a}(u+v)H+uv))\\
&=(\frac{1}{2a^2}[H]^2+\frac{1}{2a}(u+v)[H]+uv)[\lambda]
\end{align*}
by the projection formula, so the result again follows.
\end{comment}

We will now show that $\Lambda_{\mc{P}}=A^\bullet_{PGL_2}((\mb{P}^1)^d)\otimes \mb{Q}$ when $d\geq 3$.

Up to degree $d-2$, we can take $k,\ell=0$ as the classes $\Delta_P$ for $P\in \mc{P}$ generate $A^{\le d-2}_{PGL_2}((\mb{P}^1)^d)$ by \Cref{generatePlow}. Hence to conclude the proof of \Cref{pushforwardlemma}, it suffices to show that $\prod_{k \ne i}H_k'$ for all $i$ and $\prod H_k'$ are in $\Lambda_{\mc{P}}$. 

For $\prod_{k \ne i}H_k$, without loss of generality suppose $i=1$. We have each of
\begin{align*}
\frac{1}{a_{1}a_2}(\prod_{k\ne 1,2}H_k') \cap \Phi_\lambda^*\hat{\iota}_\lambda^*H' &= \frac{1}{a_1}\prod_{k \ne 1}H_k'+\frac{1}{a_2}\prod_{k \ne 2}H_k'+\frac{1}{4a_1a_2}(u-v)^2\sum_{j \ne 1,2}a_j\prod_{k \ne 1,2,j}H_k'\\
\frac{1}{a_{1}a_3}(\prod_{k\ne 1,3}H_k') \cap \Phi_\lambda^*\hat{\iota}_\lambda^*H &= \frac{1}{a_1}\prod_{k \ne 1}H_k+\frac{1}{a_3}\prod_{k \ne 3}H_k+\frac{1}{4a_2a_3}(u-v)^2\sum_{j \ne 2,3}a_j\prod_{k \ne 2,3,j}H_k'\\
\frac{1}{a_{2}a_3}(\prod_{k\ne 2,3}H_k') \cap \Phi_\lambda^*\hat{\iota}_\lambda^*H &= \frac{1}{a_2}\prod_{k \ne 2}H_k+\frac{1}{a_3}\prod_{k \ne 3}H_k+\frac{1}{4a_1a_3}(u-v)^2\sum_{j \ne 1,3}a_j\prod_{k \ne 1,3,j}H_k'
\end{align*}
lie in $\Lambda_{\mc{P}}$ as we have already shown each $\prod_{k\ne i,j}H_k$ lies in $\Lambda$. Also, the last term on each right hand side lies in $\Lambda_{\mc{P}}$ as the number of terms in the $H_k'$ monomial is $d-3$. Hence taking a linear combination we get $\prod_{k \ne 1}H_k'\in \Lambda_{\mc{P}}$.

%If $d=2$ and if $a_1=a_2=a$ then $\Phi'_*(H_1)=\Phi'_*(H_2)=\frac{1}{a}\Phi'_*(\Phi'^*(H))$ so the result again follows. If $a_1 \ne a_2$ then $H_1=\frac{1}{a_1-a_2}(\Phi'^*(H)-a_2([\{1,2\}]-(u+v)))$ and $H_2=\frac{1}{a_2-a_1}(\Phi'^*(H)-a_1([\{1,2\}]-(u+v)))$, so the result again follows. For $d=1$, $H_1=\frac{1}{a_1}\Phi'^*(H)$, so the result again follows.

To show $\prod_{i=1}^{n} H_i'\in \Lambda_{\mc{P}}$, we can proceed similarly to above, or expand $$\frac{1}{a_1\ldots a_n}\Phi_\lambda^{*}\hat{\iota}_\lambda^*(H')^d=\prod H_i' + (u-v)^2\left(\text{lower order terms in the $H_i'$}\right),$$
using $H_i'^2=\frac{1}{4}(u-v)^2$.
\end{proof}

\begin{comment}
\begin{thm}
$[a_1,\ldots,a_d]$ can be written in terms of $[a,b;d-2]$ as follows. Define $\alpha_i$ by
$$1+s^{n-(d-2)}-\frac{\prod (s^{a_i}-1)}{(s-1)^{d-2}}=\sum_{2i \le n-d-2} \alpha_i (s^i+s^{n-(d-2)-i}).$$
Then $$[a_1,\ldots,a_d]=\sum_{2i \le n-(d-2)} \alpha_i [i,n-(d-2)-i;d-2].$$
\end{thm}
\begin{proof}
Note that using square relations we can reduce every symmetrized class down to $[a,b;d-2]$ classes, so it suffices to show the above expression is preserved under applying a square relation, and is true for $[a,b;d-2]$ classes. The former follows from the identity
\begin{align*}
&(1-s^{a+b})(1-s^c)(1-s^d)+(1-s^a)(1-s^b)(1-s^{c+d})\\
=&(1-s^{a+c})(1-s^b)(1-s^d)+(1-s^{b+d})(1-s^a)(1-s^c)
\end{align*}
and the latter is immediately clear.
\end{proof}
\end{comment}

\begin{proof}[Proof of \Cref{pushforwardtheorem}]
%Consider $(\mb{P}^1)^{\sum e_i} \to \mb{P}^n$. By \Cref{pushforwardlemma}, we only need to push forward $[\{A,B\}]$ classes.
This follows from the excision exact sequence \cite[Proposition 1.8]{Fulton} and \Cref{pushforwardlemma}.
%and the affine case follows from the projective case, \Cref{divideleray}, and \Cref{constantterm}.
\begin{comment}
The first statement about $A^\bullet_{GL_2}(\mathbb{P}^n\setminus Z_{\lambda})\otimes \mathbb{Q}$ follows from the exact sequence \cite[Proposition 1.8]{Fulton}
\begin{align*}
A^{\bullet}_{GL_2}(Z_{\lambda})\otimes \mb{Q}\to A^\bullet_G(\mathbb{P}^n)\otimes \mb{Q}\to A^\bullet_G(\mathbb{P}^n\setminus Z_{\lambda})\otimes \mathbb{Q}\to 0
\end{align*}
and \Cref{pushforwardlemma}. 
\end{comment}
\end{proof}

\begin{lem}
\label{2generatorslemma}
Let $\lambda=\{a,1^b\}$ be a partition of $n$. Define $\mc{P}_\lambda$ to be the set of partitions
\begin{align*}
\mc{P}_\lambda=\{T\}\sqcup
\begin{cases}
\{T_{1,i}\}_{i \ge 2}&\text{a $\ne$ b}\\
\{T_{i,j}\}_{2 \le i<j \le n}&\text{a=b},
\end{cases}
\end{align*}
where $T$ is the trivial partition and $T_{i,j}$ is the partition with $n-1$ parts and $i,j$ in the same part. Then $$\Lambda^{PGL_2}_{\mc{P_\lambda}}=A^\bullet_{PGL_2}((\mb{P}^1)^{b+1})^{S_1 \times S_b}\otimes \mb{Q}.$$
\end{lem}
\begin{proof}
Define \begin{align*}
H'&=H+\frac{n}{2}(u+v)\quad\text{and}\quad H_i'=H_i+\frac{1}{2}(u+v).
\end{align*}
Then in particular,
\begin{align*}
\Delta_{i,j}&=H_i'+H_j'\\
\Phi_\lambda^*\hat{\iota}_\lambda^*(H')&=aH_1'+H_2'+\ldots+H_{b+1}',
\end{align*}
%$$\Delta_{i,j}=H_i'+H_j'.$$
%We have $$\Phi_\lambda^*\hat{\iota}_\lambda^*(H')=aH_1'+H_2'+\ldots+H_{b+1}',$$
so $\Lambda_{\mc{P}_{\lambda}}$ is the $\mb{Q}$-linear span
$$\Lambda_{\mc{P}_{\lambda}}=\on{Span}_{\mb{Q}}\{\Delta_P(u-v)^{2k}(aH_1'+H_2'+\ldots+H_{b+1}')^\ell\mid k,\ell \ge 0,\ P \in \mc{P}_\lambda\}.$$
%$$\sum_{\substack{k,l \ge 0\\P \in \mc{P}_\lambda}} \mb{Q} \cdot \Delta_P(u-v)^{2k}(aH_1'+H_2'+\ldots+H_{b+1}')^\ell.$$
We first show that $H_1' \in \Lambda_{\mc{P}_\lambda}$. Consider the case $b \ne a$. Then
$$H_1'=\frac{1}{a-b}\left(\Phi_\lambda^*\hat{\iota}_\lambda^*(H')-\sum_{i \ge 2}\Delta_{1,i}\right) \in \Lambda_{\mc{P}_\lambda}.$$
Now consider the case $b=a$. Then
$$H_1'=\frac{1}{a}\left(\Phi_\lambda^*\hat{\iota}_\lambda^*(H')-\frac{1}{a-1}\sum_{2 \le i < j \le a+1}\Delta_{i,j}\right)\in \Lambda_{\mc{P}_\lambda}.$$
Now that we have shown that $H_1' \in \Lambda$, it therefore suffices to show that the invariant subring $A^\bullet_{PGL_2}((\mb{P}^1)^{b+1})^{S_1 \times S_b}$ is given by
$$\on{Span}_{\mb{Q}}\{(u-v)^{2k}(aH_1'+H_2'+\ldots+H_{b+1}')^\ell,\ H_1' (u-v)^{2k}(aH_1'+H_2'+\ldots+H_{b+1}')^\ell\mid k,\ell\ge 0\}$$
%$$\sum \mb{Q} \cdot (u-v)^{2k}(aH_1'+H_2'+\ldots+H_{b+1}')^\ell+\sum \mb{Q} \cdot H_1'\cdot (u-v)^{2k}(aH_1'+H_2'+\ldots+H_{b+1}')^\ell$$
Note that by using the relation $H_1'^2=\frac{1}{4}(u-v)^2$, we see this is the same as
\begin{align*}
&\on{Span}_{\mb{Q}}\{H_1'^k(aH_1'+H_2'+\ldots+H_{b+1}')^{\ell}\mid k,\ell\geq 0\}\\
=&\on{Span}_{\mb{Q}}\{H_1'^k(H_2'+\ldots+H_{b+1}')^{\ell}\mid k,\ell\geq 0\}\\
=& \on{Span}_{\mb{Q}}\{(u-v)^{2k}(H_1'+H_2'+\ldots+H_{b+1}')^\ell,\\ &H_1' (u-v)^{2k}(H_1'+H_2'+\ldots+H_{b+1}')^\ell\mid k,\ell\ge 0\}.
\end{align*}
%\begin{align*}
%&\sum \mb{Q}\cdot H_1'^k(aH_1'+H_2'+\ldots+H_{b+1}')^{\ell}\\
%=&\sum \mb{Q}\cdot H_1'^k(H_2'+\ldots+H_{b+1}')^{\ell}\\
%=&\sum \mb{Q} \cdot (u-v)^{2k}(H_2'+\ldots+H_{b+1}')^\ell+\sum \mb{Q} \cdot H_1'\cdot (u-v)^{2k}(H_2'+\ldots+H_{b+1}')^\ell.
%\end{align*}
By using the relations $H_i'^2=\frac{1}{4}(u-v)^2$ whenever possible, we see that an element of the invariant subring is a sum of terms of the form $(u-v)^{2k}e_j(H_2',\ldots,H_{b+1}')$ and $(u-v)^{2k}H_1'e_j(H_2',\ldots,H_{b+1}')$ where $e_j$ is the $j$th elementary symmetric polynomial, hence it suffices to show that
$$\on{Span}_{\mb{Q}}\{(u-v)^{2k} e_j(H_2',\ldots,H_{b+1}')\mid j,k\geq 0 \}\subset \on{Span}_{\mb{Q}}\{(u-v)^{2k}(H_2'+\ldots+H_{b+1}')^{\ell}\mid k,\ell\geq 0\}.$$
This follows by induction on $j$ and the relation
\begin{align*}
&e_j(H_2',\ldots,H_{b+1}')(H_2'+\ldots+H_{b+1}')
\\&=(j+1)e_{j+1}(H_2'+\ldots+H_{b+1}')+\frac{1}{4}(u-v)^2(n-j+1)e_{j-1}(H_2',\ldots,H_{b+1}').
\end{align*}

\end{proof}

\begin{proof}[Proof of \Cref{2generators}]
This follows from the excision exact sequence \cite{Fulton}[Proposition 1.8], \Cref{twocrit}, and \Cref{2generatorslemma}.
\end{proof}

\section{Excision of unordered strata in $[\text{\rm Sym}^n\mb{P}^1/GL_2]$ and $[\text{\rm Sym}^nK^2/GL_2]$}
In this section, we show how our results about excision of unordered strata in $[\text{\rm Sym}^n\mb{P}^1/PGL_2]$ imply similar results in $[\text{\rm Sym}^n\mb{P}^1/GL_2]$ and $[\text{\rm Sym}^nK^2/GL_2]$, recovering and extending some results of \cite{FNR06} (see \Cref{Feherrmk}).

\begin{defn}
Given a partition $\lambda$ of $n$, let $\wt{I}_\lambda$ be the ideal of $A^\bullet_{GL_2}(\mb{A}^{n+1})$ given by the image of the pushforward $A^{GL_2}_\bullet(\wt{Z_\lambda}) \hookrightarrow A^{GL_2}_\bullet(\mb{A}^{n+1})$ and the identification $A^{GL_2}_\bullet(\mb{A}^{n+1})\cong A^{n+1-\bullet}_{GL_2}(\mb{A}^{n+1})$ via Poincar\'e duality \cite[Proposition 4]{EG98}.
\end{defn}

\begin{thm}
\label{GL2andaffine}
$I_\lambda^{GL_2}\otimes \mb{Q}$ (respectively $\wt{I}_\lambda\otimes\mb{Q}$) is generated by all $[Z_{\lambda'}]$ (respectively $[\wt{Z}_{\lambda'}]$) with $\lambda'$ formed by merging parts of $\lambda$. For $\lambda=\{a,1^{n-a}\}$ only two generators are required, namely $[Z_\lambda]$ (respectively $[\wt{Z}_\lambda]$) and $[Z_{\lambda'}]$ (respectively $[\wt{Z}_{\lambda'}]$) where
$$\lambda'=\begin{cases}\{a+1,1^{n-a-1}\}&\text{if $a \ne \frac{n}{2}$}\\
\{a,2,1^{n-a-2}\}&\text{if $a = \frac{n}{2}$.}\end{cases}$$
\end{thm}
\begin{rmk}
\label{Feherrmk}
In the affine case, when $n$ is odd and $a=\lceil \frac{n}{2}\rceil$ this recovers \cite[Theorem 4.3]{FNR06}, and when $n$ is even and $a=\frac{n}{2}$ this recovers the rational Chow ring of the stable locus in \cite[Theorem 4.10]{FNR06}.
\end{rmk}

\begin{lem}
\label{PGLtoGL}
We have
\begin{align*}
\mb{Q}[u,v]^{S_2}\left(A^\bullet_{PGL_2}((\mb{P}^1)^d)\otimes \mb{Q}\right)^{S_{e_1}\times \ldots \times S_{e_k}}&=\left(A^\bullet_{GL_2}((\mb{P}^1)^d)\otimes \mb{Q}\right)^{S_{e_1}\times \ldots \times S_{e_k}}\text{ and}\\
\mb{Q}[u,v]^{S_2}\left(A^\bullet_{PGL_2}(\mb{P}^n)\otimes \mb{Q}\right)&=A^\bullet_{GL_2}(\mb{P}^n)\otimes \mb{Q}.
\end{align*}
In particular, if a set of partitions $\mc{P}$ satisfies the hypotheses of \Cref{twocrit} for $G=PGL_2$, then they also satisfy the hypotheses of \Cref{twocrit} for $G=GL_2$.
\end{lem}
\begin{proof}
We identify $A^\bullet_{PGL_2}((\mb{P}^1)^d) \otimes \mb{Q}$ as the subring of $A^\bullet_{GL_2}((\mb{P}^1)^d) \otimes \mb{Q}$ via \Cref{PGLunordered} generated by $H':=H+\frac{n}{2}(u+v)$ and $(u-v)^2$. Since $A^\bullet_{GL_2}((\mb{P}^1)^d) \otimes \mb{Q}$ is generated by $H'$ over $\mb{Q}[u,v]^{S_2}$, and $(u-v)^2$ and $u+v$ generate $\mb{Q}[u,v]^{S_2}$, $\mb{Q}[u,v]^{S_2}\left(A^\bullet_{PGL_2}(\mb{P}^n)\otimes \mb{Q}\right)=A^\bullet_{GL_2}(\mb{P}^n)\otimes \mb{Q}$. 

For the other equality, we use \Cref{PGL2injective} to identify $A^\bullet_{PGL_2}(\mb{P}^n)\otimes \mb{Q}$ as the subring of $A^\bullet_{GL_2}(\mb{P}^n)\otimes \mb{Q}$ generated by $H_i':=H_i+\frac{u+v}{2}$. Then, $\left(A^\bullet_{PGL_2}((\mb{P}^1)^d)\otimes \mb{Q}\right)^{S_{e_1}\times \ldots \times S_{e_k}}$ is generated $\mb{Z}$-linearly by all $p(H_1',\ldots,H_n')$, where $p$ is a polynomial invariant under the action of $S_{e_1}\times\cdots\times S_{e_k}$. Similarly, $\left(A^\bullet_{GL_2}((\mb{P}^1)^d)\otimes \mb{Q}\right)^{S_{e_1}\times \ldots \times S_{e_k}}$ is generated by all such $p(H_1',\ldots,H_n')$, together with $u+v$ and $uv$. Therefore, 
$$\mb{Q}[u,v]^{S_2}\left(A^\bullet_{PGL_2}((\mb{P}^1)^d)\otimes \mb{Q}\right)^{S_{e_1}\times \ldots \times S_{e_k}}=\left(A^\bullet_{GL_2}((\mb{P}^1)^d)\otimes \mb{Q}\right)^{S_{e_1}\times \ldots \times S_{e_k}}.$$
%$A^\bullet_{PGL_2}(\mb{P}^1)^d \otimes \mb{Q}$ is generated by products of elementary symmetric polynomials in the groupings of $H'$ times $(u-v)^{2k}$. $\left(A^\bullet_{GL_2}((\mb{P}^1)^d)\otimes \mb{Q}\right)^{S_{e_1}\times \ldots \times S_{e_k}}$ is identically described but has the additional generator $u+v$ and $uv$. The result now follows.
\end{proof}

As we will now see, the cones over generators of $I^{GL_2}_\lambda \otimes \mb{Q}$ also generate $\wt{I_\lambda} \otimes \mb{Q}$. We will use a certain property about the classes of unordered strata to prove this, which as we will see is that $Z_\lambda$ contains a cycle whose class divides the class of the origin in $A^\bullet_{GL_2}(\mb{A}^{n+1})\otimes \mb{Q}$.

\begin{lem}\label{divideleray}
Given a partition $\lambda$ of $n$ and a set of generators $S$ of $I^{GL_2}_{\lambda}\otimes \mb{Q}$ of degree at most $n$, %including $[\{n\}]$, 
$\wt{I}_{\lambda}\otimes \mb{Q}$ is generated by
\begin{align*}
\{\alpha_0\mid \alpha\in S\},
\end{align*}
where $\alpha_0$ is the constant term of $\alpha\in A_{GL_2}^{\bullet}(\mb{P}^{n})\otimes \mb{Q}$, after writing $\alpha$ as a polynomial in $H,u,v$ that is degree at most $n$ in $H$ using the relation $G(H)=0$ (see \Cref{ringdescriptions}). 
\end{lem}

\begin{proof}
Let $\wt{I}'_{\lambda}\subset A^\bullet_{GL_2}(\mb{A}^{n+1})\otimes \mb{Q}$ be the ideal generated by $\{\alpha_0 \mid \alpha \in S\}$, so we want to show $\wt{I}'_\lambda=\wt{I}_\lambda\otimes \mb{Q}$. Consider the diagram of rational Chow rings (we omit $\otimes \mb{Q}$ for brevity)
\begin{center}
\begin{tikzcd}
%A_T^{\bullet}(Z_{\lambda}) \ar[d] & A_{T\times \Gm}^{\bullet}(\wt{Z_{\lambda}}\backslash\{0\}) \ar[d] \ar[swap, l, "\sim"] \ar[r, two heads] & A_{T}^{\bullet}(\wt{Z_{\lambda}}\backslash\{0\}) \ar[d] & A^{\bullet}_T(\wt{Z_{\lambda}}) \ar[l,two heads] \ar[d]\\
%& A_{T\times \Gm}^{\bullet}(\mb{A}^{n+1}) \ar[ld, two heads] \ar[d, two heads] \ar[dr, two heads] \ar[drr, two heads] & &\\
A_{GL_2}^{\bullet}(\mb{P}^n) \ar[d, two heads, "\pi_1"] & A_{GL_2\times \Gm}^{\bullet}(\mb{A}^{n+1}\backslash\{0\}) \ar[swap, l, "\sim"] \ar[r, two heads] \ar[d, two heads, "\pi_2"] & A_{GL_2}^{\bullet}(\mb{A}^{n+1}\backslash\{0\}) \ar[d, two heads, "\pi_3"] & A^{\bullet}_{GL_2}(\mb{A}^{n+1}) \ar[l,two heads]  \ar[d,two heads, "\pi_4"]\\
A_{GL_2}^{\bullet}(\mb{P}^n\backslash Z_{\lambda}) & A_{GL_2\times \Gm}^{\bullet}(\mb{A}^{n+1}\backslash\wt{Z_{\lambda}}) \ar[swap, l, "\sim"] \ar[r, two heads] & A_{GL_2}^{\bullet}(\mb{A}^{n+1}\backslash \wt{Z_{\lambda}}) & A^{\bullet}_{GL_2}(\mb{A}^{n+1}\backslash  \wt{Z_{\lambda}}) \ar[l,swap,"\sim"]
\end{tikzcd}
\end{center}
where $\Gm$ acts by scaling on $\mb{A}^{n+1}$. %Since the elements $\alpha\in S$ have degree at most $n$, they lift uniquely to $A^{\bullet}_{T\times \Gm}(\mb{A}^{n+1})$. Let $J_{\lambda}\subset A^{\bullet}_{T\times \Gm}(\mb{A}^{n+1})$ be the ideal generated by these lifts. 

%The image of $J_{\lambda}$ under $A^{\bullet}_{T\times \Gm}(\mb{A}^{n+1})\to A_T^{\bullet}(\mb{P}^n)$ is $I_{\lambda}$ and the image of $J_{\lambda}$ under $A^{\bullet}_{T\times \Gm}(\mb{A}^{n+1})\to A_T^{\bullet}(\mb{A}^{n+1})$ is $\wt{I}'_{\lambda}$. 
We know $I_{\lambda}\otimes \mb{Q}$ is the kernel of $\pi_1$, so it maps surjectively to the kernel of $\pi_3$ in $A_{GL_2}^{\bullet}(\mb{A}^{n+1}\backslash\{0\})$. Each generator $\alpha \in S$ maps to the image of $\alpha_0$ in $A^{\bullet}_{GL_2}(\mb{A}^{n+1}\backslash\{0\})\otimes \mb{Q}$. 
%By commutativity, the image of $\wt{I}'_{\lambda}$ in $A_{T}^{\bullet}(\mb{A}^{n+1}\backslash\{0\})$ is the kernel of $\pi_3$.
Since the kernel of $A_{GL_2}^{\bullet}(\mb{A}^{n+1})\otimes \mb{Q}\to A_{GL_2}^{\bullet}(\mb{A}^{n+1}\backslash\{0\})\otimes \mb{Q}$ is generated by $\prod_{i=0}^{n}{(iu+(n-i)v)}$, we have $\wt{I}'_{\lambda}+\langle \prod_{i=0}^{n}{(iu+(n-i)v)}\rangle = \wt{I}_{\lambda} \otimes \mb{Q}$. To finish, it suffices to see $\prod_{i=0}^{n}{(iu+(n-i)v)}\in \wt{I}'_{\lambda}$.

\begin{comment}
We see that the image of $I_{\lambda}$ in 
$$A_{T}^{\bullet}(\mb{A}^{n+1}\backslash\{0\})\otimes\mb{Q}\cong \mb{Q}[u,v]/(\prod_{i=0}^{n}{(iu+(n-i)v)})$$
is the image of $\wt{I}'_{\lambda}$ under the restriction 
$$A_{T}^{\bullet}(\mb{A}^{n+1})\otimes\mb{Q}\to  A_{T}^{\bullet}(\mb{A}^{n+1}\backslash\{0\})\otimes\mb{Q}.$$
To finish, it suffices to see that $\prod_{i=0}^{n}{(iu+(n-i)v)}\in \wt{I}_{\lambda}$. 
\end{comment}
As $Z_{{\{n\}}}$ is a cycle in $Z_\lambda$, $[\{n\}]$ can be expressed as an $A^{\bullet}_{GL_2}(\mb{P}^n)\otimes \mb{Q}$-linear combination of the elements of $S$, and taking the constant terms yields
$$[\{n\}]_0=n\prod_{i=1}^{n-1}{(iu+(n-i)v)}\in \wt{I}_{\lambda}'$$ 
by \Cref{diagonal} and \Cref{pushforward}, which divides $\prod_{i=0}^{n}{(iu+(n-i)v)}$.
\end{proof}
\begin{proof}[Proof of \Cref{GL2andaffine}]
Apply \Cref{PGLtoGL} to \Cref{pushforwardlemma,2generatorslemma} to get the statements on $I_{\lambda}^{GL_2}\otimes\mb{Q}$. Then, apply \Cref{divideleray} to get the statements on $\wt{I}_{\lambda}$. 
\end{proof}

\appendix
\section{Multiplicative relations between symmetrized strata}
\label{multiplicative}
%Continuing the study in Feh\'er and Rim\'anyi, we study the Thom polynomials of symmetrized orbits. Given a $K^*$-invariant subvariety $X$ of a vector space $V$, which is further invariant under an algebraic group $G$, the equivariant Chow classes $[X] \in A^\bullet_G(V) \equiv A^\bullet_G(\pt)$ and $[\wt{X}] \in A^\bullet_G(\mb{P}(V))\equiv A^\bullet_G(\pt)[H]$ determine each other as follows. To obtain $[X]$ from $[\wt{X}]$ we take the constant term in $A^\bullet_G(\pt)$. Conversely, if $u_1,\ldots,u_k$ are the Chern roots of $G$, then we obtain $[\wt{X}]$ from $[X]$ by replacing $u_i$ with $u_i+\frac{1}{k}H$ for all $i$. The class $[X]$ is called the \emph{Thom polynomial}.

In this section, we investigate certain multiplicative relations between the classes $[\wt{Z_\lambda}]\in A^\bullet_{GL_2}(\text{\rm Sym}^n K^2)$. These are equivalent to certain relations between the degree 0 terms of the expressions for $[\lambda]\in A^\bullet_{GL_2}(\mb{P}^n)$ by \Cref{reconstructionsection}. For this, it suffices to restrict ourselves to the $\mb{Q}$-basis given by the $[a,b,1^c]$-classes from \Cref{basis}.

\begin{defn}
Denote by $[a,b,1^c]_0\in \mb{Z}[u,v]^{S_2}$ be the term of $[a,b,1^c]\in H_{GL_2}^{\bullet}(\mb{P}^n)$ that is degree zero in $H$.
\end{defn}

We show how to write $(u+v)[a,b,1^c]_0$ and $uv[a,b,1^c]_0$ as a $\mb{Q}$-linear combination of strata. A few of these multiplicative relations have been explicitly written down \cite[Remark 3.9]{FNR06} and shown to exist abstractly \cite[Theorems 4.3 and 4.10]{FNR06} using the degeneration of a spectral sequence of a filtered CW-complex. We give a combinatorial method to do this in general in \Cref{uplusv,utimesv}. 

\begin{thm}
\label{uplusv}
For $c\ge 1$ and $a+b+c=n$,
\begin{align*}
n(u+v)[a,b,1^c]_0=(c+a-b)&[a+1,b,1^{c-1}]_0\\
+(b+c-a)&[a,b+1,1^{c-1}]_0\\
+(a+b-c)&[a+b,1,1^{c-1}]_0.
\end{align*}
\end{thm}
\begin{proof}
We will prove \Cref{uplusv} by pulling back to $(\mb{P}^1)^n$. By \Cref{reconstruction}, we want to show
\begin{align*}
(2H+nu+nv)[a,b,1^c]=(c+a-b)&[a+1,b,1^{c-1}]\\
+(b+c-a)&[a,b+1,1^{c-1}]\\
+(a+b-c)&[a+b,1,1^{c-1}].
\end{align*}
Let $A=\{1,\ldots,a\}$, $B=\{b+1,\ldots,a+b\}$. By the projection formula, the right hand side is
$$\Phi_*(\Delta_{\{A,B\}}\cap \Phi^*(2H+nu+nv)).$$ 
The pullback of $2H+nu+nv$ along $\Phi$ is
\begin{align*}
(H_1+H_2+u+v)+(H_2+H_3+u+v)+\ldots+(H_n+H_1+u+v)\\
=\Delta_{1,2}+\Delta_{2,3}+\ldots+\Delta_{n,1}
\end{align*}
by \Cref{psidelta}.
In this way, we now only have to intersect strata using \Cref{trivialfacts} and the square relation as in \Cref{P14}. 

There are 6 cases: $1 \le i \le a-1$, $i=a$, $a+1 \le i \le a+b-1$, $i=a+b$, $a+b+1 \le i \le n-1$, and $i=n$. We will deal with each of these cases in the same way outlined above. 

To calculate $\Phi_*(\Delta_{i,i+1}\Delta_{\{A,B\}})$ for $1 \le i \le a-1$, we use the square relation to replace $\Delta_{i,i+1}$ with $\Delta_{i,n}-\Delta_{n,a+1}+\Delta_{a+1,i+1}$. Using \Cref{trivialfacts}, each of the products is itself a strata, and the pushforward is
$$[a+1,b,1^{c-1}]-[a,b+1,1^{c-1}]+[a+b,1,1^{c-1}].$$

\begin{center}
\begin{tikzpicture}[scale=0.25]
\draw[fill=black] (0+0.4,0-0.4) circle (4pt);
\draw[fill=black] (1+0.1,1-0.1) circle (4pt);
\draw[fill=black] (2,2) circle (4pt);
\draw[fill=black] (3+0.1,3-0.1) circle (4pt);
\draw[fill=black] (4+0.4,4-0.4) circle (4pt);
%\draw [red] plot [smooth cycle] coordinates {(0,0) (1,1) (3,1) (1,0) (2,-1)};
\draw [black,dotted] plot [smooth cycle] coordinates {(-1+1,-1) (-1+1,2) (5,4.2) (2+1,1)};

\draw[fill=black] (8-0.4,4-0.4) circle (4pt);
\draw[fill=black] (9-0.1,3-0.1) circle (4pt);
\draw[fill=black] (10,2) circle (4pt);
\draw[fill=black] (11-0.1,1-0.1) circle (4pt);
\draw[fill=black] (12-0.4,0-0.4) circle (4pt);
\draw [black,dotted] plot [smooth cycle] coordinates {({12-(-1+1)},-1) ({12-(-1+1)},2) ({12-(5)},4.2) ({12-(2+1)},1)};

\draw[fill=black] (4,-3+0.4) circle (4pt);
\draw[fill=black] (5,-3+0.1) circle (4pt);
\draw[fill=black] (6,-3) circle (4pt);
\draw[fill=black] (7,-3+0.1) circle (4pt);
\draw[fill=black] (8,-3+0.4) circle (4pt);
%\draw [black,dotted] plot [smooth cycle] coordinates {(3,-3.5) (3,-2) (9,-2) (9,-3.5)};

\draw (-1.2,2) node {A};
\draw (13,2) node {B};

\draw (2,2) node[anchor=south east] {i};
\draw (3+0.1,3-0.1) node[anchor=south] {i+1};
\draw (8-0.4,4-0.4) node[anchor=south] {a+1};
\draw (8,-3+0.4) node[anchor=south west] {n};

\draw (2,2) -- (3+0.1,3-0.1) -- (8-0.4,4-0.4) -- (8,-3+0.4) -- cycle;
\end{tikzpicture}
\end{center}

For $i=a$, \Cref{trivialfacts} implies $\Delta_{a,a+1}\Delta_{A,B}=\Delta_{\{A \sqcup B\}}$, which pushes forward to $$[a+b,1,1^{c-1}].$$

Similarly to before, for $a+1 \le i \le a+b-1$, the pushforward is
$$[a,b+1,1^{c-1}]-[a+1,b,1^{c-1}]+[a+b,1,1^{c-1}].$$

For $i=a+b$, \Cref{trivialfacts} implies $\Delta_{a+b,a+b+1}\Delta_{\{A,B\}}=\Delta_{\{A,B\sqcup\{a+b+1\}\}}$, which pushes forward to $$[a,b+1,1^{c-1}].$$

For $a+b+1 \le i \le n-1$, replace $\Delta_{i,i+1}$ with $\Delta_{i,a}-\Delta_{a,a+1}+\Delta_{a+1,i+1}$, and similarly to before we get the pushforward is $$[a+1,b,1^{c-1}]-[a+b,1,1^{c-1}]+[a,b+1,1^{c-1}].$$

Finally, for $i=n$, using \Cref{trivialfacts}, $\Delta_{n,1}\Delta_{\{A,B\}}=\Delta_{\{A\sqcup\{n\},B\}}$, so this will pushforward to $$[a+1,b,1^{c-1}].$$

Combining these yields the desired result.
\end{proof}

\begin{rmk}
Given a partition $\lambda$ of $n$ with at least three nontrivial parts, the argument of \Cref{uplusv} is a combinatorial algorithm that can non-canonically express $n(u+v)[\lambda]$ in terms of other classes $[\lambda']$ with one fewer part. The number of square relations can be drastically reduced in practice by an appropriate choice of the partition pushing forward to $[a_1,\ldots,a_d]$.%The answer one gets at the end is not in terms of the basis of $[a,b;c]$ given in \Cref{basis}, so it is not canonical. 

%When computing examples, it is often helpful to replace use the relation
%$$[\{\{1,2\}\}]+[\{\{2,3\}\}]+\ldots+[\{\{n,1\}\}]=[\{\{\sigma(1),\sigma(2)\}\}]+\ldots+[\{\{\sigma(n),\sigma(1)\}\}]$$
%for any permutation $\sigma\in S_n$ to reduce the number of times one needs to apply the square relation. 
\end{rmk}

\begin{thm}
\label{utimesv}
For $c \ge 2$, and $a+b+c=n$
\begin{align*}
n^2uv[a,b,1^c]_0=(2ab+ac+bc+c(c-1))&[a+1,b+1,1^{c-2}]_0\\
+(-ab-bc)&[a+2,b,1^{c-2}]_0\\
+(-ab-ac)&[a,b+2,1^{c-2}]_0\\
+(-ac-bc-c(c-1))&[a+b+1,1,1^{c-2}]_0\\
+(ac+bc)&[a+b,2,1^{c-2}]_0.
\end{align*}
\end{thm}
\begin{comment}
n^2uv[a,b,1^c]_0=(2ab+ac+bc+c(c-1))&[a+1,b+1,1^{c-2}]_0\\
+(-ab-ac)&[a+2,b,1^{c-2}]_0\\
+(-ab-bc)&[a,b+2,1^{c-2}]_0\\
+(-ac-bc)&[a+b+1,1,1^{c-2}]_0\\
+(ac+bc-c(c-1))&[a+b,2,1^{c-2}]_0.
\end{comment}
\begin{proof}
As in the previous theorem letting $A=\{1,\ldots,a\}$, $B=\{a+1,\ldots,a+b\}$ the statement is equivalent to
\begin{align*}
\Phi_*((\sum H_i + nu)(\sum H_i+nv)\Delta_{\{A,B\}})\\
=(2ab+ac+bc+c(c-1))&[a+1,b+1,1^{c-2}]\\
+(-ab-bc)&[a+2,b,1^{c-2}]\\
+(-ab-ac)&[a,b+2,1^{c-2}]\\
+(-ac-bc-c(c-1))&[a+b+1,1,1^{c-2}]\\
+(ac+bc)&[a+b,2,1^{c-2}].
\end{align*}
We have $(H_i+u)(H_i+v)=0$, so
\begin{align*}
(\sum H_i + nu)(\sum H_i+nv)=&\sum_{1 \le i<j \le n} (H_i+u)(H_j+v)+(H_j+u)(H_i+v)\\
=&\sum_{1 \le i<j \le n} -(H_i-H_j)^2\\
=&\sum_{1 \le i<j \le n} -(\Delta_{i,k_{i,j}}-\Delta_{j,k_{i,j}})^2
\end{align*}
where $k_{i,j} \in [n]\setminus \{i,j\}$ is arbitrary. There are $6$ cases depending on which of $A,B,[n]\setminus \{A,B\}$ each of $i,j$ lie in, and for each of these cases an appropriate choice of $k_{i,j}$ can be made so that the strata combine via \Cref{trivialfacts} as in the proof of \Cref{uplusv} and push forward to $[a',b',1^{c-2}]$-classes.
\end{proof}

\begin{rmk}
Similarly to \Cref{uplusv}, the argument of \Cref{utimesv} is a combinatorial algorithm that can express $n^2uv[\lambda]$ in terms of other classes $[\lambda']$ with two fewer parts for any partition $\lambda$ of $n$ with at least four parts. 
\end{rmk}

\begin{comment}
For each of $[\{i,i+1\}]$ with $1 \le i \le a-1$, we can use the square relation to replace it with $[\{i,c\}]-[\{c,a+1\}]+[\{a+1,i+1\}]$, and for each $[\{i,i+1\}]$ with $b \le i \le b+1$, we can use the square relation to replace it with $[\{i,c\}]-[\{c,a\}]+[\{a,i+1\}]$. Therefore the intersection before applying $\Phi_*$ is
\begin{align*}
(a-1)([\{1,\ldots,a,c\},\{a+1,\ldots,a+b\}]-[\{1,\ldots,a\},\{a+1,\ldots,a+b,c\}]+[\{1,\ldots,a+b\}])\\
+[\{1,\ldots,a+b\}]\\
+(b-1)(-[\{1,\ldots,a,c\},\{a+1,\ldots,a+b\}]+[\{1,\ldots,a\},\{a+1,\ldots,a+b,c\}]+[\{1,\ldots,a+b\}])\\
+[\{1,\ldots,a\},\{a+1,\ldots,a+b+1\}]\\
+\sum_{i=a+b+1}^{n-1} [\{1,\ldots,a\},\{a+1,\ldots,b\},\{i,i+1\}]\\
+[\{1,\ldots,a,c\},\{a+1,\ldots,a+b\}]
\end{align*}
\end{comment}

\begin{comment}
For the first one, we note that $$2\sum H_i+u+v = \sum [\sigma(i)\sigma(j)]$$ for any permutation $\sigma$, where we let $\sigma(n+1)=\sigma(1)$.

For the second one, we note that $$(H_i+u)(H_j+v)+(H_j+u)(H_i+v)=-([ik]-[jk])^2$$ for any choice of $k$.
\end{comment}

\bibliographystyle{plain}
\bibliography{references.bib}

\begin{thebibliography}{10}

\bibitem{Anderson}
Dave Anderson.
\newblock Introduction to equivariant cohomology in algebraic geometry.
\newblock In {\em Contributions to algebraic geometry}, EMS Ser. Congr. Rep.,
  pages 71--92. Eur. Math. Soc., Z\"urich, 2012.

\bibitem{Vistoli}
Alessandro Arsie and Angelo Vistoli.
\newblock Stacks of cyclic covers of projective spaces.
\newblock {\em Compos. Math.}, 140(3):647--666, 2004.

\bibitem{Shamil}
Shamil Asgarli and Giovanni Inchiostro.
\newblock The picard group of the moduli of smooth complete intersections of
  two quadrics.
\newblock {\em preprint}, 2017.
\newblock arXiv:1710.10113.

\bibitem{Brion}
Michel Brion.
\newblock Lectures on the geometry of flag varieties.
\newblock In {\em Topics in cohomological studies of algebraic varieties},
  Trends Math., pages 33--85. Birkh\"auser, Basel, 2005.

\bibitem{Starr}
T.~D. Browning and D.~R. Heath-Brown.
\newblock The density of rational points on non-singular hypersurfaces. {II}.
\newblock {\em Proc. London Math. Soc. (3)}, 93(2):273--303, 2006.
\newblock With an appendix by J. M. Starr.

\bibitem{Laza}
Charles Cadman and Radu Laza.
\newblock Counting the hyperplane sections with fixed invariants of a plane
  quintic---three approaches to a classical enumerative problem.
\newblock {\em Adv. Geom.}, 8(4):531--549, 2008.

\bibitem{Lorenzo}
Andrea Di~Lorenzo.
\newblock The chow ring of the stack of hyperelliptic curves of odd genus.
\newblock {\em preprint}, 2018.
\newblock arXiv:1802.04519.

\bibitem{Fulghesu}
Dan Edidin and Damiano Fulghesu.
\newblock The integral {C}how ring of the stack of hyperelliptic curves of even
  genus.
\newblock {\em Math. Res. Lett.}, 16(1):27--40, 2009.

\bibitem{EG98}
Dan Edidin and William Graham.
\newblock Equivariant intersection theory.
\newblock {\em Invent. Math.}, 131(3):595--634, 1998.

\bibitem{EG98b}
Dan Edidin and William Graham.
\newblock Localization in equivariant intersection theory and the {B}ott
  residue formula.
\newblock {\em Amer. J. Math.}, 120(3):619--636, 1998.

\bibitem{3264}
David Eisenbud and Joe Harris.
\newblock {\em 3264 and all that---a second course in algebraic geometry}.
\newblock Cambridge University Press, Cambridge, 2016.

\bibitem{FNR05}
L.~M. Feh\'er, A.~N\'emethi, and R.~Rim\'anyi.
\newblock Degeneracy of 2-forms and 3-forms.
\newblock {\em Canad. Math. Bull.}, 48(4):547--560, 2005.

\bibitem{FNR06}
L.~M. Feh\'er, A.~N\'emethi, and R.~Rim\'anyi.
\newblock Coincident root loci of binary forms.
\newblock {\em Michigan Math. J.}, 54(2):375--392, 2006.

\bibitem{quiver}
Hans Franzen and Markus Reineke.
\newblock Cohomology rings of moduli of point configurations on the projective
  line.
\newblock {\em Proc. Amer. Math. Soc.}, 146(6):2327--2341, 2018.

\bibitem{FV11}
Damiano Fulghesu and Filippo Viviani.
\newblock The {C}how ring of the stack of cyclic covers of the projective line.
\newblock {\em Ann. Inst. Fourier (Grenoble)}, 61(6):2249--2275 (2012), 2011.

\bibitem{FP}
W.~Fulton and R.~Pandharipande.
\newblock Notes on stable maps and quantum cohomology.
\newblock In {\em Algebraic geometry---{S}anta {C}ruz 1995}, volume~62 of {\em
  Proc. Sympos. Pure Math.}, pages 45--96. Amer. Math. Soc., Providence, RI,
  1997.

\bibitem{Fulton}
William Fulton.
\newblock {\em Intersection theory}, volume~2 of {\em Ergebnisse der Mathematik
  und ihrer Grenzgebiete. 3. Folge. A Series of Modern Surveys in Mathematics
  [Results in Mathematics and Related Areas. 3rd Series. A Series of Modern
  Surveys in Mathematics]}.
\newblock Springer-Verlag, Berlin, second edition, 1998.

\bibitem{Hassett}
Brendan Hassett.
\newblock Moduli spaces of weighted pointed stable curves.
\newblock {\em Adv. Math.}, 173(2):316--352, 2003.

\bibitem{HK00}
Yi~Hu and Sean Keel.
\newblock Mori dream spaces and {GIT}.
\newblock {\em Michigan Math. J.}, 48:331--348, 2000.
\newblock Dedicated to William Fulton on the occasion of his 60th birthday.

\bibitem{Kapranov}
M.~M. Kapranov.
\newblock Chow quotients of {G}rassmannians. {I}.
\newblock In {\em I. {M}. {G}el'fand {S}eminar}, volume~16 of {\em Adv. Soviet
  Math.}, pages 29--110. Amer. Math. Soc., Providence, RI, 1993.

\bibitem{Katz}
Nets~Hawk Katz.
\newblock The flecnode polynomial: a central object in incidence geometry.
\newblock In {\em Proceedings of the {I}nternational {C}ongress of
  {M}athematicians---{S}eoul 2014. {V}ol. {III}}, pages 303--314. Kyung Moon
  Sa, Seoul, 2014.

\bibitem{Keel}
Sean Keel.
\newblock Intersection theory of moduli space of stable {$n$}-pointed curves of
  genus zero.
\newblock {\em Trans. Amer. Math. Soc.}, 330(2):545--574, 1992.

\bibitem{FML}
Mitchell Lee, Anand Patel, Hunter Spink, and Dennis Tseng.
\newblock Orbits in {$(\mathbb{P}^r)^n$} and equivariant quantum cohomology.
\newblock {\em preprint}, 2018.
\newblock arXiv:1805.08181.

\bibitem{LP11}
Xiaobo Liu and Rahul Pandharipande.
\newblock New topological recursion relations.
\newblock {\em J. Algebraic Geom.}, 20(3):479--494, 2011.

\bibitem{classical}
Luis~Alberto Molina~Rojas and Angelo Vistoli.
\newblock On the {C}how rings of classifying spaces for classical groups.
\newblock {\em Rend. Sem. Mat. Univ. Padova}, 116:271--298, 2006.

\bibitem{P02}
R.~Pandharipande.
\newblock Three questions in {G}romov-{W}itten theory.
\newblock In {\em Proceedings of the {I}nternational {C}ongress of
  {M}athematicians, {V}ol. {II} ({B}eijing, 2002)}, pages 503--512. Higher Ed.
  Press, Beijing, 2002.

\bibitem{Pandharipande}
Rahul Pandharipande.
\newblock Equivariant {C}how rings of {${\rm O}(k),\ {\rm SO}(2k+1)$}, and
  {${\rm SO}(4)$}.
\newblock {\em J. Reine Angew. Math.}, 496:131--148, 1998.

\bibitem{Romagny}
Matthieu Romagny.
\newblock Group actions on stacks and applications.
\newblock {\em Michigan Math. J.}, 53(1):209--236, 2005.

\bibitem{Segre}
B.~Segre.
\newblock The maximum number of lines lying on a quartic surface.
\newblock {\em Quart. J. Math., Oxford Ser.}, 14:86--96, 1943.

\bibitem{Vainsencher}
Israel Vainsencher.
\newblock Counting divisors with prescribed singularities.
\newblock {\em Trans. Amer. Math. Soc.}, 267(2):399--422, 1981.

\end{thebibliography}
\end{document}